\documentclass[11pt,letterpaper]{amsart}
\pdfoutput=1

\usepackage{eucal}
\usepackage{graphicx}
\usepackage{color}
\usepackage{amsopn,amssymb}

\definecolor{verydarkblue}{rgb}{0,0,0.4}
\usepackage{hyperref}
\hypersetup{pdfauthor={David Dumas and Michael Wolf},pdftitle={Polynomial cubic differentials and convex polygons in the projective plane},colorlinks=true,linkcolor=verydarkblue,citecolor=verydarkblue}

\usepackage{caption}
\makeatletter%
\def\@commafont{\check@mathfonts
    \fontsize\sf@size\z@\selectfont}
\DeclareRobustCommand{\cb}[1]{{%
\setbox\z@\hbox{#1}%
\ifdim\dp\z@<.1\ht\z@\ooalign{\unhbox\z@\crcr\hidewidth\lower.3ex\hbox{\@commafont,}\hidewidth}%
\else\ooalign{\unhbox\z@\crcr\hidewidth\raise.5ex\hbox{\@commafont`}\hidewidth}%
\fi}}%
\makeatother

\newcommand{\boldpoint}[1]{\smallskip\par\noindent\textbf{#1}}

\newcommand{\bdry}{\partial}
\newcommand{\transverse}{\pitchfork}

\newcommand{\into}{\hookrightarrow}

\newcommand{\id}{\mathrm{Id}}
\newcommand{\Z}{\mathbb{Z}}
\newcommand{\R}{\mathbb{R}}
\newcommand{\C}{\mathbb{C}}
\newcommand{\N}{\mathbb{N}}

\DeclareMathOperator{\Sym}{Sym}
\DeclareMathOperator{\Vect}{Vect}
\DeclareMathOperator{\Hess}{Hess}
\DeclareMathOperator{\Ad}{Ad}
\DeclareMathOperator{\im}{Im}
\DeclareMathOperator{\re}{Re}
\renewcommand{\Re}{\re}
\renewcommand{\Im}{\im}

\newcommand{\CP}{\mathbb{CP}}
\renewcommand{\H}{\mathbb{H}}
\renewcommand{\P}{\mathbb{P}}

\newcommand{\eps}{\varepsilon}
\newcommand{\gl}{\mathfrak{gl}}

\newcommand{\CM}{\mathcal{CM}}
\newcommand{\RP}{\mathbb{RP}}

\newcommand{\h}{\mathfrak{h}}
\newcommand{\D}{\mathcal{D}}
\newcommand{\suchthat}{\:|\:}
\newcommand{\bla}{Blaschke }
\newcommand{\titeica}{\cb{T}i\cb{t}eica }

\newcommand{\tec}{Teichm\"uller }

\newcommand{\sT}{\mathcal{T}}

\newcommand{\rot}{\varrho}
\newcommand{\bdlcubic}{\mathcal{C}}

\newcommand{\Cubic}{\mathcal{C}}
\newcommand{\NormCubic}{\mathcal{TC}}
\newcommand{\ModCubic}{\mathcal{MC}}

\newcommand{\Poly}{\mathcal{P}}
\newcommand{\NormPoly}{\mathcal{TP}}
\newcommand{\ModPoly}{\mathcal{MP}}

\newcommand{\cone}{\mathcal{K}}

\newcommand{\boldalpha}{\boldsymbol\alpha}

\newcommand{\unity}{\boldsymbol\mu}

\renewcommand{\leq}{\leqslant}
\renewcommand{\geq}{\geqslant}

\newcommand{\noproof}{\hfill\qedsymbol}

\DeclareMathOperator{\GL}{\mathrm{GL}}
\DeclareMathOperator{\SL}{\mathrm{SL}}

\DeclareMathOperator{\SO}{\mathrm{SO}}

\DeclareMathOperator{\Aut}{Aut}

\renewcommand{\hat}{\widehat}

\newcommand{\param}{{\mathchoice{\mkern1mu\mbox{\raise2.2pt\hbox{$\centerdot$}}\mkern1mu}{\mkern1mu\mbox{\raise2.2pt\hbox{$\centerdot$}}\mkern1mu}{\mkern1.5mu\centerdot\mkern1.5mu}{\mkern1.5mu\centerdot\mkern1.5mu}}}

\theoremstyle{plain}
\newtheorem{thm}{Theorem}[section]
\newtheorem{cor}[thm]{Corollary}
\newtheorem{conj}[thm]{Conjecture}
\newtheorem{lem}[thm]{Lemma}
\newtheorem{prop}[thm]{Proposition}

\newtheorem{bigthm}{Theorem}

\theoremstyle{definition}
\newtheorem*{question}{Question}

\theoremstyle{definition}
\newtheorem*{remark}{Remark}
\newtheorem*{remarksenv}{Remarks}

\newenvironment{rmenumerate}{\begin{enumerate}}{\end{enumerate}}

\newenvironment{prooflist}{\begin{proof}\mbox{}\begin{list}{(\roman{enumi})}{\usecounter{enumi}\setlength{\labelwidth}{30pt}\setlength{\itemindent}{0pt}\setlength{\leftmargin}{20pt}\setlength{\itemsep}{5pt}}}{\end{list}\end{proof}}

\begin{document}

\title[Cubic differentials and convex polygons]{Polynomial cubic differentials and\\ convex polygons in the
  projective plane}

\author[D. Dumas]{David Dumas}
\author[M. Wolf]{Michael Wolf}

\date{September 24, 2015.  (v2: January 4, 2015.  v1: July 30, 2014.)}

\maketitle

\begin{abstract}
We construct and study a natural homeomorphism between the moduli
space of polynomial cubic differentials of degree $d$ on the complex
plane and the space of projective equivalence classes of oriented
convex polygons with $d+3$ vertices.  This map arises from the
construction of a complete hyperbolic affine sphere with prescribed
Pick differential, and can be seen as an analogue of the
Labourie-Loftin parameterization of convex $\RP^2$ structures on a
compact surface by the bundle of holomorphic cubic differentials over
\tec space.
\end{abstract}

\section{Introduction}

\subsection*{Motivating problem}
Labourie \cite{labourie:flat-projective} and Loftin
\cite{loftin:parameterization} have independently shown that the
moduli space of convex $\RP^2$ structures on a compact surface $S$ of
genus $g \geq 2$ can be identified with the vector bundle
$\bdlcubic(S)$ of holomorphic cubic differentials over the \tec space
$\sT(S)$.

By definition, a convex $\RP^2$ structure on $S$ is the quotient of a
properly convex open set in $\RP^2$ by a free and cocompact action of
a group of projective transformations.  Identifying this convex domain
with the projectivization of a convex cone in $\R^3$, one can consider
smooth convex surfaces in $\R^3$ that are asymptotic to the boundary
of the cone.  A fundamental theorem of Cheng and Yau \cite{cheng-yau} (proving a
conjecture of Calabi \cite{calabi1972}) shows that there is a unique such surface which
is a \emph{complete hyperbolic affine sphere}.  Classical
affine-differential constructions equip this affine sphere with a
$\pi_1S$-invariant Riemann surface structure and holomorphic cubic
differential (the \emph{Pick differential}), and the respective
$\pi_1S$-quotients of these give a point in $\bdlcubic(S)$.  The
surjectivity of this map from $\RP^2$ structures to $\bdlcubic(S)$ is
established using a method of C.~P.~Wang, wherein the reconstruction
of an affine sphere from the cubic differential data is reduced to the
solution of a quasilinear PDE.

Since Wang's technique and the Cheng-Yau theorem apply to any properly
convex domain, it would broadly generalize the Labourie-Loftin
parameterization if one could characterize those pairs of
simply-connected Riemann surfaces and holomorphic cubic differentials
that arise from properly convex open sets in $\RP^2$.  That is the basic
motivating question we consider in this paper.

As stated, this question is probably too broad to have a
satisfactorily complete answer.  However, the question can be
specialized in many ways by asking to characterize the cubic
differentials corresponding to a special class of convex sets, or the
convex sets arising from a special class of cubic differentials.  Of
course the Labourie-Loftin parameterization can be described this way,
where one requires the convex domain to carry a cocompact group action
of a given topological type.  Benoist and Hulin have also considered
aspects of this question, first for convex domains covering a
noncompact surface of finite area \cite{benoist-hulin:finite-volume},
and later showing that domains with Gromov-hyperbolic Hilbert metrics
correspond to the Banach space of $L^\infty$ cubic differentials on
the hyperbolic plane \cite{benoist-hulin:hyperbolic}.

\subsection*{Main theorem} We consider another specialization of the
motivating question that is essentially orthogonal to those of
Labourie-Loftin and Benoist-Hulin, namely, identifying Riemann surfaces and
holomorphic cubic differentials that correspond to \emph{convex
  polygons} in $\RP^2$.  Our main result is that the associated affine
spheres give parabolic Riemann surfaces (biholomorphic to $\C$) and
cubic differentials that are complex polynomials, with the degree
of the polynomial determined by the number of vertices of the polygon.

For example, the affine sphere over the regular pentagon corresponds
to the complex plane with the cubic differential $z^2 dz^3$.  The
$5$-fold rotational symmetry of the pentagon corresponds to the
invariance of $z^2 dz^3$ under the automorphism $z \mapsto
e^{2 \pi i/5} z$.

In fact, as in the compact surface case, the affine sphere
construction gives a homeomorphic identification between two moduli
spaces.  Consider the space of cubic differentials $p(z) dz^3$ on the
complex plane where $p(z)$ is a polynomial of degree $d$, and let
$\ModCubic_d$ denote the quotient of this set by the action of the
holomorphic automorphism group $\Aut(\C) = \{ z \mapsto az + b \}$.
Let $\ModPoly_{n}$ denote the space of projective equivalence classes
of convex polygons in $\RP^2$ with $n$ vertices.  We show:

\begin{bigthm}
\label{introthm:main}
The affine sphere construction determines a homeomorphism
$$\boldalpha: \ModCubic_d \to \ModPoly_{d+3}.$$
That is, each polynomial cubic differential $C$ is the Pick differential
of a complete hyperbolic affine sphere $S \subset \R^3$, uniquely
determined up to the action of $\SL_3\R$ and asymptotic to the cone
over a convex polygon $P$.  The map $\boldalpha$ is defined by
$$\boldalpha([C]) = [P],$$
where $[C]$ denotes the $\Aut(\C)$-equivalence class of $C$ and $[P]$ the
$\SL_3\R$-equivalence class of $P$.
\end{bigthm}

The spaces $\ModPoly_n$ and $\ModCubic_d$ and their properties are
discussed in more detail in Sections \ref{sec:polygons} and
\ref{sec:cubic}, respectively.  Since both of these are smooth
orbifolds, it is natural to ask about the smoothness of the map
$\boldalpha$ itself.  This issue and related conjectures are discussed
in Section \ref{sec:conjectures}.

Our proof of the main theorem is direct:  We construct mutually
inverse maps $\ModCubic_d \to \ModPoly_{d+3}$ and $\ModPoly_{d+3} \to
\ModCubic_d$ and show that each is continuous.  (Most of the time we
actually work with the lifted map between manifold covers of these orbifold
moduli spaces, but we elide this distinction in the introduction.)

\subsection*{Existence of polynomial affine spheres}
The construction
of the map $\boldalpha: \ModCubic_d \to \ModPoly_{d+3}$ first requires
an existence theorem, i.e.~for any polynomial cubic differential on
$\C$ there exists a corresponding complete hyperbolic affine sphere.  As in the
work of Labourie and Loftin, through the technique of Wang this
becomes a problem about finding a conformal metric---the
\emph{Blaschke metric} of the affine sphere---that satisfies a
quasilinear PDE involving the norm of the cubic differential.  The
cubic differential and the Blaschke metric together determine a flat
$\mathfrak{sl}_3\R$-valued connection form, the trivialization of
which gives the affine sphere (and a framing thereof).

The technique we apply to solve for the Blaschke metric (the method of super-
and sub-solutions) naturally gives a result for a somewhat more
general class of equations.  In Section \ref{sec:vortex} we show:

\begin{bigthm}
\label{introthm:vortex}
Let $\phi = \phi(z) dz^k$ be a holomorphic differential on $\C$ with
$\phi(z)$ a polynomial.  Then there exists a unique complete and
nonpositively curved conformal metric $\sigma$ on $\C$ whose Gaussian
curvature function $K$ satisfies
$$K = (-1 + |\phi|_\sigma^2).$$
Here
$|\phi|_\sigma$ denotes the pointwise norm of $\phi$ with respect
to the Hermitian metric $\sigma^{-k}$ on the bundle of order-$k$
holomorphic differentials.
\end{bigthm}

The existence part of this result appears in Theorem
\ref{thm:existence}, and the uniqueness in Theorem
\ref{thm:uniqueness}; the details are modeled on those found in
\cite{wan:thesis}, \cite{au-wan94}, \cite{han} and \cite{httw}.  We
call the differential equation considered in this theorem the
\emph{coupled vortex equation}; the connection between this equation
and the classical vortex equation from gauge theory is described at
the beginning of Section \ref{sec:vortex}.

Returning to the construction of $\boldalpha$, the case $k=3$ of
Theorem \ref{introthm:vortex} implies that there exists a complete
affine sphere with any given polynomial Pick differential. The next
step is to show that these affine spheres are asymptotic to cones over
convex polygons with the right number of vertices.

\subsection*{\titeica asymptotics}
The simplest examples of complete affine spheres with nonzero Pick
form are the \emph{\titeica surfaces}.  These can be characterized as
the affine spheres conformal to $\C$ and having constant nonzero Pick
differential; all such surfaces are projectively equivalent, and each
is asymptotic to the cone over a triangle.

A polynomial cubic differential on $\C$ can be identified, after
removing a compact set containing the zeros, with the result of gluing
of a finite collection of half-planes, each equipped with a constant
differential $c dz^3$.  Using this description, we show that the asymptotic
geometry of an affine sphere with polynomial Pick differential
(briefly, a \emph{polynomial affine sphere}) can be asymptotically modeled by gluing
together pieces of finitely many \titeica surfaces.  Corresponding
pieces of the asymptotic triangular cones glue to give the cone over a
convex polygon.

An important aspect of the picture that is missing from the sketch
above represents most of the work we do in Section
\ref{sec:polynomials-to-polygons}: In order to compare an affine
sphere to a \titeica surface, one must control not only the Pick
differential, but also the Blaschke metric.  Analyzing the structure
equations defining an affine sphere, this becomes a question of
comparing the \emph{Blaschke error}, meaning the difference between the
Blaschke metric of a polynomial affine sphere and that of a \titeica
surface, to the \emph{frame size} of the \titeica surface, meaning the
spectral radius of its affine frame in the adjoint representation.  If
the product of frame size and Blaschke error decays to zero in some region, we
find a unique \titeica surface to which the polynomial affine sphere
is asymptotic.

Using our estimates for solutions to the coupled vortex equation, we
find that the \bla error decays exponentially, and more precisely that 
$$ \text{Blaschke error} \: = \:O \left (\frac{e^{-2 \sqrt{3}\: r}}{\sqrt{r}} \right
),$$
where $r$ is the distance from the zeros of the polynomial (measured
in a natural coordinate system).  
In the same coordinates, the affine frame of the \titeica surface grows
exponentially, but the rate depends on direction, i.e.~
$$ \text{Frame size} \: \approx \: C e^{c(\theta) r}$$
along a ray of angle $\theta$, where $c(\theta)$ is an explicit
function satisfying
$$c(\theta) \leq 2 \sqrt{3}$$
with equality for exactly $2(d+3)$ directions, for a polynomial of
degree $d$.  Away from these \emph{unstable directions}, the error
decay is therefore more rapid than the frame growth and we have a unique asymptotic
\titeica surface, giving $2(d+3)$ \titeica surfaces in all.

\subsection*{Assembling a polygon}
The final step in the construction of $\boldalpha$ is to understand
how the limiting \titeica surface changes when we cross an unstable
direction.  Perhaps surprisingly, in this analysis the $1/\sqrt{r}$
factor in our bound for the error turns out to be crucial (and sharp).

By integrating the affine connection form over an arc of a circle of
radius $R$ joining two rays that lie on either side of an unstable
direction, we determine the element of $\SL_3\R$ relating the limiting
\titeica surfaces along these rays.  The interplay between the
quadratic approximation to $c(\theta)$ near its $2 \sqrt{3}$ maximum
and the \bla error estimate show that this integral is essentially a
Gaussian approximation to a delta function in $\theta$ multiplied by
an off-diagonal elementary matrix.  Letting $R \to \infty$ we find
that the neighboring \titeica surfaces are related by a particular
type of unipotent projective transformation.

These unipotent factors are the ``glue'' that let us move from an
understanding of asymptotics in one direction, or in a sector, to the
global picture.  Using them, we find that each of the $2(d+3)$
transitions across an unstable direction reveals either a new edge or
a new vertex, alternating to give a chain that closes up to form a
convex polygon $P$ with $(d+3)$ vertices.  The polynomial affine
sphere is asymptotic to the cone over $P$.

After the fact, the triangles associated to the individual \titeica
surfaces can also be described directly in terms of the polygon $P$:
Each vertex of $P$ forms a triangle with its two neighbors, giving
$(d+3)$ \emph{vertex inscribed triangles} of $P$.  Each edge of $P$
forms a triangle with the lines extending the two neighboring edges,
giving $(d+3)$ \emph{edge circumscribed triangles} of $P$.  Each
transition across an unstable direction flips from a vertex inscribed
triangle to one of its neighboring edge circumscribed triangles, or
vice versa (see Figure \ref{fig:triangles}).

\begin{figure}
\begin{center}
\includegraphics[width=\textwidth]{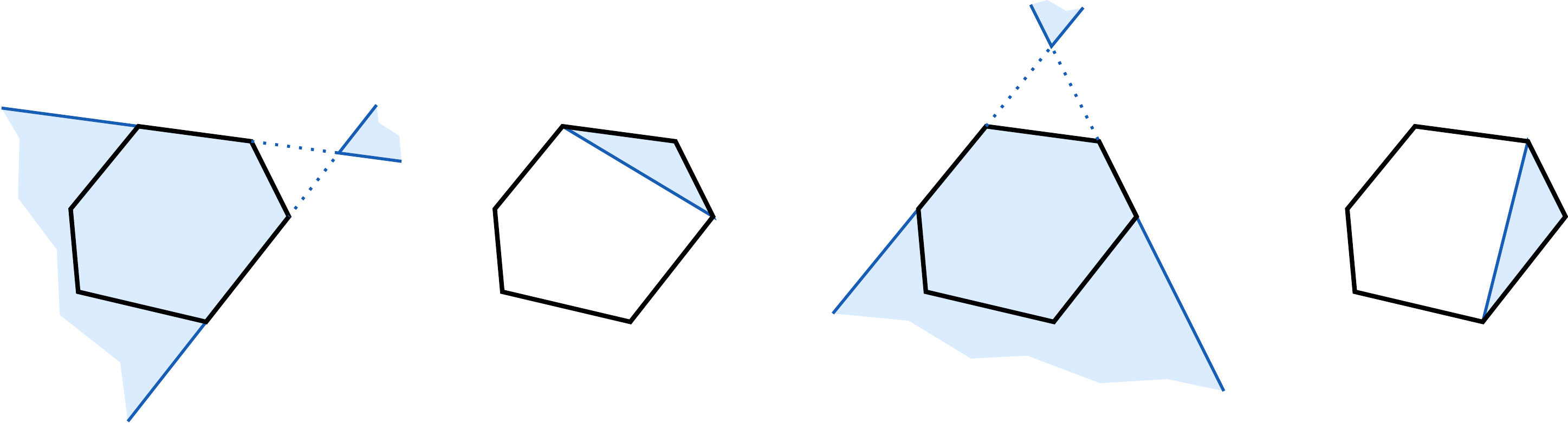}
\caption{The vertex inscribed triangles and edge circumscribed
  triangles of a convex polygon in $\RP^2$ alternate in a natural
  cyclic order.  \titeica surfaces over these triangles assemble into
  an asymptotic model for the affine sphere over the polygon.}
\label{fig:triangles}
\end{center}
\end{figure}

\subsection*{From polygons to polynomials}
Constructing the inverse map $\boldalpha^{-1} : \ModPoly_{d+3} \to
\ModCubic_d$ amounts to showing that an affine sphere over a convex
polygon has parabolic conformal type (i.e.~is isomorphic to $\C$ as a
Riemann surface) and that the Pick differential is a polynomial in the
uniformizing coordinate.  These properties are established in Section
\ref{sec:polygons-to-polynomials}.

In this case there is no question of existence, as the Cheng-Yau
theorem gives a complete affine sphere over any convex polygon (or
indeed any properly convex set).  To understand the conformal type and
Pick differential of this affine sphere, we once again use the
\titeica surface as the key model and comparison object.  The
arguments are somewhat simpler than in the construction of
$\boldalpha$ described above.

Benoist and Hulin, using interior estimates of Cheng-Yau, showed that
the $k$-jet of the affine sphere at a point depends continuously on
the corresponding convex domain \cite{benoist-hulin:finite-volume}.  This implies
that the Blaschke metric and Pick differential vary continuously in
this sense (see Theorem \ref{thm:support-function-continuous} and
Corollary \ref{cor:blaschke-and-pick-continuous} below).  We use this
continuity principle and projective naturality to compare
the affine sphere over a polygon to the \titeica surface over one of
its vertex inscribed polygons.

We find that the Blaschke metric and the Pick differential of the
polygonal affine sphere are comparable (with uniform multiplicative
constants) to those of the \titeica surface, at any point that is
sufficiently close to the edges shared by the triangle and the
polygon.  Applying this construction to each of the vertex inscribed
triangles, we find a ``buffer'' around the boundary of the polygon
that contains no zeros of the Pick differential, and where the
Blaschke metric is approximately Euclidean.  It follows
easily that the Pick form has finitely many zeros, that the
conformal type of the Blaschke metric is $\C$, and finally that the
Pick form is a polynomial.

\subsection*{Mapping of moduli spaces}

Having established that polynomial affine spheres correspond to
polygons and vice versa, the proof of the main theorem is completed by
Theorem \ref{thm:main-in-text}, where we use estimates and results
from the preceding sections to show that this bijection between moduli
spaces and its inverse are continuous.  Also, since the normalization
conventions of Sections \ref{sec:polynomials-to-polygons} and
\ref{sec:polygons-to-polynomials} implicitly require working in manifold
$\Z/(d+3)$-covers of the moduli spaces, we verify at this stage
that our the constructions have the necessary equivariance properties
to descend to the orbifolds themselves.

\subsection*{Related error estimates}

The comparison between rates of error decay and frame growth which
yields the finite set of asymptotic \titeica surfaces for a polynomial
affine sphere is an apparently novel element of our work on this class
of surfaces.  However, we build upon a substantial history of error
estimates for Wang's equation and other geometric PDE.  While both
Loftin \cite{loftin:parameterization} and Labourie
\cite{labourie:flat-projective} focused on the Wang equation in their
studies of affine spheres, Loftin (\cite{loftin:compactification-i},
\cite{loftin:limits}, \cite{loftin:neck-separation}) began and
developed a theory of error estimates for this equation.  The
structurally similar Bochner equation governing harmonic maps to
hyperbolic surfaces had an analogous development of error estimates in
\cite{minsky-harmonic}, \cite{wolf:degenerations}, and
\cite{han}. Crucial to refined error estimates are the use of sub- and
super-solutions, first constructed in the present setting by Loftin
\cite{loftin:compactification-i}; for the analogous Bochner equation
and for open Riemann surfaces, this technique began with
\cite{wan:thesis} (see also \cite{au-wan94}).

The passage from estimates for the solution of Wang's equation to an
asymptotic description of the affine frame field requires an
understanding the behavior of solutions to certain ordinary differential
equations (i.e.~the affine structure equations restricted to a curve).
The technique we use here was introduced by Loftin in
\cite{loftin:compactification-i}; there was no lower-dimensional
Bochner equation analogue to this technique.

\subsection*{Other perspectives}

To conclude the introduction we briefly comment on the potential
relations between our main theorem and techniques from areas such as
Higgs bundles and minimal surfaces.  Some more concrete conjectures
and possible extensions of our work are described after the proof of
the main theorem, in Section \ref{sec:conjectures}.

In \cite{labourie:flat-projective}, Labourie interprets the
parameterization of convex $\RP^2$ structures by cubic differentials
in terms of certain rank-$3$ Higgs bundles, thus identifying Wang's
equation for the Blaschke metric of an affine sphere with Hitchin's
self-duality equations for these bundles.  The same construction
applies in our setting: Compactifying $\C$ to $\CP^1$, our polynomial
affine spheres correspond to \emph{irregular Higgs bundles} on the
projective line defined by a vector bundle and a meromorphic
endomorphism-valued $1$-form (which in this case has a
single high-order pole).

Our existence theorem for Wang's equation (Theorem
\ref{thm:existence}) is therefore equivalent to the existence of
solutions to the self-duality equations for these bundles.  The
self-duality equations have been studied in the irregular case by
Biquard-Boalch \cite{biquard-boalch}, extending Simpson's work for the
case of a simple pole \cite{simpson}.  However, there is a mismatch of
hypotheses preventing us from deducing Theorem \ref{thm:existence}
directly from results in the Higgs bundle literature: The
\emph{irregular type} of the Higgs fields we consider, i.e.~the
Laurent expansion of the Higgs field at the pole, has nilpotent
coefficient in the most singular term, while the type is usually
assumed to be a regular semisimple element (for example in the results
of \cite{biquard-boalch}).  Witten has proposed a way to generalize
results from the semisimple case to arbitrary irregular types using
branched coverings \cite{witten:wild-ramification}, however a sketch
is provided only in rank $2$ and the details of a corresponding
existence theorem have not appeared.  We expect that pursuing these
ideas further, one could give an alternate proof of Theorem
\ref{thm:existence} entirely by Higgs bundle methods.

The connection with irregular Higgs bundles also suggests that our
main theorem could be related to the general phenomenon of equivalence
between \emph{Dolbeault moduli spaces} of Higgs bundles and
\emph{Betti moduli spaces} of representations of the fundamental group
of a Riemann surface into a complex Lie group.  It is known that the
generalized Betti moduli space which corresponds to irregular Higgs
bundles includes extra \emph{Stokes data} at each pole of the Higgs
field (see e.g.~\cite{boalch:geometry-and-braiding}).  This data
consists of a collection of unipotent matrices cyclically ordered
around each pole, and the possibility of a connection with the
unipotent factors we find for a polygonal affine sphere in Lemma
\ref{lem:unipotent-factors} is intriguing. We hope that by further
developing this connection, one could interpret Theorem
\ref{introthm:main} as identifying the space of convex polygons in
$\RP^2$ (stratified by the number of vertices) as a generalization of
the Hitchin component of $\SL_3\R$ representations to the punctured
Riemann surface $\C$ (with a stratification by the order of pole of
the associated Higgs field at infinity).

Finally, we mention another interpretation of the affine sphere
construction discussed by Labourie in \cite{labourie:flat-projective}:
The combination of the Blaschke metric and the affine frame field of a
convex surface in $\R^3$ induces a map to the symmetric space $\SL_3\R
/ \SO(3)$ which is a minimal immersion exactly when the
original surface is an affine spherical immersion (the local version
of being an affine sphere).  By this construction, the flats of the
symmetric space correspond to \titeica affine spheres.  Our main
theorem can therefore be seen as identifying a moduli space of minimal
planes in $\SL_3\R / \SO(3)$ that are asymptotic to finite collections
of flats (in some sense that corresponds to Theorem
\ref{thm:exponential-bound}) with a space of polynomial cubic
differentials.  It would be interesting to develop this picture more
fully, for example by characterizing these minimal planes directly in
terms of Riemannian geometry of the symmetric space, and possibly
generalizing to the symmetric space of $\SL_n\R$ for $n > 3$.

\subsection*{Acknowledgments}

The authors thank David Anderson, Steven Bradlow, Oscar Garcia-Prada,
Fran\c{c}ois Labourie, St\'ephane Lamy, John Loftin, and Andrew
Neitzke for helpful conversations.  They thank Qiongling Li, Xin Nie
and an anonymous referee for pointing out errors in an earlier draft
of this paper and the referee for a careful reading and helpful advice
on the final draft.  The authors gratefully acknowledge support from
U.S.~National Science Foundation through individual grants DMS 0952865
(DD), DMS 1007383 (MW), and through the GEAR Network (DMS 1107452,
1107263, 1107367, ``RNMS: GEometric structures And Representation
varieties'') which supported several workshops and other programs
where parts of this work were conducted.  MW appreciates the
hospitality of the Morningside Center where some of this work was
done.

\section{Polygons}
\label{sec:polygons}

As in the introduction we will consider polygons in $\RP^2$ up to the
action of the group $\SL_3\R \simeq \mathrm{PGL}_3\R$ of projective
transformations.  An elementary fact we will frequently use is:

\begin{prop}
The group $\SL_3\R$ acts simply transitively on $4$-tuples of points
in $\RP^2$ in general position. \noproof
\end{prop}

Our convention is that a \emph{convex polygon} in $\RP^2$ is a bounded
open subset of an affine chart $\R^2 \subset \RP^2$ that is the
intersection of finitely many half-planes.  In particular a polygon is
an open $2$-manifold homeomorphic to the disk $D^2$.  As usual, a
polygon can also be specified by the $1$-complex that forms its
boundary, or by its set of vertices.

By \emph{orientation} of the polygon we mean an orientation of its
interior, as a $2$-manifold, or equivalently an orientation of its
boundary as a $1$-complex.  Whenever we list the vertices of an
oriented polygon, it is understood that the list is ordered
consistently with the orientation.

\subsection{Spaces of polygons}
Let $\Poly_k$ denote the set of oriented convex polygons in $\RP^2$
with $k$ vertices (briefly, \emph{convex $k$-gons}); this is an open
subset of the symmetric product $\Sym^k(\RP^2)$.  The group
$\SL_3\R$ of projective automorphisms acts on $\Poly_k$ with quotient
$$\ModPoly_k = \Poly_k / \SL_3\R,$$
the \emph{moduli space} of convex polygons.  For $P \in \Poly_k$ we
denote by $[P]$ its equivalence class in $\ModPoly_k$.

By suitably normalizing a polygon, one can construct a natural
``quasi-section'' of the map $\Poly_k \to \ModPoly_k$: Choose an
oriented convex quadrilateral $Q_0 \subset \RP^2$ with vertices
$(q_1,q_2,q_3,q_4)$.
This polygon $Q_0$ and the labeling of its vertices will be fixed
throughout the paper.  We say that $P \in \Poly_k$, $k\geq 4$, is
\emph{normalized} if it is obtained from $Q_0$ by attaching a convex
$(k-2)$-gon to its $(q_4,q_1)$ edge.  Equivalently, an oriented convex polygon is
normalized if its vertices are
$$ (q_1, q_2, q_3, q_4, p_5, \ldots, p_k)$$
for some $p_i \in \RP^2$, $i = 5 \ldots k$.  In particular the
vertices of a normalized polygon have a canonical labeling by $1,
\ldots, k$.

Let $\NormPoly_k \subset \Poly_k$ denote the set of normalized convex
$k$-gons in $\RP^2$.  Having fixed four vertices, the set
$\NormPoly_k$ is naturally is an open subset of $(\RP^2)^{k-4}$.  In
fact $\NormPoly_k$ is contractible:

\begin{prop}
\label{prop:polygon-global-topology}
The space $\NormPoly_k$ is diffeomorphic to $\R^{2n-8}$.
\end{prop}

\begin{proof}
Choose an affine chart of $\RP^2$ in which $q_1 = (0,1)$, $q_2 =
(0,0)$, $q_3 = (1,0)$, and $q_4 = (1,1)$.  By convexity, the remaining vertices
of a normalized polygon must lie in the half-strip
$$\{ (x,y) \suchthat 0 < x <
1, \:y > 1 \}.$$  Setting $p_k = (x_k,y_k)$ we also have that $x_k$ is
monotonically decreasing with $k$, while the slopes $m_k$ of the
segments $\overline{p_{k-1}p_k}$ are monotonically increasing.  Up to
those two constraints, we may freely choose the pairs $(x_j, m_j)$ for
$j=5,...,k$, which determine the polygon completely.  Thus
$\NormPoly_k$ is parameterized by a product of open simplices:
$$\{(x_j, m_j) \suchthat  1 > x_5>\cdots>x_k>0, \: -\infty < m_5 < \cdots < m_k < \infty\} \simeq
D^{2k-8}.$$
\end{proof}

It is easy to see that $\NormPoly_k$ intersects every $\SL_3\R$-orbit
in $\Poly_k$: Given $P \in \Poly_k$, any four adjacent vertices of $P$
are in general position and can therefore be mapped to
$(q_1,\ldots,q_4)$ by a unique element $A \in \SL_3\R$.  Thus $A
\cdot P \in \NormPoly_k$ and we say that $A$ \emph{normalizes} $P$.

The only choice in the normalization construction is that of a vertex
to map to $q_1$, and so there are exactly $k$ ways to normalize an
oriented convex $k$-gon $P$ (though possibly some of them give the
same normalized polygon).  Equivalently, the set of intersection
points of a $\SL_3\R$-orbit in $\Poly_k$ with $\NormPoly_k$ has
cardinality at most $k$, and the projection $\NormPoly_k \to
\ModPoly_k$ is finite-to-one.

Furthermore, each fiber of this projection is the orbit of a natural
$\Z/k$-action on $\NormPoly_k$: Given $P \in \NormPoly_k$ with
vertices $(q_1,\ldots,q_4,p_5,\ldots, p_k)$, there is a projective
transformation $A = A(P)$ uniquely determined by its action on the
$4$-tuple,
$$ A : (q_2,q_3,q_4,p_5) \mapsto (q_1,q_2,q_3,q_4), $$
Then defining $\rot(P) = A(P) \cdot P$ we have a map $\rot : \NormPoly_k
\to \NormPoly_k$.  By construction $P$ and $\rot(P)$ lie in the same
$\SL_3\R$-orbit and $\rot^k = \id$ follows since $\rot^k(P) = B \cdot
P$ where $B$ is a projective transformation fixing the vertices of
$Q_0$, hence $B = \id$.  It is straightforward to check that $\rot$
acts diffeomorphically on $\NormPoly_k \subset (\RP^2)^{k-4}$, and in
fact it acts by the restriction of a rational map defined over $\Z$.

Summarizing the discussion above, we find:

\begin{prop}
\label{prop:polygon-cyclic-quotient}
The projection $\NormPoly_k \to \ModPoly_k$ can be identified with the
quotient map of the $\Z/k$ action on $\NormPoly_k$.  Thus
$\ModPoly_k$ has the structure of an orbifold with universal
cover $\NormPoly_k \simeq \R^{2n-8}$. \noproof
\end{prop}

The fact that any convex polygon has a unique normalization once a
vertex is chosen also shows that $\NormPoly_k$ can be identified with
a quotient space related to $\ModPoly_k$: If we consider a space of
pairs $(P,v)$ where $P \in \Poly_k$ and $v$ is a vertex of $P$, then
the quotient of this space of \emph{labeled polygons} by $\SL_3\R$
is in canonical bijection with $\NormPoly_k$.  In this
description, the $\Z/k$ action cycles $v$ around $P$ while the map
$(P,v) \mapsto P$ corresponds to $\NormPoly_k \to \ModPoly_k$.

The notations $\NormPoly_k, \ModPoly_k$ are meant to suggest
\emph{Teichm\"uller space} and \emph{Moduli space}, respectively.  In
this analogy, the additional data of a normalization of a polygon is
like the marking of a Riemann surface, and the $\Z/k$ action
plays the role of the mapping class group.  On both sides of the
analogy, the Teichm\"uller space is contractible and smooth, while the
quotient moduli space is only an orbifold.  (Compare the ``toy model''
of the space of convex projective structures described in
\cite{fock-goncharov:real-projective}.)

So far we have considered our spaces of polygons to be topological
spaces using what could be called the \emph{vertex topology}, i.e.~by
considering the set of vertices of the polygon as a point in the
symmetric product of $\RP^2$.  Alternatively, one could introduce a
topology on polygons using the Hausdorff metric on closed subsets of $\RP^2$; here
(and throughout the paper) we take the closures of the convex domains
whenever the Hausdorff topology is considered.

In fact these topologies are equivalent; it is immediate that vertex
convergence implies Hausdorff convergence, and conversely we have:

\begin{prop} \label{prop:hausdorff-vertex-equivalent} If a sequence
$P_n$ of convex polygons converges in the Hausdorff topology to a
convex polygon $P$, and if $P$ and $P_n$ all have the same number of
vertices, then $P_n$ also converges to $P$ in the vertex topology.
\end{prop}

\begin{proof}
First observe\footnote{This is an example of a general
fact in convex geometry: In a Hausdorff-convergent sequence of compact convex
subsets of a $\R^N$, each extreme point of the limit set is a limit of
extreme points of the sequence.}
 that each vertex of $v$ of $P$ is a limit of a sequence
$v_n$ of vertices of $P_n$: Otherwise $v$ would be a limit of
interior points of edges, but not of their endpoints.  By passing to a
subsequence we could arrange for this sequence of edges to converge, giving
a line segment in $P$ of which $v$ is an interior point, contradicting
the assumption that $v$ is a vertex.

Suppose $P$ has $k$ vertices.  Applying the observation above to each
vertex in turn we obtain $k$ sequences, each having as
$n^{\mathrm{th}}$ element a vertex of $P_n$.  While these sequences
may overlap for some $n$, this can only happen for finitely many
terms since these sequences have distinct limit points.  Thus for all
large $n$ we have labeled the $k$ vertices of $P_n$ in such a way that
vertices with a given label converge, as required.
\end{proof}

\subsection{Example: pentagons}
We have $\NormPoly_5 \simeq \R^2$, and in fact the space is naturally
an open triangle in $\RP^2$: The fifth vertex $p_5$ can be any point
inside the triangle on the exterior of $Q_0$ formed by the $[q_4,q_1]$
edge and the lines extending its neighboring edges.  (This model of
$\NormPoly_5$ is shown in Figure \ref{fig:circles} in Section
\ref{sec:conjectures} below.)

Working in the affine coordinates of Proposition
\ref{prop:polygon-global-topology} the space becomes a half-strip,
$$\NormPoly_5 = \{ (x,y) \suchthat 0 \leq x \leq 1, \: y > 1 \}.$$  In the
same coordinates, the generator of the $\Z/5$ action is the Cremona
transformation
$$  \rot(x,y) = \left ( \frac{y(y-1)}{x^2 + x(y-1) + y(y-1)}, \frac{y(x +
  y -1)}{(x^2 + x(y-1) + y(y-1))} \right )$$
which can be expressed in a suitable homogeneous coordinate system as
$$ [X:Y:Z] \mapsto [X(Z-Y):Z(X-Y):XZ].$$
This birational automorphism of $\P^2$ resolves to a biregular
automorphism of a degree $5$ del Pezzo surface by blowing up the four
vertices of $Q_0$ (see e.g.~\cite[Sec.~4.6]{de-fernex}).  We are grateful
to St\'ephane Lamy for explaining this projective model.  Returning to the
affine half-strip model above, the unique fixed point of the $\Z/5$
action on $\NormPoly_5$ is $\left ( \tfrac{1}{2}, \tfrac{1}{4}(3 +
\sqrt{5}) \right )$, which corresponds to the regular pentagon.  The
differential of $\rot$ at this fixed point is linearly conjugate to a
rotation by $3/5$ of a turn.

Topologically, the $\Z/5$ action rotates the interior of triangle $\NormPoly_5$
about the fixed point.  Blowing up the two vertices the
triangle shares with $Q_0$ gives a pentagon in which the $\Z/5$ action
can be seen as standard pentagonal rotational symmetry.

The quotient $\ModPoly_5 = \NormPoly_5 / \langle \rot \rangle$ is a
topological open disk with an interior orbifold point (cone point) of
order $5$.

\section{Cubic differentials}
\label{sec:cubic}

We define a \emph{polynomial cubic differential} to be a holomorphic
differential on $\C$ of the form $C(z) dz^3$, where $C(z)$ is a
polynomial function.

\subsection{Spaces of cubic differentials}
Let $\Cubic_d \simeq \C^* \times \C^d$ denote the vector
space of polynomial cubic differentials of degree $d$ (with nonzero
leading coefficient).

The group $\Aut(\C) = \{ z \mapsto az+b \}$ acts on these
differentials by pushforward.  Let $\ModCubic_d$ denote the quotient
of $\Cubic_d$ by this action.  Given a cubic differential $C$, we
denote its equivalence class by $[C]$.

As in the polygon case, the relationship between $\Cubic_d$ and
$\ModCubic_d$ is clarified by considering an intermediate object
space of ``normalized'' objects:  If a polynomial cubic differential
is written as
\begin{equation}
\label{eqn:polynomial}
 C = \left ( c_d z^d + c_{d-1} z^{d-1} + \cdots + c_0 \right )dz^3
\end{equation}
then we say $C$ is \emph{monic} if $c_d=1$ and \emph{centered} if
$c_{d-1}=0$.  The latter condition means that the roots of $C$ sum to
zero.  A cubic differential that is both monic and centered is \emph{normalized}.

Let $\NormCubic_d \subset \Cubic_d$ denote the space of normalized
polynomial cubic differentials. The set $\NormCubic_d\simeq \C^{d-1}$
intersects every $\Aut(\C)$-orbit in $\Cubic_d$.  Note that $z \mapsto
(az + b)$ maps the differential \eqref{eqn:polynomial} to
$$T^* C = \left ( c_d a^{d+3}
(z + b/a)^d + c_{d-1} a^{d+2} (z + b/a)^{d-1} + \cdots + c_0
a^3\right)dz^3.$$ Thus acting by $z \mapsto c_d^{-1/(d+3)} z + b$ makes an
arbitrary differential monic, and a suitable translation factor $b$
moves the root sum to zero.  Moreover, if two normalized cubic
differentials are related by $T(z) = az+b$, then the monic condition
gives $a^{d+3} = 1$ and the centering implies $b=0$, hence $T$ is
multiplication by a $(d+3)$-root of unity.

We therefore recover a description of $\ModCubic_d$ as a space of
orbits in $\NormCubic_d$ of the group $\unity_{d+3} \simeq \Z/(d+3)$
of roots of unity, where $\zeta \in \unity_{d+3}$ acts on coefficients
by
\begin{equation}
\label{eqn:weighted-action}
(c_{d-2}, c_{d-3}, \ldots, c_{0}) \mapsto (\zeta^{d+1} c_{d-2},
\zeta^d c_{d-3}, \ldots, \zeta^3 c_0).
\end{equation}
The quotient by this action is a ``weighted affine space'', i.e.~an
affine chart of the weighted projective space $\CP(d+3,\hat{d+2},d+1,
\ldots, 3)$.  To summarize,
\begin{prop}
The space $\ModCubic_d$ is a complex orbifold (and a complex algebraic
variety), the quotient of $\NormCubic_d \simeq \C^{d+1}$ by the action
of $\unity_{d+3}$ described in \eqref{eqn:weighted-action}.
\end{prop}

\subsection{Example:  Quadratic polynomials}

We have $\NormCubic_2 = \{ (z^2 + c) dz^3 \} \simeq \C$.  The action
of $\Z/5$ is generated by the rotation $z \mapsto \zeta z$, where $\zeta =
\exp(2 \pi i/5)$, acting by $c \mapsto \zeta^3 c$.  The unique
fixed point is $c=0$ and the quotient $\ModCubic_2 = \NormCubic_2 /
\langle \zeta \rangle$ is a Euclidean cone with cone angle $2 \pi /5$.
Alternatively $\ModCubic_2$ is the affine chart of $\CP(5,3)$ in which
the first homogeneous coordinate is nonzero.

\subsection{Natural coordinates}

A \emph{natural coordinate} for a cubic differential $C$ is a local
coordinate $w$ on an open subset of $\C$ in which $C = 2 dw^3$.  (The
factor of $2$ here is not standard, but it simplifies calculations
later.)  Such a coordinate always exists locally away from the zeros
of $C$, because near such a point one can choose a holomorphic cube
root of $C$ and take
\begin{equation}
\label{eqn:developing-map}
w(z) = \int_{z_0}^z \left ( \tfrac{1}{2}C\right )^{1/3}.
\end{equation}
Up to adding a constant, every natural coordinate for $C$ has this
form.  Thus any two natural coordinates for $C$ differ by
multiplication by a power of $\omega = \exp(2 \pi i/3)$ and adding a
constant.

This local construction of a natural coordinate (and the integral
expression above) analytically continues to any simply-connected set
in the complement of the zeros of $C$ to give a \emph{developing map},
a holomorphic immersion that pulls back $2 dw^3$ to $C$.  The
developing map need not be injective, but any injective restriction of
it is a natural coordinate.

The metric $|C|^{2/3}$ defines a flat structure on $\C$ with
singularities at the zeros of $C$.  In a natural coordinate this is
simply the Euclidean metric $2^{2/3} |dw|^2$.  We call this the
\emph{flat structure} or \emph{flat metric} associated to $C$.  A zero
of $C$ of multiplicity $k$ is a cone point with angle
$\frac{2\pi}{3}(3 + k)$.  Straight lines in a natural coordinate are
geodesics of this metric.

\subsection{Half-planes and rays} \label{subsec:half-planes-rays}
We define a $C$-\emph{right-half-plane} to be a pair $(U,w)$ where $U
\subset \C$ is open and $w$ is a natural coordinate for $C$ that maps $U$
diffeomorphically to the right half-plane $\{ \Re w > 0 \}$.  Note
that $U$ then determines $w$ up to addition of a purely imaginary
constant.  Given a half-plane $(U,w)$ there is an associated family of
parallel right-half-planes $(U^{(t)},w^{(t)})$, $t \in \R^+$, defined by $w^{(t)}
= w - t$ and $U^{(t)} = w^{-1}( \{ \Re w > t \}) \subset U$.

A path in $\C$ whose image in a natural coordinate for $C$ is a
Euclidean ray with angle $\theta$ will be called a $C$\emph{-ray with
  angle $\theta$}.
Note that the angle is well-defined mod $2 \pi/3$.  In a
suitable natural coordinate, a $C$-ray has the parametric form $t \mapsto b
+ e^{i \theta} t$.

Similarly, a $C$-\emph{quasi-ray with angle $\theta$} is a path
that can be parameterized so that its image in a natural coordinate
$w$ has the form $t \mapsto e^{i \theta} t + \delta(t)$ where
$|\delta(t)| = o(t)$.

Before discussing the geometry of an arbitrary polynomial cubic
differential, we describe a configuration of $C$-rays and
$C$-right-half-planes for the cubic differential $C=z^d dz^3$ that we intend to generalize:  Consider the
``star'' formed by the $(d+3)$ Euclidean rays from the origin in $\C$,
$$ \star_d := \{ z \suchthat z^{d+3} \in \R^+ \} = \{ \arg z = 0 \: \text{ mod } 2
\pi/(d+3) \}.$$ Since there is a natural coordinate for $z^d dz^3$
that is a real multiple of $z^{(d+3)/3}$, these are also $z^d
dz^3$-rays with angle zero.

Now consider Euclidean sectors of angle $3 \pi / (d+3)$ centered on
each of the rays in $\star_d$; each such sector is naturally a $z^d
dz^3$-right-half-plane in which the corresponding ray of $\star_d$
maps to $\R^+$.  These sectors are pairwise disjoint except when
surrounding neighboring rays in $\star_d$, in which case they overlap
in a sector of angle $\pi/(d+3)$.  In particular this overlap, when
nonempty, maps by a natural coordinate to a sector in that coordinate of angle $\pi/3$.
Finally, we observe that each of these Euclidean sectors of 
angle $3 \pi / (d+3)$ is contained in the region
between its neighboring rays from $\star_d$.

Thus we have constructed a system $\{(U_k,w_k)\}_{k=0,\ldots,d+2}$ of
$z^ddz^3$-right-half-planes that cover $\C^*$ and which are neighborhoods
of the rays in $\star_d$, each neighborhood being disjoint from the other
rays.  Replacing these with the associated parallel half-planes
$(U_k^{(t)},w_k^{(t)})$, for some $t \in \R^+$, we have a collection
of ``eventual neighborhoods'' of the rays of $\star_d$ that cover all
but a compact set in $\C$.

\begin{figure}
\begin{center}
\includegraphics[width=\textwidth]{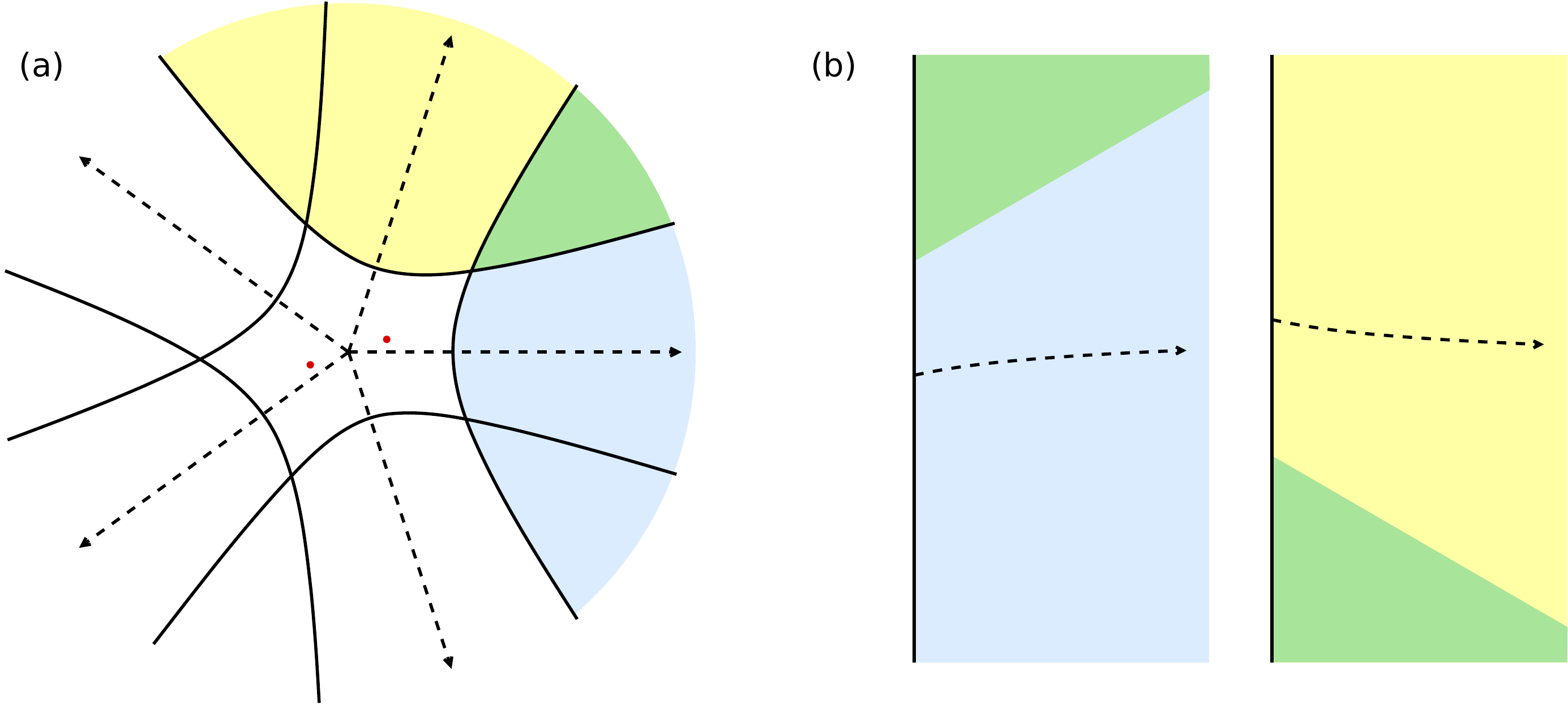}
\caption{(a) Covering a neighborhood of infinity by five $C$-right-half-planes
  for the differential $C=(z^2 - (3+i)^2)dz^3$.  (b) The edges of
  $\star_d$ (dashed) map to $C$-quasi-rays in the natural coordinates of
  these half-planes.}
\label{fig:halfplanes}
\end{center}
\end{figure}

In fact a collection of half-planes like this exists for any monic
polynomial cubic differential, except that the rays of $\star_d$ will
now only be \emph{quasi-rays} in their respective half-plane
neighborhoods.  Specifically, we have:

\begin{prop}[Standard half-planes]
\label{prop:standard-half-planes}
Let $C$ be a monic polynomial cubic differential.  Then there are
$(d+3)$ $C$-right-half-planes $\{ (U_k, w_k) \}_{k=0,\ldots,d+2}$ with the
following properties:
\begin{rmenumerate}
\item The complement of $\bigcup_k U_k$ is compact.

\item The ray $\{ \arg(z) = \frac{2 \pi k}{d+3} \}$ is
eventually contained in $U_k$.

\item The rays $\{ \arg(z) = \frac{2 \pi (k\pm 1)}{d+3} \}$ 
are disjoint from $U_k$.

\item On $U_k \cap U_{k+1}$ we have $w_{k+1} = %\exp(-2 \pi i/3) 
\omega^{-1}
w_k +
c$ for some constant $c$, and each of $w_k, w_{k+1}$ maps this
intersection onto a sector of angle $\pi/3$ based at a point on $i \R$. (Recall $\omega = \exp(2
\pi i/3)$.)

\item Each ray of $\star_d$ is a $C$-quasi-ray of angle zero in the
associated half-plane $U_k$.  More generally any Euclidean ray in $\C$
is a $C$-quasi-ray and is eventually contained in $U_k$ for some $k$.
\end{rmenumerate}
\end{prop}

Figure \ref{fig:halfplanes} shows an example of the configuration of
half-planes given by this proposition.

Considering $C$ as a meromorphic differential on $\CP^1$, this
proposition describes the local structure of natural coordinates near
a higher-order pole, and in this formulation it is well-known.  For
example, the corresponding description of half-planes for a
meromorphic quadratic differential is given in
\cite[Sec.~10.4]{strebel:book}, and those arguments are easily adapted
to cubic differentials.  A proof of the proposition above is given in
Appendix \ref{appendix:half-planes}.

\section{Affine spheres and convex sets}
\label{sec:affine}

In this section, we briefly recall the definitions and results on
affine spheres necessary to prove our main results.  For more detailed
background material we refer the reader to \cite{nomizu-sasaki}
\cite{li-simon-zhao} \cite{loftin:survey}.

\subsection{Affine spheres} \label{subsec:affine sphere basics}

We consider locally strictly convex surfaces $M \subset \R^3$.  A
basic construction in affine differential geometry associates to such
a surface a transverse vector field $\xi \transverse M$, the
\emph{affine normal} field, which is equivariant with respect to
translations and the linear action of $\SL_3\R$.  An \emph{affine
  sphere} is a surface whose affine normal lines are concurrent at a
point (the \emph{center}).  By applying a translation we can
move the center of the affine sphere to the origin, in which case we have
$$ \xi(p) = -H p, \text{ for all } p \in M \subset \R^3$$
for some constant $H \in \R$, the \emph{mean curvature}.  We assume
this normalization of the center from now on.  We will consider only
\emph{hyperbolic} affine spheres, which are those with $H < 0$; by
applying a dilation such a sphere can be further normalized so that $H
\equiv -1$.

The second fundamental form of the convex surface $M$ (relative to the
transversal $\xi$) can be used to define an 
$\SL_3\R$-invariant Riemannian metric $h$, the \emph{Blaschke metric};
specifically, this metric is seen in the Gauss equation which
decomposes the flat connection of $\R^3$ into its tangential ($TM$)
and normal $(\R\xi)$ components:
$$ D_X Y = \nabla_X Y + h(X,Y) \xi, \;\;X,Y \in \Vect(M)$$
We then have two connections on $TM$: The tangential component
$\nabla$ of the flat connection, and the Levi-Civita connection
$\nabla^h$ of the Blaschke metric.  The difference $(\nabla -
\nabla^h)$ is a tensor of type $(2,1)$, and using the isomorphism $TM
\simeq T^*M$ induced by $h$ we have an associated cubic form $A$ on
$TM$, the \emph{Pick form} \cite{pick}.

We use the conformal class of the Blaschke metric to regard $M$ as a
Riemann surface.  Blaschke showed that for an affine sphere, the Pick
form $A$ is the real part of a holomorphic cubic differential $C =
C(z) dz^3$ \cite[p.~211]{blaschke}.  We call $C$ (which is uniquely
determined by $A$) the \emph{Pick differential} of the affine sphere.

All of the affine differential-geometric constructions above are
local and can therefore be applied to immersed (rather than embedded)
locally strictly convex surfaces in $\R^3$.  This gives the notion of
an \emph{affine spherical immersion} $f: M \to \R^3$ of a Riemann
surface $M$ into $\R^3$ and associated Blaschke conformal metric $h$
and Pick differential $C \in H^0(M,K_M^3)$.

We will refer to an affine spherical immersion as \emph{complete} if
the domain is complete with respect to the Blaschke metric.  (This
notion is sometimes called ``affine complete''.)  Completeness in this
sense has strong consequences:

\begin{thm}[A.-M.~Li \cite{li:calabi-conjecture-1} \cite{li:calabi-conjecture-2}]
\label{thm:completeness-equivalence}
If an affine spherical immersion $f : M \to \R^3$ is complete, then it
is a proper embedding, and its image is the boundary of an open convex set. \noproof
\end{thm}

While stated here only in the $2$-dimensional case, the result of Li
applies to affine spheres of any dimension.  The proof uses some of
the estimates and the techniques developed by Cheng-Yau to prove a
fundamental existence theorem for affine spheres (stated below as
Theorem \ref{thm:cheng-yau}).  More recently, Trudinger and Wang
\cite{trudinger-wang:affine-complete} have shown that completeness of
the Blaschke metric gives the same conclusions as the theorem above
(properness, etc.) under much weaker conditions---rather than
requiring the immersion to be affine spherical immersion, local strict
convexity alone is enough.

Completeness of an affine sphere also implies that the Blaschke metric
is nonpositively curved.  This was first proved (as a statement about
non-positive Ricci curvature for hyperbolic affine spheres of
dimension at least two) by Calabi \cite{calabi1972}, using a
differential inequality on the norm of the Pick differential.  A
different proof was given by Benoist-Hulin
\cite{benoist-hulin:finite-volume} using Wang's equation (see
\eqref{eqn:wang} below).  Li-Li-Simon \cite{Li-Li-Simon} show (by
techniques in the spirit of those of Calabi) that if the curvature
vanishes at a point, then it is identically zero and the surface is
projectively equivalent to a specific example, the \titeica surface,
which discussed in more detail below.  Summarizing, we have:

\begin{thm}[{\cite{calabi1972},
\cite{benoist-hulin:finite-volume}; \cite{Li-Li-Simon}}] 
\label{thm:npc} 
The Blaschke metric of a complete hyperbolic affine sphere $M$ has
nonpositive curvature.  In fact, the curvature is either strictly
negative or identically equal to zero, and in the latter case the
affine sphere is $\SL_3\R$-equivalent to a surface of the form
$x_1x_2x_3 = c$ for some nonzero constant $c$. \noproof
\end{thm}

\subsection{Frame fields}

In much the same way that a surface immersed in Euclidean space can be
locally reconstructed (up to an ambient Euclidean isometry) from its
first and second fundamental forms, an affine spherical immersion can
be recovered (up to the action of $\SL_3\R$) from the data of its
Blaschke metric and Pick differential.  In both cases one can consider
this reconstruction as the integration of a connection $1$-form with
values in a Lie algebra to obtain a suitable frame field on the
surface.

To describe the relevant integration process for an affine sphere, we
introduce the complexified frame field $F$ of an affine spherical
immersion $f : M \to \R^3$,
$$ F = \left ( f \; f_z \; f_{\bar{z}} \right ) \in \GL_3\C.$$
In fact, since $f$ is real-valued, the frame $F$ takes values in a
fixed right coset of $\GL_3\R$ within $\GL_3\C$.  Our standing assumption
that $f$ is normalized to have center at the origin means that $f$ is
proportional to its affine normal $\xi$, hence the components of this
frame give both the affine normal direction and a complex basis for
the tangent space.  Here and throughout we use $z$ to denote a local
conformal coordinate for the Blaschke metric. 

Following Wang \cite{wang} and Simon-Wang \cite{simon-wang}, the Gauss
and Weingarten structure equations for an affine spherical immersion
can be stated in terms of the Darboux derivative $F^{-1} dF$ of the
frame field.  Writing $h= e^{u}|dz|^2$ and $C=C(z) dz^3$, the frame
field of an affine sphere with mean curvature $H \equiv -1$ satisfies
\begin{equation} \label{eqn:structure}
F^{-1} dF =
\begin{pmatrix}
0 & 0 & \frac{1}{2}e^u\\
1 & u_z & 0\\
0 & C e^{-u} & 0
\end{pmatrix}
dz + 
\begin{pmatrix}
0 & \frac{1}{2}e^u & 0\\
0 & 0 & \bar{C} e^{-u}\\
1 & 0 & u_{\bar{z}}
\end{pmatrix}
d\bar{z}.
\end{equation}

The integrability of this $\gl_3\C$-valued form is equivalent to two
additional (structure) equations on $u$ and $C$:

\begin{equation}
C_{\bar{z}} =0 \label{eqn:PickHolomorphic}
\end{equation} 
\begin{equation}
\Delta u = 2 \exp(u) - 4 |C|^2 \exp(-2u) \label{eqn:wang}
\end{equation}

The first equation simply requires the cubic differential $C$ to be
holomorphic.  In the second equation \eqref{eqn:wang}, the (flat) Laplacian
$\Delta$ is the operator $\Delta = 4\frac{\partial^2}{\partial
  z \partial \bar{z}}$.  This nonlinear condition can be seen as
imposing a relationship between the curvature of the Blaschke metric
and the norm of the holomorphic cubic differential.  More precisely
\eqref{eqn:wang} can be written as
\begin{equation}
\label{eqn:wang-intrinsic}
 K(h) = -1 + 2 |C|_h^2,
\end{equation}
where $K(h)$ denotes the Gaussian curvature function of the conformal
metric $h = e^u |dz|^2$, and $|C|_h = |C| e^{-\frac{3}{2} u}$ is the
pointwise $h$-norm of the cubic differential $C$.

Condition \eqref{eqn:wang} (or its equivalent formulation
\eqref{eqn:wang-intrinsic}) is referred to in the affine sphere
literature as \emph{Wang's equation}.  In the paper \cite{wang},
C.~P.~Wang studied its solutions to develop an intrinsic theory of
affine spheres invariant under a cocompact group of automorphisms. It
is also a variant of the \emph{vortex equation} appearing in the gauge
theory literature (see e.g. \cite{jaffe-taubes}), as explained in the
next section.  Labourie has interpreted (in
\cite[Sec.~9]{labourie:flat-projective}) this equation as an instance
of Hitchin's self-duality equations for a rank-$3$ real Higgs bundle
$(E,\Phi)$ over $M$ with trivial determinant.  In this perspective the
Higgs field $\Phi$ is determined by the cubic differential $C$, a
unitary connection $A$ on $E$ comes from the metric $h$, and
integration of \eqref{eqn:structure} corresponds to finding a local
horizontal trivialization of the associated flat connection $A + \Phi
+ \Phi^*$.

Section \ref{sec:vortex} below is devoted to a study of solutions to
equation \eqref{eqn:wang-intrinsic} for polynomial cubic differentials
on $\C$, and to a generalization to polynomial holomorphic
differentials of any degree.  These PDE results are applied in
subsequent sections to construct the mapping of moduli spaces that
is the subject of the main theorem.

\subsection{Monge-Ampere, the Cheng-Yau Theorem and estimates}

As we noted in the discussion of the completeness of affine spheres in
Section~\ref{subsec:affine sphere basics}, the seminal existence
result for hyperbolic affine spheres is due to Cheng-Yau
\cite{cheng-yau}, with some clarifications on the notions of
completeness (see the last sentence of the statement below) due to
Gigena \cite{gigena}, Sasaki\cite{sasaki}, and Li
\cite{li:calabi-conjecture-1} \cite{li:calabi-conjecture-2}. (The book
\cite{li-simon-zhao} gives a comprehensive and coherent account of
this theory.)  The Cheng-Yau result says that hyperbolic affine
spheres of a given mean curvature in $\R^3$ correspond to properly
convex sets in $\RP^2$.

\begin{thm}[Cheng-Yau \cite{cheng-yau:complete-affine}] 
\label{thm:cheng-yau}
Let $\cone \subset \R^3$ denote an open convex cone which
contains no lines.  Then there is a unique complete hyperbolic affine
sphere $M \subset \R^3$ of mean curvature $H=-1$ which is asymptotic
to $\partial \cone$.

On the other hand, any complete affine sphere $M \subset \R^3$ with
center $0$ is asymptotic to the boundary of such a convex cone; this
cone can be described as the convex hull of $M \cup \{0\}$. \noproof
\end{thm}

The Cheng-Yau theorem emerges from an approach to affine differential
geometry through the analysis of nonlinear PDE of Monge-Ampere type.
This approach is quite different from the frame field integration
methods described above, and some of the estimates on affine invariants
that come from the Monge-Ampere theory will be used in subsequent
sections.  Therefore, we will now briefly review the basics of this
approach before formulating the estimates we need.

Cones $\cone$ of the type considered in the Cheng-Yau theorem
above correspond to properly convex open sets $\Omega \subset \RP^2$.
(\emph{Properly convex} means that the set can be realized as a
bounded convex subset of an affine chart.)  Given $\cone$, we
define $\Omega \subset \RP^2$ to be the set of lines through the
origin in $\R^3$ that intersect $\cone$ nontrivially (and hence
in a ray). Conversely, the union of lines corresponding to points
of a properly convex set $\Omega$ gives a ``double cone'' $\cup
\Omega$ in which the origin is a cut point.  Removing one of the two
sides of the origin gives a convex cone containing no line, to which
the Cheng-Yau theorem applies.

In fact, applying the Cheng-Yau theorem to either side of the double
cone $\cup \Omega$ gives essentially the same affine sphere; like the
cones themselves, the spheres are related by the antipodal map $p
\mapsto -p$.  Up to this ambiguity, one can therefore think of
complete affine spheres in $\R^3$ as being parameterized by properly convex
open sets in $\RP^2$.

Using this correspondence, we can consider any complete affine sphere
in $\R^3$ as being parameterized as a ``radial graph'' over its
corresponding projection to $\RP^2$.  More precisely, if we consider
$\Omega$ as a subset of an affine chart which we identify with the
plane $\{ (x,y,1) \in \R^3\}$, then the point of the affine sphere
that lies on the ray through $(x,y,1)$ has the form
$$ - \frac{1}{u(x,y)} \cdot (x,y,1) $$
where $u = u(x,y)$ is a certain real, negative function on $\Omega$.
We call $u = u_\Omega$ the \emph{support function} that defines the affine
sphere.  Since the surface is properly embedded and asymptotic to the
boundary of the cone, we have $u \to 0$ on the boundary of $\Omega$.
Moreover, the condition that the surface is an affine sphere becomes a
Monge-Ampere equation that the support function must satisfy:
$$ \det(\Hess(u)) = u^{-4}.$$
The Cheng-Yau theorem is established by studying the Dirichlet problem
for this equation (and its generalization to higher dimensions).

Following this approach, Benoist and Hulin use the maximum principle
and interior estimates of Cheng-Yau to show that the support function
and its derivatives depend continuously on the convex domain, in a
sense which we now describe.

Let $\frak{C}_*$ denote the set of pointed properly convex open sets
in the real projective plane $\RP^2$, i.e.
$$\frak{C}_*= \{(\Omega,x) \suchthat \Omega \subset \R^2 \subset
\RP^2 \text{ open, bounded, and convex}, x \in \Omega\}.$$ We equip
$\frak{C}_*$ with the product of the Hausdorff topology on closed sets
$\bar{\Omega}$ and the $\RP^2$ topology.  Similarly let $\frak{C}$
denote the set of open properly convex sets with the Hausdorff
topology.  Then we have:

\begin{thm}[{Benoist and Hulin \cite[Cor.~3.3]{benoist-hulin:finite-volume}}] \label{thm:support-function-continuous}
For any $k \in \N$, the $k$-jet of the support function $u_\Omega$ at
$p \in \Omega$ is continuous as a function of $(\Omega,p) \in
\frak{C}_*$.

More generally, the restriction of the support function and its
derivatives to a fixed compact set depend continuously on the domain.
That is, consider a properly convex open set $\Omega \subset \RP^2$, a
compact set $K \subset \Omega$, and a neighborhood $U$ of $\Omega$ in
$\frak{C}$ small enough so that $K \subset \Omega'$ for all $\Omega'
\in U$.  Then the restriction of the support function to $K$ varies
continuously in the $C^k$ topology as a function of $\Omega \in U$.
\end{thm}

\begin{remark}
The statement of Corollary 3.3 in \cite{benoist-hulin:finite-volume} involves only
the $k$-jet at a point.  However, their proof also establishes the 
uniform $C^k$ continuity on a compact subset that we have included
in the theorem above.  In fact, they derive the pointwise $k$-jet continuity
from the $C^k$ continuity on compacta.
\end{remark}

As pointwise differential invariants, the Blaschke metric and the Pick
form can be computed from derivatives of the support function.
In fact, the Blaschke metric at a point is determined by the $2$-jet of the
support function at that point, and the Pick differential by the $3$-jet.
Therefore, the continuous variation statement above also yields:

\begin{cor}
\label{cor:blaschke-and-pick-continuous}
The Blaschke metric and the Pick differential of the affine sphere over a
properly convex domain $\Omega$ depend continuously on the domain, in
the same sense considered in Theorem
\ref{thm:support-function-continuous} (i.e.~pointwise or on a
fixed compact subset).  \noproof
\end{cor}

\subsection{Fundamental examples}

Either connected component of a two-sheeted hyperboloid in $\R^3$ is a
hyperbolic affine sphere.  The center is the vertex of the cone to
which the hyperboloid is asymptotic, and the corresponding
convex domain in $\RP^2$ is bounded by a conic.  The \bla metric is
the hyperbolic metric, considering the hyperboloid as the Minkowski
model of $\H^2$, and the Pick differential vanishes identically.  This affine
sphere is homogeneous, in that it carries a transitive action of a
subgroup of $\SL_3\R$, in this case conjugate to $\mathrm{SO}(2,1)$.

Another homogeneous affine sphere, having nonzero Pick differential, will play
an essential role in our proof of the main theorem.  For any nonzero
real constant $c$, each connected component of the surface $x_1 x_2
x_3 = c$ in $\R^3$ is a hyperbolic affine sphere centered at the
origin \cite{titeica}; the mean curvature of this surface is constant, depending on
$c$.  We call any surface that is equivalent to one of these under the
action of $\SL_3\R$ (and hence also an affine sphere) a \emph{\titeica
  surface}.  Each such surface carries a simply transitive action of a
maximal torus of $\SL_3\R$; in the case of $x_1 x_2 x_3 = c$, it is
the diagonal subgroup.

We now introduce a parameterization of a \titeica surface that will be
used extensively in Section \ref{sec:polynomials-to-polygons}.  Let
$\h : \C \to \mathfrak{sl}_3\R$ be the map defined by
\begin{equation}  \label{eqn:hdef}
 \h(z) = \begin{pmatrix}
 2 \Re(z) & 0 & 0\\
0 & 2 \Re(z / \omega) & 0\\
0 & 0 & 2\Re(z / \omega^2)
\end{pmatrix},
\end{equation}
where $\omega = e^{2 \pi i / 3}$. 
Let 
\begin{equation} \label{eqn:Hdef}
H(z) = \exp(\h(z))
\end{equation}
denote the associated map to the diagonal
subgroup of $\SL_3\R$.  Then we obtain a parameterization of the
component of $x_1 x_2 x_3 = c$ in the positive octant by the orbit map
\begin{equation}\label{eqn:norm-titeica-def}
T(z) = H(z) \cdot \frac{1}{c^{1/3}}(1,1,1) 
= \frac{1}{c^{1/3}} (e^{2 \Re(z)}, e^{2 \Re (z/\omega)}, e^{2 \Re(z/\omega^2)}).
\end{equation}

Taking $c = \frac{1}{3 \sqrt{3}}$ gives a surface with mean curvature
$H = -1$.  In terms of the parameterization above, the Blaschke metric
of this affine sphere is $e^u |dz|^2 = 2 |dz|^2$ and the Pick differential is
$C = 2dz^3$.  We call this parameterized affine sphere $T$ the
\emph{normalized \titeica surface}.

Let $F_T$ denote the frame field of the normalized \titeica surface $T$.
By homogeneity under the action of the diagonal group we have
$$ F_T(z) = H(z) \cdot F_T(0).$$

Because they correspond to complete, flat Blaschke metrics and
constant, nonzero Pick differentials, the \titeica surfaces are
natural comparison objects for any class of affine spheres with small
(or decaying) Blaschke curvature. Later we will see that affine
spheres corresponding to polynomial cubic differentials on $\C$ have
this behavior at infinity.

Because it is asymptotic to the boundary of the positive octant in
$\R^3$, projecting the normalized \titeica surface to $\RP^2$ gives an
open set $\P(T)$ that is the interior of a triangle whose vertices
correspond to the coordinate axes; we denote this triangle by
$\triangle_0$, and its vertices by $v_{100}$, $v_{010}$, and $v_{001}$
according to their homogeneous coordinates.  We use analogous notation
for the three edges of the triangle, calling them $e_{011}, e_{101},
e_{110}$ according to the homogeneous coordinates of a point in the
interior of the edge.

\subsection{Affine spheres from planar data}

Anticipating the construction in Section
\ref{sec:polynomials-to-polygons} of an affine sphere with prescribed
polynomial Pick differential, we finish this section by observing that
the results on affine spheres described thus far reduce this problem
to one of constructing suitable planar data (i.e.~a conformal metric
on $\C$ satisfying Wang's equation).

\begin{prop}\label{prop:conformal data enough}
Let $C(z)dz^3$ be a holomorphic cubic differential on the complex
plane $\C$.  Let $h = e^u|dz|^2$ solve the Wang
equation~\eqref{eqn:wang} and suppose that $e^u |dz|^2$ is a complete
metric on $\C$.  Then integration of the form~\eqref{eqn:structure}
gives the complexified frame field $F : \C \to \GL_3\C$ of an
affine sphere $f : \C \to \R^3$ with Blaschke metric $h$ and Pick differential $C$.
The map $f$ is a proper embedding, and its image is asymptotic to the
boundary of the cone over a convex domain in $\RP^2$.
\end{prop}

\begin{proof}
Integrability of the structure equations \eqref{eqn:structure} follows
because we assumed that $C$ is holomorphic and that $h,C$ satisfy
\eqref{eqn:wang}. This guarantees an affine spherical immersion $\C
\to \R^3$.  By hypothesis the Blaschke metric of this immersion is
complete, hence by Theorem \ref{thm:completeness-equivalence} it is a
proper embedding.
The last sentence of the Cheng-Yau theorem (\ref{thm:cheng-yau}) then
completes the proof.
\end{proof}

Of course we have already seen one instance of this proposition: The \titeica
surface is the result of integrating $e^u = 2 |dz|^2$ and $C = 2
dz^3$; it is easy to check that this pair satisfies \eqref{eqn:wang}.

\section{The coupled vortex equation}
\label{sec:vortex}

\subsection{Existence theorem}
\label{subsec:existence}

In this section we study the problem of prescribing a certain
relationship between the curvature of a conformal metric on $\C$ and
the norm of a holomorphic differential.  The Wang equation
\eqref{eqn:wang} is one example of the class of equations we consider
(and the only instance that is used in subsequent sections), but
in this section we consider a more general class of equations to
which our techniques naturally apply.

We begin with the following existence result:
\begin{thm}
\label{thm:existence}
Let $\phi = \phi(z) dz^k$ be a holomorphic differential of order $k$
on $\C$, with $\phi(z)$ a polynomial that is not identically zero.
Then there exists a complete, smooth, nonpositively curved conformal metric
$\sigma = \sigma(z) |dz|^2$ on $\C$ satisfying
\begin{equation}
\label{eqn:vortex}
K_\sigma = (-1 + |\phi|_\sigma^2)
\end{equation}
where
\begin{itemize}
\item $K_\sigma(z) = -(2\sigma(z))^{-1} \Delta (\log \sigma(z))$ is the
Gaussian curvature, and
\item $|\phi|_\sigma(z) = \sigma(z)^{-k/2} |\phi(z)|$ is the pointwise norm
with respect to $\sigma$.
\end{itemize}
Furthermore this metric can be chosen to satisfy $\sigma \geq
|\phi|^{2/k}$, with equality at some point if and only if $\phi(z)$ is
constant and $\sigma = |\phi|^{2/k}$.
\end{thm}

Note that up to scaling of the holomorphic differential by a constant
factor, the case $k=3$ of \eqref{eqn:vortex} is Wang's equation.

Before proceeding with the proof, we briefly explain a connection
(also noted in \cite[Sec.~3.1]{dunajski}) between \eqref{eqn:vortex}
and the \emph{vortex equations} from gauge theory.  These equations
were introduced in the Ginzburg-Landau model of superconductivity
\cite{ginzburg-landau} and subsequently generalized and extensively
studied in relation to Yang-Mills-Higgs theory (see
e.g.~\cite{jaffe-taubes} \cite{bradlow:special-metrics}
\cite{garcia-prada:direct-existence}
\cite{witten:superconductors}). In one formulation, the vortex
equations on a Riemann surface reduce to a single equation for a
Hermitian metric of a holomorphic line bundle; specifically, one fixes
a holomorphic section $\phi$ of the bundle and asks for a Hermitian
metric whose curvature differs from the pointwise norm of $\phi$ by a
constant.  Since the curvature of the Hermitian metric is an
endomorphism-valued $2$-form, it is first contracted with the K\"ahler
form of fixed background metric on the surface to define a scalar
equation.

In our situation we have a holomorphic section $\phi$ of the
$k^{\mathrm{th}}$ tensor power of the canonical bundle, and a
Hermitian metric on this bundle is simply the tensor power of a
conformal metric $\sigma$ on the Riemann surface itself.  If we use
the K\"ahler form of $\sigma$ instead of a fixed background metric,
the associated vortex equation becomes \eqref{eqn:vortex}.  Thus, our
equation involves an additional ``coupling'' between the curvature and
norm functions that does not appear in the classical vortex equation
setting, and we refer to \eqref{eqn:vortex} as the \emph{coupled
  vortex equation}.

\begin{proof}[Proof of Theorem \ref{thm:existence}.]
Writing $\sigma(z) = \exp(u(z))$ the equation \eqref{eqn:vortex} from
the theorem becomes
\begin{equation}
\label{eqn:vortex-log}
\Delta u = 2 e^u - 2 |\phi|^2 e^{-(k-1)u}
\end{equation}
and it is this form we use in the proof.  We denote
the right hand side of the equation above by $F(z,u)$.

We apply the method of sub-solutions and super-solutions for complete
noncompact manifolds to \eqref{eqn:vortex-log}: For
equations of the form $\Delta u = F(z,u)$ where $\partial F/\partial u
> 0$ it suffices to construct a pair of continuous functions on $\C$
which weakly satisfy
\begin{equation*}
\begin{split}
\Delta u_+ &\leq F(z,u_+),\\
\Delta u_- &\geq F(z,u_-).
\end{split}
\end{equation*}
and where $u_- \leq u_+$.  Then the method
(cf.~\cite[Thm.~9]{wan:thesis}) gives a smooth solution $u$ on $\C$
satisfying $u_- \leq u \leq u_+$.

In our case, both $u_-$ and $u_+$ will be slight modifications of the
function
$$u_\phi := \frac{1}{k}\log |\phi|^2,$$
which corresponds to the conformal metric $|\phi|^{2/k}$.  Define
\begin{equation*}
\begin{split}
u_+ &= \frac{1}{k} \log \left ( a + |\phi|^2 \right ),\\
u_- &= \begin{cases}
u_\phi & \text{ if }|z|>d,\\
\max(u_\phi, h_{2d} ) & \text{ otherwise,}
\end{cases}
\end{split}
\end{equation*}
where $a$ and $d$ are positive constants whose values will be chosen later and 
$$h_R(z) = 2 \log \left ( \frac{2 R}{R^2 - |z|^2} \right )$$
is the logarithm of the Poincar\'e metric density on the disk $\{ |z| < R \}$
of constant curvature $-1$.  In the exceptional case that $\phi$ is a
constant function we modify the definition above and take $u_- = u_\phi$.

We must verify that for suitable choices of $a$ and $d$ these
functions satisfy the required conditions.  First, differentiating the
expression for $u_+$ above we find that $\Delta u_+ \leq F(z,u_+)$ is
equivalent to
$$ 2 |\phi_z|^2 \leq k (a + |\phi|^2)^{\frac{1}{k} + 1}.$$
Using that $\phi$ is a polynomial and comparing the rates of growth of
the two sides, we find this inequality is always satisfied for $|z|$
sufficiently large.  Thus, we may choose $a$ large enough 
the inequality holds for all $z$.

Turning to the function $u_-$, 
if $\phi$ is nonconstant then we must first check the continuity 
of $u_-$ on $\{ |z| =
d \}$ and at the zeros of $\phi$.  We assume that $d$ is large enough so
that all points $z$ where $|\phi(z)| \leq 1$ lie in $\{ |z| < d \}$, and
also that $d > 4/3$ so $h_{2d}(z) < 0$ for $|z| \leq d$.  This
means that $u_- = u_\phi$ a neighborhood of $|z|=d$, and $u_-$ is
continuous there.  Also, since $h_{2d}$ is continuous on $|z| < d$ and
bounded below, the function $\max(u_\phi,h_{2d})$ is continuous at the
zeros of $\phi$.

The function $u_\phi$ is subharmonic and $F(z,u_\phi) = 0$ on the
complement of the zeros of $\phi$, while $h_{R}$ satisfies $ \Delta
h_{R} = 2 e^{h_{R}} \geq F(z,h_{R})$, thus, in a neighborhood of any
point, the function $u_{-}$ is either a subsolution of \eqref{eqn:vortex-log} or a
supremum of two subsolutions.  Hence $u_-$ is itself a subsolution.

Finally we must compare $u_+$ and $u_-$.  It is immediate from the
definition that $u_+ \geq u_\phi$, and taking $a > 1$ we also have $u_+
\geq h_{2d}$ on $\{ |z| < d\}$ because then
$$ \inf u_+ \geq \frac{1}{k} \log a > 0 > \sup_{|z| < d} h_{2d}.$$
It follows that $u_+ \geq u_-$, and so the sub/supersolution method yields
a $C^\infty$ solution $u$
and thus a corresponding metric $\sigma = e^u |dz|^2$.

By construction $u \geq u_- \geq u_\phi$, implying $\sigma = e^u \geq
|\phi|^{2/k}$.  Since the metric $|\phi|^{2/k}$ is complete, 
the metric $\sigma$
is also complete.  The condition $u \geq u_\phi$ also gives $e^u \geq
2 |\phi|^2 e^{-(k-1)u}$ and thus by \eqref{eqn:vortex-log}
we have $\Delta u \geq 0$, which implies that the metric $\sigma$ is
nonpositively curved.

Since $u \geq u_\phi$ with $u$ a solution and $u_\phi$ a subsolution of the
equation \eqref{eqn:vortex-log}, the strong comparison
principle (e.g.~\cite[Thm.~2.3.1]{jost:pde}) implies that on any
domain where $u_\phi$ is continuous up to the boundary we have either $u
> u_\phi$ or $u \equiv u_\phi$.  Thus if $u(z_0) = u_\phi(z_0)$ for some $z_0$
(which therefore satisfies $\phi(z_0) \neq 0$), then $u$ and $u_\phi$ agree
in the complement of the zero set of $\phi$.  Since $u_\phi$ is unbounded near
these zeros, while $u$ extends smoothly over them, this means $\phi$ has
no zeros at all, i.e.~$\phi$ is constant, and thus $u = u_\phi$ everywhere.
\end{proof}

By its construction from super- and sub-solutions, the proof above
also gives the following basic bounds on the solution $u$:

\begin{cor}[Coarse bounds]
\label{cor:coarse-bound}
Let $e^u |dz|^2$ be the solution of \eqref{eqn:vortex} constructed
in the proof of Theorem \ref{thm:existence}.  Then there exist
constants $m,M$ depending continuously on the coefficients of the
polynomial $\phi$ such that
$$ u(z) \geq \max(-m, u_\phi(z)) $$
and
$$ u(z) \leq u_\phi(z) + \frac{M}{|\phi(z)|^2}.$$
In particular we have $u(z) - u_\phi(z) \to 0$ as $z \to \infty$.
\end{cor}

\begin{proof}
We assume $\phi$ is not constant, since otherwise $u = u_\phi$ and all
of the bounds are trivial.

The subsolution $u_-$ from Theorem
\ref{thm:existence} satisfies
$$ u_-(z) \geq \max(\inf h_{2d}, u_\phi(z)),$$
since $\inf h_{2d}$ is achieved at $z=0$ (and in particular within $|z|<d$).
Taking $m = -\inf h_{2d} = 2 \log d$ gives the lower bound.

On the other hand we have
$$ u_+  = u_\phi + \frac{1}{k} \log \left ( 1 + \frac{a}{|\phi|^2} \right )
\leq u_\phi + \frac{a}{k |\phi|^2}$$
where $a$ is a (again positive) constant that depends on the coefficients of $\phi$.
Taking $M = a/k$ gives the desired upper bound.

Finally, as $z \to \infty$ we have $|\phi(z)| \to \infty $ and thus
these bounds give $u(z) - u_\phi(z) \to 0$.
\end{proof}

We will improve these coarse bounds in Theorem~\ref{thm:exponential-bound}.

\subsection{Uniqueness}
\label{subsec:uniqueness}

Complementing the existence theorem above, we have:

\begin{thm}
\label{thm:uniqueness}
For any polynomial holomorphic differential $\phi$ of degree $k$,
there is a unique complete and nonpositively curved solution of
\eqref{eqn:vortex}.
\end{thm}

\begin{proof}
Suppose that $u$ and $w$ are log-densities of solutions to
\eqref{eqn:vortex}, with $u$ complete and nonpositively
curved.  Note that both metrics have curvature bounded below by $-1$.
We will show that $w \leq u$, following the method of
\cite[Sec.~5]{wan:thesis}.

Let $\eta = w - u$.
In terms of the Laplace-Beltrami operator $\Delta_u = e^{-u}\Delta$ of
$u$ and the pointwise norm $|\phi|_u = |\phi|e^{-ku/2}$, the fact that both
$w$ and $u$ are solutions implies
$$ \Delta_u \eta = 2e^\eta - 2|\phi|_u^2 e^{-(k-1)\eta} - 2 +  2|\phi|_u^2.$$
By \eqref{eqn:vortex} the nonpositive curvature of $u$
implies that $|\phi|_u \leq 1$, giving
$$ \Delta_u \eta \geq 2e^\eta -  2e^{-(k-1)\eta} - 2.$$
By a result \cite{cheng-yau} of Cheng and Yau, this differential inequality implies
that $\eta$ is bounded above: Since $u$ is complete and has a lower curvature
bound, applying Theorem 8 of \cite{cheng-yau} with $f(t) = 2e^t -
2e^{-(k-1)t}-2$ and $g(t) = e^t$ gives 
$$\sup \eta = \bar{\eta} < \infty.$$
Applying the generalized maximum principle \cite{omori} \cite{yau} to
$\eta$ we find that there is a sequence $z_k \in \C$ such that
\begin{equation*}
\begin{split}
\lim_{k \to \infty} \eta(z_k) &= \bar{\eta}\\
\limsup_{k \to \infty} \Delta_u \eta(z_k) &\leq 0.
\end{split}
\end{equation*}
Using the bound $|\phi|_u \leq 1$ and passing to a subsequence we can
assume $|\phi|_u(z_k)^2$ converges, say to $\lambda \in [0,1]$.  Then
substituting the expression for $\Delta_u \eta$ into the inequality
above we find
$$(e^{\bar{\eta}}
- 1) - \lambda(e^{-(k-1)\bar{\eta}} - 1) \leq 0.$$
This gives $\bar{\eta} \leq 0$, or equivalently, $w \leq u$.

If $u$ and $w$ are \emph{both} complete and nonpositively curved
we can apply this argument with their roles reversed to conclude
$u=w$, hence there is at most one solution with these properties.  By Theorem
\ref{thm:existence} there is at least one such solution.
\end{proof}

\subsection{Continuity}
\label{subsec:continuity}

For later applications it will be important to know that the metric
satisfying \eqref{eqn:vortex} depends continuously on the holomorphic
differential $\phi$.

Let $\D^k_d \simeq \C^{d}$ denote the space of holomorphic
differentials of the form $\phi(z) dz^k$, with $\phi$ a monic
polynomial of degree $d$.

Let $\CM$ denote the set of smooth, strictly positive conformal metrics
on $\C$, which we identify with $C^\infty(\C)$ using log-density
functions (i.e.~the metric $e^u |dz|^2$ is represented by the function
$u$).

\begin{thm}[Global $C^0$ and local $C^1$ continuity]
\label{thm:continuity}
For any $k,d \in \Z^{\geq 0}$, considering the unique complete and
nonpositively curved solution of \eqref{eqn:vortex} as a function of
$\phi$ defines an embedding
$$\D^k_d \into \CM$$
which is continuous in the uniform topology, i.e.~as a map into
$C^0(\C)$.  Furthermore, restricting to any compact set $K \subset \C$
defines a continuous map into $C^1(K)$.
\end{thm}

\begin{proof}
A monic polynomial is determined by its absolute value, so it is
immediate from \eqref{eqn:vortex} that the map is injective.

Consider a pair $\phi, \psi \in \D^k_d$.  As monic polynomials of the
same fixed degree, their difference is small in comparison to either one.
Making this precise, for any $\epsilon > 0$ we can for example ensure
that
$$ \left ||\psi(z)|^2 - |\phi(z)|^2 \right | \leq \epsilon (1 +
|\phi(z)|^2) \text{ for all } z \in \C,$$
just by requiring the coefficients of $\phi$ and $\psi$ to be
sufficiently close.

Let $u,v$ be the solutions to \eqref{eqn:vortex-log} corresponding to
$\phi$ and $\psi$, respectively, and define $\eta = v - u$.  To
establish continuity of the map $\D^k_d \to \CM$ it suffices to bound
the relevant norm of $\eta$ in terms of $\epsilon$.

Since $|\phi|_u \leq 1$ and $u$ is
bounded below (by Corollary \ref{cor:coarse-bound}), multiplying
previous inequality above by $e^{-ku}$ we find
\begin{equation}
\label{eqn:polynomial-norms-close}
| |\phi|_u^2 - |\psi|_u^2 | \leq C \epsilon,
\end{equation}
for some constant $C$ and for all $\psi$ in some neighborhood of
$\phi$.  In particular $|\psi|_u$ is bounded.

Calculating as in the proof of Theorem \ref{thm:uniqueness} we find
that $\eta$ satisfies
\begin{equation}
\label{eqn:continuity-error-equation}
\begin{split}
\Delta_u \eta &= 2e^\eta - 2|\psi|_u^2 e^{-(k-1)\eta} - 2 + 2|\phi|_u^2\\
&= 2(e^\eta - 1) -  2|\psi|_u^2 (e^{-(k-1)\eta} - 1) + 2(|\phi|_u^2  -
|\psi|_u^2).
\end{split}
\end{equation}
By \cite[Thm.~8]{cheng-yau} the associated differential inequality
$$\Delta_u \eta \geq 2e^\eta - M e^{-(k-1)\eta} - 2,$$
where $M = 2 \sup |\psi|_u^2 < \infty$, implies 
as in Theorem~\ref{thm:uniqueness} that $\eta$ is bounded
above.  Applying the generalized maximum principle and passing to a
suitable subsequence gives $\{z_k\}$ with
\begin{equation*}
\lim_{k \to \infty} \eta(z_k) = \bar{\eta}, \;\;
\lim_{k \to \infty} |\phi|_u^2(z_k) = \lambda, \;\; \lim_{k \to \infty} |\psi|_u^2(z_k) = \mu
\end{equation*}
and
$$  (e^{\bar{\eta}} - 1) -  \mu (e^{-(k-1)\bar{\eta}} - 1) +
(\lambda-\mu) \leq 0.$$
Since \eqref{eqn:polynomial-norms-close} gives $|\lambda - \mu| \leq C
\epsilon$, the inequality above implies
$$ \bar{\eta} \leq \log(1 + C \epsilon) \leq C \epsilon. $$
Repeating this argument with the roles of $\phi$ and $\psi$ reversed
we find that
$$\sup_\C |\eta| \leq C' \epsilon,$$ and global $C^0$ continuity follows.

Substituting this bound on $|\eta|$ and the bound on $| |\phi|_u^2 -
|\psi|_u^2|$ from \eqref{eqn:polynomial-norms-close} into the right
hand side of \eqref{eqn:continuity-error-equation}, we also find a
uniform bound on the $u$-Laplacian,
$$\sup_\C |\Delta_u \eta| \leq C'' \epsilon.$$

Local $C^1$ continuity now follows by standard estimates for the
Laplace equation: Given a compact set $K \subset \C$, fix an open disk
$\{ |z| < R \}$ containing $K$.  Corollary \ref{cor:coarse-bound}
provides an upper bound on $\sup_{|z| < R} u$ depending on $R$, so we
get a bound on the flat Laplacian $\Delta = e^u \Delta_u$ of the form
$$\sup_{|z|<R} | \Delta \eta| = C(R) \: \epsilon.$$
Since $|\eta|$ and $|\Delta u|$ are each bounded by a fixed multiple
of $\epsilon$ on $\{ |z| < R\}$, standard interior gradient estimates
(e.g.~\cite[Thm.~3.9]{gilbarg-trudinger}) give a proportional $C^1$
bound on $\eta$ in the compact subset $K$.  This establishes the local
$C^1$ continuity of the map $\D^k_d \to \CM$.
\end{proof}

Restricting attention to monic polynomials in the theorem above is a
convenience that ensures injectivity of the map to $\CM$.  However,
continuity holds in general. Since \eqref{eqn:vortex} is invariant by
holomorphic automorphisms, we can pull back by an automorphism of the form $z
\to \lambda z$ to make an arbitrary polynomial differential monic, and
we conclude:

\begin{cor}
The global $C^0$ and local $C^1$ continuity statements of Theorem
\ref{thm:continuity} also apply to the space of polynomial differentials of fixed
degree with arbitrary nonzero leading coefficient.
\end{cor}

\subsection{Estimates}
\label{subsec:estimates}

Our final goal in this section is to compare the solution $e^u |dz|^2$ of 
equation \eqref{eqn:vortex} to the conformal metric $|\phi|^{2/k}$ by
studying the difference of their logarithmic densities, the ``error function''
$$ u - u_\phi = u - \frac{1}{k} \log |\phi|^2. $$

We have already derived coarse bounds for this difference in Corollary
\ref{cor:coarse-bound}.  We begin by reinterpreting these in terms of
$|\phi|^{2/k}$ metric geometry:

\begin{cor}[Coarse bound, intrinsic version]
\label{cor:coarse-bound-intrinsic}
Let $\phi$ and $u$ be as above, and suppose $k>1$.  There exist
constants $A',R'$ and an exponent $\alpha > 1$ with the following
property: If the $|\phi|^{2/k}$-distance from $p$ to the zero set of
$\phi$ is $r > R'$, then
$$ 0 \leq u(p) - u_\phi(p) \leq A' r^{-\alpha}.$$
\end{cor}

We note that for the purposes of the next theorem, a key feature of
this coarse bound is that it yields an integrable function of $r$.

\begin{proof}
As before the constant case is trivial and we assume degree $d>0$.
Outside of a large open disk $D$ containing the zeros of $\phi$, the
polynomial $\phi(z)$ is comparable to $z^d$ by uniform multiplicative
constants.  Thus the $|\phi|^{2/k}$-distance $r=r(p)$ from a point $p$
outside $D$ to the zero set of $\phi$ is bounded above by a fixed
multiple of the $|z|^{2d/k} |dz|^2$-distance from $p$ to the origin.
The latter distance can be explicitly calculated, giving
$$r < C |p|^{(d+k)/k}.$$
Since $|\phi|$ is also bounded below by a multiple of $|z|^d$, this implies
$$ |\phi(p)| > C' |p|^d > C'' r^{dk/(d+k)}$$
and applying Corollary \ref{cor:coarse-bound} we have
$$ u(p) - u_\phi(p) < \frac{M}{|\phi(p)|^2} < \frac{A'}{r^{2dk/(d+k)}}$$
for $p$ outside $D$, with $A'$ determined by $M$ and $C''$.  Let
$\alpha = 2dk /(d+k)$, and note that $\alpha > 1$ for $d,k$ integers and
$k>1$.  Fix $R'$ large enough so that $r > R'$ implies
that $p$ is outside $D$.  Then the statement follows for these $A',R',\alpha$.
\end{proof}

Our next goal is to show that the error $|u - u_\phi|$ is not just
bounded at infinity, but exponentially small as a function of $r$.

\begin{thm}[Exponentially small error]
\label{thm:exponential-bound}
Let $\phi$ and $u$ be as above.  Then there exist
constants $A$ and $R$ with the following property: If the
$|\phi|^{2/k}$-distance from $p$ to the zero set of $\phi$ is $r > R$,
then
$$ u(p) - u_\phi(p) \leq A \frac{\exp(-\sqrt{2k} \: r)}{\sqrt{r}}.$$
\end{thm}

In preparation for the proof, we introduce a convenient change of
coordinates: Let $w$ be a local coordinate in which $|\phi| = |dw|^k$.
We continue to use $u$ to denote the log-density of $\sigma$, but now
considered relative to $|dw|^2$, i.e.
$$ \sigma = e^u |dw|^2.$$
so that $u$ satisfies the equation
\begin{equation}
\label{eqn:vortex-in-natural-coord}
\Delta u = 2e^u - 2e^{-(k-1)u},
\end{equation}
with $\Delta = 4\frac{\partial^2}{\partial w \partial \bar{w}}$ now
denoting (here and in the rest of this section) the flat Laplacian
with respect to $w$.

In this coordinate system $u_\phi \equiv 0$, hence our goal is to show
that $u$ itself is exponentially small.  More precisely, defining
\begin{equation}
\label{eqn:h-decay}
\eps(t) = \frac{\exp(-\sqrt{2k} \: t)}{\sqrt{t}},
\end{equation}
we must show that $u = O(\eps(r))$.

Noting that the linearization of equation \eqref{eqn:vortex-in-natural-coord}
at $u=0$ is $\Delta u = 2 k u$, we first consider the Dirichlet
problem for this linear equation in the upper half-plane $\H \subset
\C$.  Write $w = x + i y$ with $x,y \in \R$.

\begin{lem}
\label{lem:linearization-solution}
Suppose $g \in C^0(\R) \cap L^1(\R)$ and $g \geq 0$.
Then there exists $h \in C^\infty(\H)$ extending continuously to
$\R$ that is a solution of the Dirichlet problem
$$
\Delta h = 2 k h, \;\;  \left .h \right |_\R = g,\\
$$
and which satisfies
\begin{equation*}
\begin{split}
0 &\leq h \leq \sup g,\\
h &= O\left (\|g\|_1 \eps(y) \right) \text{as } y \to \infty,
\end{split}
\end{equation*}
where the implicit constants in the second estimate are independent of $g$.
\end{lem}

\begin{proof}

Throughout this proof we write $w = x + i y$, where $x,y \in \R$.  A
solution of the Dirichlet problem can be constructed by convolution
\begin{equation}
\label{eqn:dirichlet}
 h(w) = \int_{\R} \frac{\partial G(w,\xi)}{\partial \Im(\xi)}
g(\xi) \: d\xi,
\end{equation}
where $G$ is the Green's function $G$ for the positive operator
$-(\Delta - 2k)$, i.e.~$(\Delta_{\xi} - 2k)G(w,\xi) = -\delta(w)$
and $G(w,\xi) = 0$ for $\xi \in \R$.
Then $G$ and its normal derivative along $\R$ are given \cite[Formula 7.3.2-3]{polyanin-zaitsev:handbook}  in terms of the modified Bessel
functions:
\begin{equation*}
\begin{split}
G(w,\xi) &= \frac{1}{2\pi} \left ( K_0
( \sqrt{2k} \: |w-\xi| ) - K_0(
\sqrt{2k} \: |w-\bar{\xi}| ) \right )\\
\frac{\partial G(w,\xi)}{\partial \Im(\xi)} &= \frac{\sqrt{2k} \: y}{\pi |w-\xi|}
K_1(\sqrt{2k} |w-\xi|) \;  \text{ for } \xi \in \R\\
\end{split}
\end{equation*}
Since $|w - \xi| \geq y$ for $\xi \in \R$, and the Bessel function
satisfies $K_1(t) = O(\eps(t))$ as $t \to \infty$
\cite[Formula~9.7.2]{abramowitz-stegun}, we therefore have

$$ \sup_{\xi \in \R} \frac{\partial G(w,\xi)}{\partial \Im(\xi)} = O
\left( \eps(y) \right ).$$
Substituting this into \eqref{eqn:dirichlet} gives $h = O(\|g\|_1
\eps(y))$ as required.

The other bounds on $h$ are immediate from \eqref{eqn:dirichlet}:
Since $K_1 > 0$ and $g \geq 0$, we have $h \geq 0$.  Since $G$ is a
fundamental solution we have, for any $w \in \H$,
$$ \int_{\R} \frac{\partial G(w,\xi)}{\partial \Im(\xi)}
 \: d\xi = 1 - 2 k \int_{\H} G(w,\xi) |d \xi|^2 < 1$$
and thus, again noting that $K_1 > 0$ and $g \geq 0$, we have that 
$$h = \int_{\R} \frac{\partial G(w,\xi)}{\partial \Im(\xi)}
g(\xi) \: d\xi\leq \sup g \int_{\R} \frac{\partial G(w,\xi)}{\partial
  \Im(\xi)} \: d\xi \leq \sup g. $$
\end{proof}

Next we use the solution of the linearization constructed above to get
a supersolution for the quadratic approximation of
\eqref{eqn:vortex-in-natural-coord} on $\H$.  Note that the right hand
side of this equation is
$$f(u) := 2e^u - 2e^{-(k-1)u} = 2 k\,u - k(k-2)\,u^2 + O(u^3).$$

\begin{lem}
\label{lem:quadratic-supersolution}
Suppose $g \in C^0(\R) \cap L^1(\R)$ and that $0 \leq g \leq
\frac{1}{(k-2)}$.  Then there exists a function $v \in C^\infty(\H)$
extending continuously to $\R$ such that
\begin{equation*}
\begin{split}
\left . v \right |_\R &\geq g,\\
\Delta v &\leq 2 k\,v - k(k-2)\,v^2,\\
v &= O\left (\|g\|_1 \eps(y) \right ) \text{ as } y \to \infty,
\end{split}
\end{equation*}

\end{lem}

\begin{proof}
First, consider an arbitrary function $h$ satisfying $\Delta h =
2k\,h$.  Then $v = h - \frac{(k-2)}{2} h^2$ satisfies
$$ \Delta v - 2k\,v + k(k-2)\,v^2 = -k(k-2)^2\,h^3 + \frac{1}{4}k(k-2)^3\,h^4 - (k-2)\,|\nabla h|^2.$$
Since $|\nabla h| \geq 0$, we find that the right hand side is
negative if $\sup h < \frac{4}{k-2}$.

Now let $h$ be the solution of $\Delta h = 2k\,h$ given by Lemma
\ref{lem:linearization-solution} with boundary values $\left . h
\right|_\R = 2g$.  Since $h \leq 2\sup g < \frac{2}{k-2} $, we find
that $v = h - \frac{(k-2)}{2} h^2$ satisfies $0 < v < h$ and the calculation
above shows
$$\Delta v \leq 2 k\,v - k(k-2)\,v^2.$$
Since $0 < v < h$, it is immediate from Lemma
\ref{lem:linearization-solution} that $v = O(\|g\|_1 \eps(t))$.
Finally, we must verify that $v \geq g$ on $\R$.  This follows because
$\left .h \right |_\R = 2g$ and $v = q(h)$ where the polynomial $q(t)
= t - \frac{(k-2)}{2}t^2$ satisfies $q(t) > \frac{1}{2}t$ on the
interval $[0,\frac{1}{k-2}]$ which contains the range of $g$.
\end{proof}

\begin{proof}[Proof of Theorem \ref{thm:exponential-bound}.]
First we find a suitable region and coordinate system in which to apply
the previous lemmas.  By Corollary \ref{cor:coarse-bound-intrinsic}
there is a compact set $K$ in the plane outside of which $u - u_\phi <
\frac{1}{2(k-2)}$.  By Proposition
\ref{prop:k-differential-half-planes}, any point $p$ sufficiently far
from $K$ lies in a $|\phi|$-upper-half-plane $(U,w)$ with $U \cap K =
\emptyset$ and with $y(p) = \im(w(p)) \geq r(p) - C$ for a constant
$C$ independent of $p$.  For the rest of the proof we work in this coordinate $w$, identifying
$U$ with $\H$ and writing $\sigma = e^u |dw|^2$ and $|\phi|^{2/k} =
e^{u_\phi} |dw|^2$.  We therefore have $u_\phi(w) \equiv 0$ and  $0
\leq u(w) < \frac{1}{2(k-2)}$.

Since Proposition \ref{prop:k-differential-half-planes} gives $r(w) \geq c |w|$ for $w \in \R$ with $|w|$
sufficiently large, it follows from Corollary \ref{cor:coarse-bound-intrinsic} 
that $u$ is integrable on $\R$.  Moreover $u$ is everywhere less than
$\tfrac{1}{2(k-2)}$, so we in fact have a bound on the $L^1$ norm of
$\left . u \right |_{\R}$ that depends on $\phi$ but which is
independent of $p$.

Let $v$ be the function on $\H$ given by Lemma
\ref{lem:quadratic-supersolution} for $g = \left . u \right |_\R$.
Since $v = O(\epsilon(y))$ and $y(p) \geq r(p) - C$, the theorem will
follow if we show $u \leq v$, or equivalently that the function $\eta
= u-v$ is nowhere positive.  Note that $\eta$ is smooth on $\H$ and
continuous on the closure $\bar{\H}$.

Suppose for contradiction that $\eta$ is positive at some point, so
the closed set $Q = \eta^{-1}([\epsilon,\infty)) \subset \bar{\H}$ is
nonempty for some $\epsilon > 0$.  Lemma
\ref{lem:quadratic-supersolution} gives $\eta < 0$ on $\partial \H$,
hence $Q \subset \H$.  The same lemma and Corollary~\ref{cor:coarse-bound} respectively
show $v \to 0$ and $u \to 0$ as $z \to \infty$, hence $\eta \to 0$ as
$z \to \infty$, and $Q$ is compact.  Therefore $\eta$ has a positive
maximum at some point in $Q$.

Recall that we set $f(u)= 2e^u - 2e^{-(k-1)u}$, that $u$ satisfies
$\Delta u = f(u)$, and that $v$ satisfies $v \geq 0$ and $\Delta v =
2k v - k(k-2)\,v^2 \leq f(v)$.  Therefore
\begin{align*}
\Delta \eta &\geq 2e^u - 2e^{-(k-1)u} - 2e^v + 2e^{-(k-1)v}\\
&=2e^v(e^{\eta} -1) - 2e^{-(k-1)v}(e^{(-k-1)\eta} -1).
\end{align*}
At a maximum we have $0 \geq \Delta \eta$, which in combination with
the inequality above gives
$$
(e^\eta -1) \leq e^{-(k-2)v}(e^{-(k-1)\eta} -1).
$$
This shows $\eta \leq 0$ at any maximum, which is the desired contradiction.
\end{proof}

The $C^0$ bound of the previous theorem is easily improved to a $C^1$ bound:

\begin{cor}
\label{cor:c1-bound}
Let $\phi$ and $u$ be as above, and let $|\nabla f|_{\phi}$ denote the
norm of the gradient of a function $f$ with respect to the
$|\phi|^{2/k}$-metric.  Let $r$ denote the $|\phi|^{2/k}$-distance
from a point $p$ to the zero set of $\phi$.  If $r > (R+1)$, where $R$
is the constant from Theorem \ref{thm:exponential-bound}, then
$$ | \nabla(u - u_\phi) |_{\phi}(p) \leq C \exp(-\sqrt{2k} \: r) / \sqrt{r}).$$
\end{cor}

\begin{proof}
Working as above in $\phi$-natural coordinates, where $u_\phi = 0$ and
the $\phi$-gradient becomes the Euclidean one, we simply require a
pointwise $C^1$ bound on the function $u$.  Since $u$ satisfies
\eqref{eqn:vortex-in-natural-coord} and
$$ e^u - e^{-(k-1)u} \leq C |u| \text{ for }|u| < 1,$$
the bound on $u$ from the previous theorem shows that $|\Delta u|$ is
also proportionally small throughout a disk of radius $1$ centered at
$p$.  Applying the standard interior gradient estimate for Poisson's
equation (see e.g.~\cite[Thm.~3.9]{gilbarg-trudinger}) to this disk then gives the desired
bound for the derivative of $u$ at its center.
\end{proof}

\section{From polynomials to polygons}
\label{sec:polynomials-to-polygons}

The results of the previous section give the following
existence theorem for affine spheres with prescribed polynomial Pick differential:
\begin{thm}
\label{thm:polynomial-sphere-existence}
For any polynomial cubic differential $C$ on the complex plane, there
exists a complete hyperbolic affine sphere in $\R^3$ that is
conformally equivalent to $\C$ and which has Pick differential $C$
with respect to some conformal parameterization.  This affine sphere
is uniquely determined by these properties, up to translation and the
action of $\SL_3\R$.
\end{thm}

\begin{proof}
Let $\sigma$ be the complete, nonpositively curved conformal metric
satisfying \eqref{eqn:vortex} for $\phi = \sqrt{2} C$, given by
Theorem \ref{thm:existence}.  By Proposition \ref{prop:conformal data
  enough} the pair $\sigma, C$ can then be integrated to an affine
spherical immersion $f : \C \to \R^3$ which by Theorem
\ref{thm:completeness-equivalence} is properly embedded.  Let $M$
denote its image.

Suppose $M'$ is another complete affine sphere conformally
parameterized by $\C$ which has Pick differential $C$.  Then the Blaschke
metric $\sigma'$ of $M'$ is another solution of \eqref{eqn:vortex} for
the same differential $\phi = \sqrt{2} C$ which is complete, and by
Theorem \ref{thm:npc}, nonpositively curved.  By Theorem
\ref{thm:uniqueness} we have $\sigma'=\sigma$.  Therefore, after
translating $M'$ so that it is centered at the origin, its
complexified frame satisfies the same structure equations
\eqref{eqn:structure} as that of $M$.  Hence the frames differ by a
fixed element $A \in \SL_3\R$, and $M'$ is the image of $M$ by the
composition of $A$ and a translation.
\end{proof}

Though we will not use the following result in the sequel, we note in
passing that uniqueness of the ``polynomial affine sphere'' considered
above can be shown even under weaker hypotheses; we can drop the
assumption of embeddedness and replace it with a curvature condition:

\begin{thm}
\label{thm:stronger-polynomial-sphere-uniqueness}
Suppose $f : \C \to \R^3$ is an affine spherical immersion with
polynomial Pick differential and nonpositively curved Blaschke metric.  Then
the image of $f$ is the complete affine sphere associated to $C$ by
Theorem \ref{thm:polynomial-sphere-existence}.
\end{thm}

\begin{proof}
By Theorem \ref{thm:completeness-equivalence}, it suffices to show
that the Blaschke metric of $f$ is complete, for then the hypotheses
of the uniqueness statement in Theorem
\ref{thm:polynomial-sphere-existence} are again satisfied.

From the intrinsic formulation of Wang's equation
\eqref{eqn:wang-intrinsic} we find that nonpositive curvature of the
Blaschke metric $h$ implies $|C|_h \leq \tfrac{1}{2}$.
Equivalently $h$ is bounded below by a constant
multiple of the conformal metric $|C(z)|^{2/3} |dz|^2$.  Since $C$ is
a polynomial, it follows that both of these metrics are bounded below
by a multiple of the Euclidean metric $|dz|^2$ on the complement of a
compact set, and hence both are complete.
\end{proof}

It would be interesting to know whether the curvature condition can
also be dropped, i.e. 
\begin{question}
Does there exist an affine spherical immersion $f : \C \to \R^3$ with
polynomial Pick differential whose Blaschke metric has positive curvature at
some point?
\end{question}
Of course it follows from the developments above that such an immersion
cannot be proper, and its Blaschke metric must be incomplete.

Having settled the basic existence and uniqueness results for
polynomial affine spheres, the goal of the rest of this section is to
show that the convex domains associated with these affine spheres are polygons.

Recall that $\star_d \subset \C$ is the union of $(d+3)$ evenly spaced
rays from the origin that includes $\R^+$.  We call a space homeomorphic to
$\star_d$ an \emph{open star}, and in such a space the homeomorphic images of the
rays are the \emph{edges}.

\begin{thm}
\label{thm:polygon}
Let $f : \C \to M \subset \R^3$ be a conformal parameterization of a
complete hyperbolic affine sphere whose Pick differential $C$ is a polynomial
of degree $d$.  Then $M$ is asymptotic to the cone over a convex
polygon $P$ with $d+3$ vertices.

Moreover, if $C$ is monic then the projectivization of $f(\star_d)$
gives an embedded open star in $P$ whose edges tend to the vertices.
\end{thm}

Before starting the proof we will introduce a key tool, the
\emph{osculation map}, and outline how the theorem will follow from an
analysis of the asymptotic behavior of this map using estimates from
the previous section.

Let $z$ be a natural coordinate for the Pick differential defined in a region
$U$.  (In this section we will \emph{not} be working with the global
coordinate in which $C$ is a polynomial, which in previous sections had also
been denoted $z$.)  Restricting the conformal parameterization $f$ to
$U$, we can consider it as a function of $z$, denoted $f(z)$.  Let
$F(z) = (f, f_z, f_{\bar{z}})$ be the associated complexified frame
field.  As before let $F_T(z)$ denote the frame field of the
normalized \titeica surface $T$.  Define the \emph{osculation map}
$\hat{F} : U \to \GL_3\R$ by
$$ \hat{F}(z) = F(z) F_T^{-1}(z).$$
Note that the value of this function lies in $\GL_3\R$ because both
frame values $F(z)$ and $F_T(z)$ lie in the same right coset of
$\GL_3\R$ within $\GL_3\C$ (as described in Section \ref{sec:affine}).

Evidently $\hat{F}$ is constant if and only if $f$ is itself a
\titeica surface, and more generally, left multiplication by
$\hat{F}(z_0)$ transforms the normalized \titeica surface to one which
has the same tangent plane and affine normal as $f$ at the point
$f(z_0)$.  In this sense $\hat{F}(z_0)$ represents the ``osculating''
\titeica surface of $f$ at $z_0$.

Recall the map $H(z) = \exp(\h(z))$ we used in
subsection~\ref{sec:affine} (equation \eqref{eqn:Hdef}) to
parameterize the \titeica surface.  Then a calculation with the affine
structure equations \eqref{eqn:structure} shows that the derivative
$\hat{F}^{-1} d\hat{F} \in \Omega^1(U,\gl_3\R)$ of the osculation is
given by
\begin{equation}
\begin{split}
\label{eqn:osculation-derivative}
\hat{F}^{-1} d\hat{F} &= \Ad_{F_T} \left ( F^{-1} dF - F_T^{-1} dF_T \right )\\
&= \Ad_{H(z)} \Theta(u(z))
\end{split}
\end{equation}
where 
\begin{equation}\label{eqn:frame-error}
\begin{split}
\Theta(u(z) ) &= \Ad_{F_T(0)} \left [ \;\; \begin{pmatrix}
0 & 0 & \frac{1}{2}e^u - 1 \\
0 & u_z  & 0 \\
0 & 2e^{-u} - 1 & 0 
\end{pmatrix} \right .
dz\\&\text{\hspace{1.83cm}}
+\left .
\begin{pmatrix}
0 & \frac{1}{2}e^u - 1 & 0 \\
0 & 0 & 2e^{-u} - 1 \\
0 & 0 & u_{\bar{z}}
\end{pmatrix}
d\bar{z} \;\;\right ]
\end{split}
\end{equation}
and $e^u |dz|^2$ is the Blaschke metric of $M$.

Notice that the estimates from the previous section show that
$\Theta(u(z))$ is rapidly decaying toward zero as the distance from
$z$ to the zeros of $C$ increases (since $F_T(0)$ is a constant matrix
and in these coordinates the functions 
$\frac{1}{2}e^u - 1$, $u_z$, and $u_{\bar{z}}$ are all
exponentially small).  Ignoring the conjugation by the diagonal matrix
$H(z)$ in \eqref{eqn:osculation-derivative} for a moment, this
suggests that $\hat{F}(z)$ should approach a constant as $z$ goes to
infinity---since its derivative is approaching zero---which would mean
that the affine sphere $f$ is asymptotic to a \titeica surface.

However, the function $H(z)$ is itself exponentially growing as a
function of $z$, with the precise rate of growth depending on the
direction.  Thus the actual asymptotic behavior of the osculation map
depends on the competition between frame growth (i.e.~$H$) and Blaschke
error decay (i.e.~$\Theta$) as described in the introduction.  In most
directions, the exponential decay of $\Theta$ is faster than the
growth of $H$, giving a well-defined limiting \titeica surface and
thus a portion of the projectivized image of $M$ that is modeled on a
triangle.  In exactly $2(d+3)$ \emph{unstable directions} there is an exact balance,
which allows the limit \titeica surface to shift.

Thus, using the osculation map, the proof Theorem \ref{thm:polygon}
splits into four steps:
\begin{enumerate}
\item \textbf{Finding stable limits:} By considering the osculation
map restricted to a ray in a standard half-plane, use the exponential
bounds from Theorem \ref{thm:exponential-bound} and Corollary
\ref{cor:c1-bound} to show the existence of a limit in any stable
direction (and locally constant as a function of the ray)

\item \textbf{Finding unipotent factors:} By considering the
osculation map restricted to an arc joining two rays on either side of
an unstable direction, use the same bounds to show that the ray limits on
either side differ by composition with a unipotent element of $\SL_3\R$.

\item \textbf{Finding triangle pieces:} Show that there is a ``vee''
(two edges of a triangle) in the boundary of the projectivization of
$M$ corresponding to each interval of stable directions.

\item \textbf{Assembling the polygon:} Use the geometry of the unipotent
factors to show that these triangle pieces glue up to form a polygon with $(d+3)$
vertices.
\end{enumerate}

\begin{proof}[Proof of Theorem \ref{thm:polygon}.]
\mbox{}
Any polynomial affine sphere can be parameterized so that its Pick
form is monic by composing an arbitrary parameterization with a
suitable automorphism $z \mapsto \lambda z$, $\lambda \in \C^*$.  Therefore
we can (and do) assume $C$ is monic throughout the proof, since the only
part of the theorem that involves a specific parameterization
(i.e.~the image $f(\star_d)$) includes monicity as a hypothesis.

Let $U$ be one of the standard half-planes for $C$ given by
Proposition \ref{prop:standard-half-planes}.  Our first goal is to
understand the part of $\bdry \P(M)$ that arises from \emph{rays} in
$U$, i.e.~paths of the form $\gamma(t) = b + e^{i \theta}t$ where
$t \geq 0$ and $b$ is arbitrary.  Here we call $\theta \in [-\pi/2,\pi/2]$ the
\emph{direction} of the ray.

\boldpoint{Step 1: Finding stable limits.}
We will say that such a ray is \emph{stable} if $\theta \not\in
\{-\pi/2,-\pi/6,\pi/6,\pi/2\}$.  Note the possible directions of stable
rays form three intervals of length $\pi/3$, which we denote by
$$ J_- = (-\pi/2,-\pi/6), \;\; J_0 = (-\pi/6,\pi/6), \;\; J_+ =
(\pi/6,\pi/2).$$

Recall from Section \ref{sec:cubic} that a \emph{quasi-ray} is a path
that can be parameterized so that it eventually lies in a half-plane,
in which it has distance $o(t)$ from a ray parameterized by arc length
$t$.  We say a quasi-ray is stable if the direction of some associated
ray is stable.

The ``stability'' of rays and quasi-rays in these directions refers to
convergence of the osculation map:

\begin{lem}
\label{lem:stability}
If $\gamma$ is a stable ray or quasi-ray, then $\lim_{t \to \infty}
\hat{F}(\gamma(t))$ exists.  Furthermore, among all such rays only
three limits are seen:  There exist $L_-, L_0, L_+ \in \GL_3\R$ such
that 
$$\lim_{t \to \infty} \hat{F}(\gamma(t)) = \begin{cases}
L_- &\text{ if }\theta \in J_-\\
L_0 &\text{ if }\theta \in J_0\\
L_+ &\text{ if }\theta \in J_+
\end{cases}$$
\end{lem}

\begin{proof}
First we consider rays, and at the end of the proof we show that
quasi-rays have the same behavior.

Let $\gamma$ be a ray and for brevity write $G(t) =
\hat{F}(\gamma(t))$. By \eqref{eqn:osculation-derivative} we have
$$ G(t)^{-1} G'(t) = \Ad_{H(\gamma(t))} \Theta(u) (\gamma'(t))$$
Applying Theorem \ref{thm:exponential-bound} and Corollary
\ref{cor:c1-bound} to $u$ and $\phi = \sqrt{2} C$, and using that
$|\gamma'(t)| = 1$, we have
$$\Theta(u)(\gamma'(t)) = O(e^{-2\sqrt{3} t} / \sqrt{t}).$$
Note that the exponential decay rate of $2
\sqrt{3}$ (rather than $\sqrt{6}$ seen in the theorems cited) reflects the fact that we are working in coordinates where $C =
2 dz^3$ and $\phi = 2^{3/2} dz^3 = (\sqrt{2} dz)^3$, so
$|\phi|^{2/3}$-distances are related to $|dz|^2$ distances by a factor of $\sqrt{2}$.

Conjugating $\Theta$ by the diagonal matrix $H(z)$ multiplies the
$(i,j)$ entry by
\begin{equation}
\label{eqn:conjugation-by-diagonal}
\lambda_{ij} := \exp( 2 \Re(z(\omega^{(1-i)} - \omega^{(1-j)})).
\end{equation}
In this case $z = \gamma(t) = b + e^{i \theta} t$, and taking the
maximum over $i$ and $j$ we find 
\begin{equation}
\label{eqn:c-of-theta}
\lambda_{ij} = O(e^{c(\theta) t})
\end{equation}
where the optimal coefficient $c(\theta)$ has a simple geometric
description: Inscribe an equilateral triangle in $|z|=2$ with one
vertex at $e^{i \theta}$.  Project the triangle orthogonally to $\R$
and let $c(\theta)$ be the length of the resulting interval.

In particular, the coefficient $c(\theta)$ achieves its maximum $2 \sqrt{3}$ exactly
when one of the sides of the triangle is horizontal, or equivalently, if
and only if the ray is \emph{not} stable.

Combining these bounds for $\Theta$ and $\lambda_{ij}$, we find that
for any stable ray, we have definite exponential decay in the
equation satisfied by $G$, i.e.
$$ G(t)^{-1} G'(t) = O(e^{-\alpha t} / \sqrt{t})$$
where $\alpha = 2 \sqrt{3} - c(\theta) > 0$. 
Standard ODE techniques (see Lemma \ref{lem:ode-small-coef}.(ii) in Appendix
\ref{appendix:ode}) then show that $\lim_{t \to \infty} G(t)$ exists.

Now suppose that $\gamma_1$ and $\gamma_2$ are stable rays with
respective angles $\theta_1,\theta_2$ that belong to the same interval
($J_-$, $J_0$, or $J_+$).  We will show that $G_1(t)^{-1} G_2(t) \to
I$ as $t \to \infty$, where $G_i(t) = \hat{F}(\gamma_i(t))$.  This
means that $\hat{F}$ has the same limit along these rays, giving
$L_-$, $L_0$, and $L_+$ as in the statement of the lemma.

For any $t \geq 0$ let $\eta_t(s) = (1-s) \gamma_1(t) + s \gamma_2(t)$ be
the constant-speed parameterization of the segment from $\gamma_1(t)$
to $\gamma_2(t)$.  Let $g_t(s) = \Hat{F}(\eta_t(0))^{-1}
\Hat{F}(\eta_t(s))$, which satisfies
\begin{equation}
\label{eqn:path-joining-rays}
\begin{split}
g_t^{-1}(s) g_t'(s) &= \Ad_{H(\eta_t(s))} \Theta( u ) (\eta_t'(s))\\
g_t(0) &= I\\
g_t(1) &= G_1(t)^{-1} G_2(t)
\end{split}
\end{equation}
Since $|\eta_t'(s)| = O(t)$, the analysis above shows that
$g_t^{-1}(s) g_t'(s) = O(\sqrt{t} e^{-\alpha t})$
where now $\alpha = \left (2 \sqrt{3} - \sup_{\theta_1 \leq \theta \leq
  \theta_2} c(\theta) \right ) > 0$ because $\pm \pi/2, \pm \pi/6 \not \in [\theta_1,\theta_2]$.
In particular by making $t$ large enough we can
arrange for $g_t^{-1}(s) g_t'(s)$ to be uniformly small for all $s \in
[0,1]$.  Once again standard ODE methods (Lemma \ref{lem:ode-small-coef}.(i)) give the desired convergence,
$$G_1(t)^{-1}G_2(t) = g_t(1) \to I \text{ as }t \to \infty.$$

Finally, suppose that $\gamma_1$ is a stable quasi-ray, and $\gamma_2$
the ray that it approximates (with direction $\theta$).  We proceed as
above to study $\eta_t(s) = (1-s) \gamma_1(t) + s \gamma_2(t)$ and the
restriction of the frame field to this homotopy from $\gamma_1$ to
$\gamma_2$.  In this case we have the stronger bound on the derivative
$|\eta_t'(s)| = o(t)$, and the previous bound on $g_t^{-1}(s) g_t'(s)$
applies again with exponent $\alpha = (2 \sqrt{3} - c(\theta))$.
Thus as before we find $g_t(1) \to I$ as $t \to \infty$, and that the
frame field has the same limit on the stable quasi-ray $\gamma_1$ as
on an associated stable ray $\gamma_2$.
\end{proof}

\boldpoint{Step 2: Finding unipotent factors.}
Next we will analyze the behavior of the osculation map near an
unstable ray in order to understand the relationship between $L_-$, $L_0$,
and $L_+$.

\begin{lem}
\label{lem:unipotent-factors}
Let $L_-, L_0, L_+$ be as in the previous lemma.  Then there exist $a,b \in \R$ such that
\begin{equation}
\label{eqn:unipotent-factors}
L_-^{-1} L_0 =  \begin{pmatrix}
1 & a & \\
 & 1 & \\
 &  & 1
\end{pmatrix}  \;\text{ and }\;
L_0^{-1} L_+ = \begin{pmatrix}
1 &  & b\\
 & 1 & \\
 &  & 1
\end{pmatrix},
\end{equation}
where the matrix entries not shown are zero.
\end{lem}

\begin{proof}
We give a detailed proof for $L_0^{-1}L_+$ and then indicate what must
be changed to handle the other case.  We begin as in the last part of the
previous proof, i.e.~joining two rays by a path and studying the
restriction of $\hat{F}$ to the path.

Consider the rays $\gamma_0(t) = t$ and $\gamma_+(t) = e^{i \pi/3} t$.
The restrictions $G_0 = \hat{F} \circ
\gamma_0$ and $G_+ = \hat{F} \circ \gamma_+$ have respective limits $L_0$
and $L_+$.  For any $t>0$, join $\gamma_0(t)$ to $\gamma_+(t)$ by
a circular arc
$$ \eta_t(s) = e^{is} t, \text{ where } s \in [0,\pi/3]$$
and let $g_t(s) = \hat{F}(\eta_t(0))^{-1} \hat{F}(\eta_t(s))$.  Then $g_t :
[0,\pi/3] \to \GL_3\R$ satisfies the ordinary differential equation
\eqref{eqn:path-joining-rays}
with $G_1, G_2$ replaced by $G_0, G_+$.

Unlike the previous case, however, the coefficient
$$ M_t(s) := \Ad_{H(\eta_t(s))} \Theta( u ) (\eta_t'(s))$$
that appears in this equation is not exponentially small in $t$
throughout the interval.  At $s=\pi/6$, conjugation by $H(\eta_t)$
multiplies the $(1,3)$ entry of $\Theta( u )$ by a factor of
$\exp(2\sqrt{3}t)$, exactly matching the exponential decay rate for
$\Theta$ and giving  
$$M_t(\pi/6) = O(|\eta_t'| / \sqrt{t}) = O(\sqrt{t}),$$
where in the second equality we used $|\eta_t'| = t$.

However, this potential growth in the coefficient matrix $M_t$ is seen
\emph{only} in this $(1,3)$ entry, because by
\eqref{eqn:conjugation-by-diagonal} the other entries are scaled by
smaller exponential factors.  (That is, the elementary matrix
$E_{1,3}$ is the leading eigenvector of $\Ad_{H(\eta_t)}$.)
Furthermore the effect rapidly decays as the angle moves away from
$\pi/6$: For $\theta \in [0,\pi/3]$ and $c(\theta)$ as in
\eqref{eqn:c-of-theta} we have
$$ c(\theta) = 2 \sqrt{3} \cos \left (\frac{\pi}{6} - \theta \right ) \leq 2
\sqrt{3} -(\theta - \pi/6)^2.$$
Combining these two observations we can separate the unbounded entry in
$M_t(s)$ and write
$$ M_t(s) = M^0_t(s) + \mu_t(s) E_{13}$$
where $M^0_t(s) = O(\exp(-\alpha t))$ for some $\alpha > 0$, $E_{13}$
is the elementary matrix, and
$$ \mu_t(s) = O \left ( |\eta_t'| \exp \left ((2 \sqrt{3} - c(s))t
\right ) / \sqrt{t} \right ) = O\left (\sqrt{t} \exp \left (-(s -
\pi/6)^2 t \right ) \right ).$$ This upper bound is a Gaussian
function in $s$, normalized such that its integral over $\R$ is
independent of $t$.  (As $t \to \infty$ this Gaussian approximates a
delta function at $s=\tfrac{\pi}{6}$.)  Therefore the function $\mu_t(s)$ is
uniformly absolutely integrable over $s \in [0,\pi/3]$ as $t \to
\infty$.

With a coefficient of this form---an integrable component with
values in a fixed $1$-dimensional space, plus a small error---it
follows from Lemma \ref{lem:ode-nearly-abelian} that the solution of the
initial value problem \eqref{eqn:path-joining-rays} satisfies
$$ \left \| g_t(\pi/3) - \exp \left ( E_{13} \int_0^{\pi/3} \mu_t(s))
\right) \right \| \to 0 \text{ as } t \to \infty.$$
Since $g_t(\pi/3) = \hat{F}(t)^{-1} \hat{F}(e^{i \pi/3} t) \to
L_0^{-1} L_+$ as $t \to \infty$, this gives the desired unipotent
form \eqref{eqn:unipotent-factors} for some $b \in \R$.

The value of $L_-^{-1}L_0$ is computed by a nearly identical argument
applied to rays at angles $-\pi/3$ and $0$.  The only difference is
that at $\theta = -\pi/6$, the leading eigenvector $\Ad_{H(e^{i
    \theta} t)}$ is the elementary matrix $E_{12}$, which becomes the
dominant term in the coefficient $M_t(s)$.  Exponentiating we find
$L_-^{-1}L_0$ has the desired form \eqref{eqn:unipotent-factors}.
\end{proof}

\boldpoint{Step 3: Finding triangle pieces.}  We now turn to studying
the shape of the projectivized image $\P(M) \subset \RP^2$. Let $V$
denote the union of the edges $e_{110}$ and $e_{101}$ of the standard
triangle $\triangle_0 \subset \RP^2$ that is the image
of the normalized \titeica surface $T$.  (Recall the notation for
vertices and edges of this triangle was described in Section
\ref{sec:cubic}.)

In the following proposition, we say that a ray has \emph{height} $y$
if it contains the point $z = 1 + i y$ in $U$, where $y \in \R$.

\begin{lem}[Projective limits in a half-plane]
\label{lem:vee}
Let $L_0 \in \GL_3\R $ be a limit of the osculation map of an affine
sphere $M$ as above.  Then the following table describes the
projective limits of $f$-images of stable (quasi-)rays in $U$:
\begin{center}
\begin{tabular}{l l l}
\hline
\textbf{Type of path $\gamma$} & \textbf{Direction $\theta$} & \textbf{Projective limit $p_\gamma$ of $f(\gamma)$}\\
\hline
Quasi-ray& $\theta \in (-\tfrac{\pi}{2}, -\tfrac{\pi}{3})$ & $p_\gamma = L_0 \cdot v_{001}$\\
Ray (of height $y$)& $\theta = -\tfrac{\pi}{3}$ & $p_\gamma \in L_0 \cdot
e_{101}^\circ$\\
& & $(p_\gamma \to L_0 \cdot v_{001}$ as $y \to -\infty)$\\
Quasi-ray& $\theta \in (-\tfrac{\pi}{3}, \tfrac{\pi}{3})$, $\theta \neq \pm \tfrac{\pi}{6}$ & $p_\gamma = L_0 \cdot v_{100}$\\
Ray (of height $y$)& $\theta = \tfrac{\pi}{3}$ & $p_\gamma \in L_0 \cdot
e_{110}^\circ$\\
& & $(p_\gamma \to L_0 \cdot v_{010}$ as $y \to \infty)$\\
Quasi-ray& $\theta \in (\tfrac{\pi}{3},\tfrac{\pi}{2})$ & $p_\gamma = L_0 \cdot v_{010}$\\
\hline
\end{tabular}
\end{center}
And in particular:
\begin{itemize}
\item The projectivization of any stable quasi-ray of angle
zero in $U$ tends to $L_0 \cdot v_{100} \in L_0 \cdot V$ (by the middle row of
the table), and
\item We have $L_0 \cdot V \subset \bdry \P(M)$ (since $V =
v_{001} \cup e_{101}^\circ \cup v_{100} \cup e_{110}^\circ \cup
v_{010}$).
\end{itemize}
\end{lem}

\begin{proof}

First, using the explicit formula \eqref{eqn:norm-titeica-def} for the
normalized \titeica surface $T$, it is easy to calculate the
projective limit $v_\gamma$ of the $T$-image of any ray or quasi-ray
in the right half-plane.  (At this point stability is not relevant.)
The result is a table for $T$ analogous to the one we seek for $f$
(compare \cite[Tbl.~2]{loftin:compactification-i}):

\begin{center}
\begin{tabular}{l l l}
\hline
\textbf{Type of path $\gamma$} & \textbf{Direction $\theta$} & \textbf{Projective limit $v_\gamma$ of $T(\gamma)$}\\
\hline
Quasi-ray& $\theta < -\tfrac{\pi}{3}$ & $v_\gamma = v_{001}$\\
Ray (of height $y$) & $\theta = -\tfrac{\pi}{3}$ & $v_\gamma \in e_{101}^\circ$\\
& & $(v_\gamma \to v_{001}$ as $y \to -\infty)$\\
Quasi-ray& $ \theta \in (-\tfrac{\pi}{3}, \tfrac{\pi}{3})$ & $v_\gamma = v_{100}$\\
Ray (of height $y$) & $\theta = \tfrac{\pi}{3}$ & $v_\gamma \in e_{110}^\circ$\\
& & $(v_\gamma \to v_{010}$ as $y \to \infty)$\\
Quasi-ray& $\tfrac{\pi}{3} < \theta$  & $p_\gamma = v_{010}$\\
\hline
\end{tabular}
\end{center}

Now suppose $\gamma$ is a \emph{stable} ray or quasi-ray in $U$, and let
$L_\gamma = \lim_{t \to \infty} \hat{F}(\gamma(t))$.  Since $f(z) =
\hat{F}(z) T(z)$, we find that the projective limits $v_\gamma$ of
$\P(T(\gamma))$ and $p_\gamma$ of $\P(f(\gamma))$ are related by
\begin{equation}
\label{eqn:ray-limit}
p_\gamma = L_\gamma \cdot v_\gamma
\end{equation}
Note that since $\gamma$ is a divergent path, each
point $p_\gamma$ obtained in this way lies on the boundary of $\P(M)$.

By Lemma \ref{lem:stability} we have $L_\gamma \in \{ L_-, L_0, L_+
\}$ with the value depending only on $\theta$.  Hence
the combination of formula \eqref{eqn:ray-limit} and the table of
\titeica limits gives the following characterization of $f$-limits:

\begin{center}
\begin{tabular}{l l l}
\hline
\textbf{Type of path $\gamma$} & \textbf{Direction $\theta$} & \textbf{Projective limit $p_\gamma$ of $f(\gamma)$}\\
\hline
Quasi-ray& $\theta \in (-\tfrac{\pi}{2},-\tfrac{\pi}{3})$ & $p_\gamma = L_- \cdot v_{001}$\\
Ray (of height $y$) & $\theta = -\tfrac{\pi}{3}$ & $p_\gamma \in L_-
\cdot e_{101}^\circ$\\
& & $(p_\gamma \to L_- \cdot v_{001}$ as $y \to -\infty)$\\
Quasi-ray& $ \theta \in (-\tfrac{\pi}{3},-\tfrac{\pi}{6})$ & $p_\gamma = L_- \cdot v_{100}$\\
Quasi-ray& $ \theta \in (-\tfrac{\pi}{6}, \tfrac{\pi}{6})$ & $p_\gamma = L_0 \cdot v_{100}$\\
Quasi-ray& $ \theta \in (\tfrac{\pi}{6}, \tfrac{\pi}{3})$ & $p_\gamma = L_+ \cdot v_{100}$\\
Ray (of height $y$) & $\theta = \tfrac{\pi}{3}$ & $p_\gamma \in L_+ \cdot
e_{110}^\circ$\\
& & $(p_\gamma \to L_+ \cdot v_{010}$ as $y \to \infty)$\\
Quasi-ray& $\theta \in (\tfrac{\pi}{3}, \tfrac{\pi}{2})$ &
$p_\gamma = L_+ \cdot v_{010}$\\
\hline
\end{tabular}
\end{center}

\medskip

This is nearly the characterization of projective limits we seek; if
we replace all instances of $L_-$ and $L_+$ with $L_0$ in the table
above (and coalesce the middle three rows, where the limit becomes the
same) we obtain exactly the statement of the lemma.

The proof is completed by using Lemma \ref{lem:unipotent-factors} to
verify that in each place that $L_-$ or $L_+$ appears in the previous
table, the projective transformation is applied to a point in $\RP^2$
where it has the same action as $L_0$.  Of the six affected rows,
there are actually only two cases to consider:

\begin{itemize}
\item Since $L_0 = L_+ U_+$ where $U_+$ is a unipotent that fixes
the line in $\RP^2$ containing $e_{110}$ pointwise, we have
$$L_0 \cdot e_{110} = L_+U_+ \cdot e_{110}= L_+ \cdot e_{110}.$$
Since $v_{100}, v_{010}$ are the endpoints of $e_{110}$, it
follows that $L_+$ can be replaced by $L_0$ in the previous table.\\

\item Since $L_0 = L_- U_+$ where $U_-$ is a unipotent that fixes
the line in $\RP^2$ containing $e_{101}$ pointwise, we have
$$L_0 \cdot e_{101} = L_-U_- \cdot e_{101} = L_- \cdot e_{101}$$
Since $v_{100}, v_{010}$ are the endpoints of $e_{110}$, it
follows that $L_-$ can be replaced by $L_0$ in the previous table.
\end{itemize}
\end{proof}

\boldpoint{Step 4: Assembling the polygon.}  
So far we have been working in a single fixed half-plane $U$ for the
Pick differential $C$.  Now we consider how the picture changes as we move
between the standard half-planes $U_0, \ldots, U_{d+2}$ associated to
$C$ by Proposition \ref{prop:standard-half-planes}.  By the
construction above, we obtain the following for each $0 \leq i \leq
d+2$:
\begin{itemize}
\item A set of limits $L^{(i)}_-,L^{(i)}_0,L^{(i)}_+ \in \GL_3\R$ of
the osculation map in $U_i$ restricted to stable rays
\item Unipotent elements as in Lemma \ref{lem:unipotent-factors} that
relate these limits, and
\item The conclusion that $L_0^{(i)} \cdot V \subset \bdry \P(M)$.
\end{itemize}

Thus each of the half-planes gives a piece of the boundary of $\bdry
\P(M)$ that is a ``vee'', i.e.~the image of $V$ by a projective
transformation.

By studying the overlap
between these edge pairs, we can finally establish:

\begin{lem}
The projectivization $\P(M)$ of the affine sphere $M$ is a convex
polygon with $d+3$ vertices.  The projectivization of $f(\star_d)$
is an embedded open star in $\P(M)$ whose edges tend to the
vertices.
\end{lem}

\begin{proof}
\renewcommand{\qedsymbol}{}

Consider the $f$-images of rays in $U_i$ with angle $\theta = \pi/3$,
which by the previous lemma projectively limit on the edge $L^{(i)}_0
\cdot e_{110}^\circ$.

By Proposition \ref{prop:standard-half-planes}, the next half-plane
$U_{i+1}$ (with index understood mod $d+3$) intersects $U_i$ in a
sector that contains all but an initial segment from each of these
rays.  In the coordinate $z_{i+1}$ of $U_{i+1}$, these rays have angle
$\theta_{i+1} = -\pi/3$.  Hence by applying the previous lemma in
$U_{i+1}$ we find the $f$-images of the same rays
projectively limit on $L^{(i+1)}_0 \cdot e_{101}^\circ$, and thus
$$ L^{(i)}_0 \cdot e_{110}^\circ = L^{(i+1)}_0 \cdot e_{101}^\circ.$$
By continuity of projective transformations, we have the same equality
for the associated closed edges.  Furthermore, the previous lemma
characterizes the behavior of the limit point as a function of the
height of the ray, determining which pairs of endpoints are identified,
namely:
\begin{equation*}
\begin{split}
L^{(i)}_0 \cdot v_{100} &= L^{(i+1)}_0 \cdot v_{001},\\
L^{(i)}_0 \cdot v_{010} &= L^{(i+1)}_0 \cdot v_{100}.
\end{split}
\end{equation*}
Thus if we orient the edge pair $V$ from $v_{001}$ to
$v_{010}$, we have found that the union of $L^{(i)}_0 \cdot V$
and $L^{(i+1)}_0 \cdot V$ is an oriented chain
of three edges in $\bdry \P(M)$.

Allowing $i$ to vary we find that the overlapping edge pairs
$\{ L^{(i)}_0 \cdot V \}_{i=0\ldots d+2}$ assemble into a map 
$$\Gamma : P_{d+3} \to \bdry \P(M) \simeq S^1$$
where $P_{d+3}$ is an abstract $(d+3)$-gon, considered as a simplicial
$1$-complex.  By construction $\Gamma$ is linear on each edge, its
restriction to any pair of adjacent edges is an embedding (with image
$L^{(i)}_0 \cdot V$, for some $i$), and the image of any vertex of
$P_{d+3}$ is a corner of the convex curve $\bdry \P(M)$ which can be
described as $L_0^{(i)} \cdot v_{100}$ for some $i$.

Because adjacent pairs of edges embed, the map $\Gamma$ is a local
homeomorphism of compact, connected Hausdorff spaces.  Thus $\Gamma$
is a covering map, and in particular surjective.  The image $\bdry
\P(M)$ is therefore a polygon and, considering that polygonal curve as
a $1$-complex, the covering is simplicial.

To identify the image as a $(d+3)$-gon, it remains to show that
$\Gamma$ is injective, which follows if it is injective on vertices.

Recall from Proposition \ref{prop:standard-half-planes} that for each
$i$ there is an edge $\gamma_i$ of the standard star $\star_d$ that is
eventually contained in that half-plane $U_i$ in which it is a
quasi-ray with direction $\theta_i = 0$. (This is the point at which
we use that $C$ is monic in an essential way.)  Applying Lemma
\ref{lem:vee} to these quasi-rays find that the projectivizations
$\P(f(\gamma_i))$ tend to the points $L_0^{(i)} \cdot v_{100}$, that
is, the images of vertices of $P_{d+3}$ by $\Gamma$.

Suppose for contradiction that $\P(\gamma_i)$ and $\P(\gamma_j)$ have
the same limit point $x \in \bdry \P(M)$ and $i \neq j$.  Note $i \neq
j \pm 1$ mod $(d+3)$ since neighboring vertices (and the edge they
span) map to distinct points by Lemma \ref{lem:vee}.

The union $\beta = \gamma_i \cup \gamma_j$ of two edges of $\star_d$
separates $\C$ into two components, and since $\P(f)$ is a
homeomorphism onto the convex domain $\P(M)$, the image curve
$\P(f(\beta))$ separates $\P(M)$.  All branches of $\star_d$ contain
the origin, but except for this common point the paths $\gamma_{j+1}$
and $\gamma_{j-1}$ lie in different components of $\C \setminus
\beta$.  Neither of $\gamma_{j \pm 1}$ has projective image
converging to $x$ since these are the neighbors of $\gamma_j$.  Thus
each component of $\P(M \setminus f(\beta))$ accumulates on at least
one boundary point of $\P(M)$ that is distinct from $x$.  This is a
contradiction, however, because $\P(f(\beta))$ is a properly embedded path in
the open disk $\P(M)$ that limits on a single boundary point $x\in
\bdry \P(M)$ in both directions, so one of its complementary disks has
$x$ as the only limit point on $\bdry \P(M)$.

Thus we find that $\Gamma$ is injective, and that the projectivized
image $\P(f(\star_d))$ gives an embedded star in $\P(M)$ that limits
on the vertices of the polygon, giving an adjacency-preserving
bijection of them with the edges of $\star_d$.

This completes the proof of the lemma, and of Theorem
\ref{thm:polygon}.\end{proof}
\end{proof}

\section{From polygons to polynomials}
\label{sec:polygons-to-polynomials}

The main goal of this section is to establish the converse of Theorem
\ref{thm:polygon}:

\begin{thm}
\label{thm:polynomial}
Let $M$ be a complete hyperbolic affine sphere in $\R^3$ asymptotic to
the cone over a convex polygon with $n$ vertices.  Then the Blaschke
metric of $M$ is conformally equivalent to the complex plane $\C$ and
its Pick differential is a polynomial cubic differential of degree $(n-3)$.
\end{thm}

\begin{proof}
If $n=3$ then $M$ is a \titeica surface, so the statement follows
immediately.  For the rest of the proof we assume $n \geq 4$.

Let $P = \P(M)$ be the convex polygon.  For any vertex $v$ of $P$ let
$\tau_v$ be the triangle formed by $v$ and its two neighboring
vertices.  We say $\tau_v$ is the \emph{vertex inscribed triangle} of
$P$ at $v$.

Considering the triangle $\tau_v$ as the projectivized image of a
\titeica affine sphere, we can choose a parameterization $T_v : \C \to
\R^3$ so that the projective image of the positive real axis is
asymptotic to $v$.

By convexity $\tau_v$ is contained in $P$, so for each $z \in \C$
there is a unique point $M_v(z) \in M$ collinear with $T_v(z)$ and
the origin.  This gives a smooth embedding $M_v : \C \to M$.

\begin{lem}
\label{lem:locally-bilip-half-planes}
Equip $\C$ with the Euclidean metric $|dz|^2$ and $M$ with either the
Blaschke metric or the Pick differential metric $|C|^{2/3}$.  Then the
differential $dM_v : T_z\C \to T_{M_v(z)}M$ is bilipschitz when $\Re(z)$ is
large.  Moreover there are constants $R, \Lambda$ such that for any $z$ with
$\Re(z) \geq R$ and any $\xi \in T_z\C$ we have
\begin{equation*}
\label{eqn:comparable-metrics}
\frac{1}{\Lambda} \|\xi\| \leq \| dM_v(\xi) \| \leq \Lambda \|\xi\|.
\end{equation*}
\end{lem}

\begin{proof}

We will show that the differential of the map from the \titeica
surface over $\tau_v$ to $M$ obtained by projecting along rays through
$0 \in \R^3$ is bilipschitz in the region corresponding to $\Re(z)
\geq R$.  Since $M_v$ is the composition of this projection with the
parameterization $T_v$, and since both the Blaschke and Pick
differential metrics of the \titeica surface are multiples of $|dz|^2$
in that parameterization, the lemma will follow.

For the remainder of the proof we consider the images of both $T_v$
and $M_v$ to be parameterized by their common projectivization, which
is the triangle $\tau_v$.  Composing this parameterization with the
inverse of $T_v$, the coordinate $z$ of $\C$ becomes a function on
the triangle $\tau_v$; we denote the image of $p \in \tau_v$ by 
$z(p)$.  We must
show that for any $p \in \tau_v$ with $\Re(z(p))$ large, the
respective metrics of $T_v$ and $M$ are uniformly comparable at $p$.

The triangle $\tau_v$ is an orbit of a maximal torus in $\SL_3\R$ (as
is the surface $T_v$ itself).  Fix a basepoint $p_0 \in \tau_v$ and
for any other $p \in \tau_v$ let $A(p)$ be the element of this torus
mapping $p$ to $p_0$.

The key observation is is that by taking $\Re(z(p))$ large enough, we
can assure that the image $A(p) \cdot P$ of the polygon $P$ is
arbitrarily close to $\tau_v$ in the Hausdorff topology.

To see this, first normalize with a projective transformation so that
$\tau_v$ is the standard triangle $\triangle_0$, the vertex $v$ is
$v_{100}$, and $T_v$ is the normalized \titeica surface.  Then $A(p) =
H(z(p_0)) H(z(p))^{-1}$ is diagonal and taking $\Re(z(p))$ to be large
makes the $(1,1)$ entry of $A(p)$ small.  This means that the
projective action of $A(p)$, while preserving the two shared edges of
$P$ and $\tau_v$, maps the rest of $P$ very close to the third edge of
$\tau_v$ (which is $e_{011}$ in this normalization): geometrically,
the map $A(p)$ sends $p$ to $p_0$ and fixes the vertices of the
triangle $\tau_v$, with $v$ being a repelling fixed point, and the
other vertices being hyperbolic fixed points.  Thus for $\Re(z(p))$
large enough, the image $A(p) \cdot P$ lies in any chosen Hausdorff
neighborhood of $\tau_v$.

Now we use the projective naturality of the Blaschke metric and the
Pick differential.  Instead of comparing the metrics of the affine
spheres over $\tau_v$ and $P$ at an arbitrary point $p \in \tau_v$, it
suffices to compare the metrics of the affine spheres over $A(p) \cdot
\tau_v = \tau_v$ and $A(p) \cdot P$ at the fixed point $p_0$.  By
Corollary \ref{cor:blaschke-and-pick-continuous}, both the Blaschke
metric and the Pick differential metric at $p_0$ vary continuously in
the Hausdorff topology on pointed convex sets.  Taking $R$ large
enough, we can assume that $A(p) \cdot P$ lies in a neighborhood of
$\tau_v$ such that the Blaschke metric on the tangent space at $p_0$
is $\Lambda$-bilipschitz to that of the \titeica surface over $\tau_v$, and
similarly for the Pick differential metric, as required.
\end{proof}

We remark that the proof above actually shows more: By taking $R$
large enough, we can make the constant $K$ as close to $1$ as we like.
However, we will not need this refined version of the estimate in what
follows.

In the normalized \titeica surface, the projectivized image of $\{
\Re(z) > R\}$ is the intersection of a neighborhood of the union of
$v_{100}$ and the open edges $e_{110}^\circ$, $e_{101}^\circ$ with the
interior of the standard triangle (see Figure
\ref{fig:barrier-and-core}a).  Correspondingly, the part of $\tau_v$
in which the estimate of the previous Lemma applies is a neighborhood
of $v$ and the adjacent open edges.  It is bounded by a curve that
joins the neighboring vertices of $v$, namely, the image of $\Re(z)=R$
in the \titeica surface over the vertex inscribed triangle (see Figure
\ref{fig:barrier-and-core}b).  We call this the \emph{barrier curve at
  $v$}.

Applying Lemma~\ref{lem:locally-bilip-half-planes} to each vertex of
$P$ in turn, we find that its conclusion applies in a set of $n$
half-planes that cover all but a compact subset $K$ of the interior of
$P$ (see Figures \ref{fig:barrier-and-core}c,\ref{fig:barrier-and-core}d); this set is a closed
curvilinear polygon bounded by arcs from the barrier curves.  We call
$K$ the \emph{core} of $P$.

\begin{figure}
\begin{center}
\includegraphics[width=\textwidth]{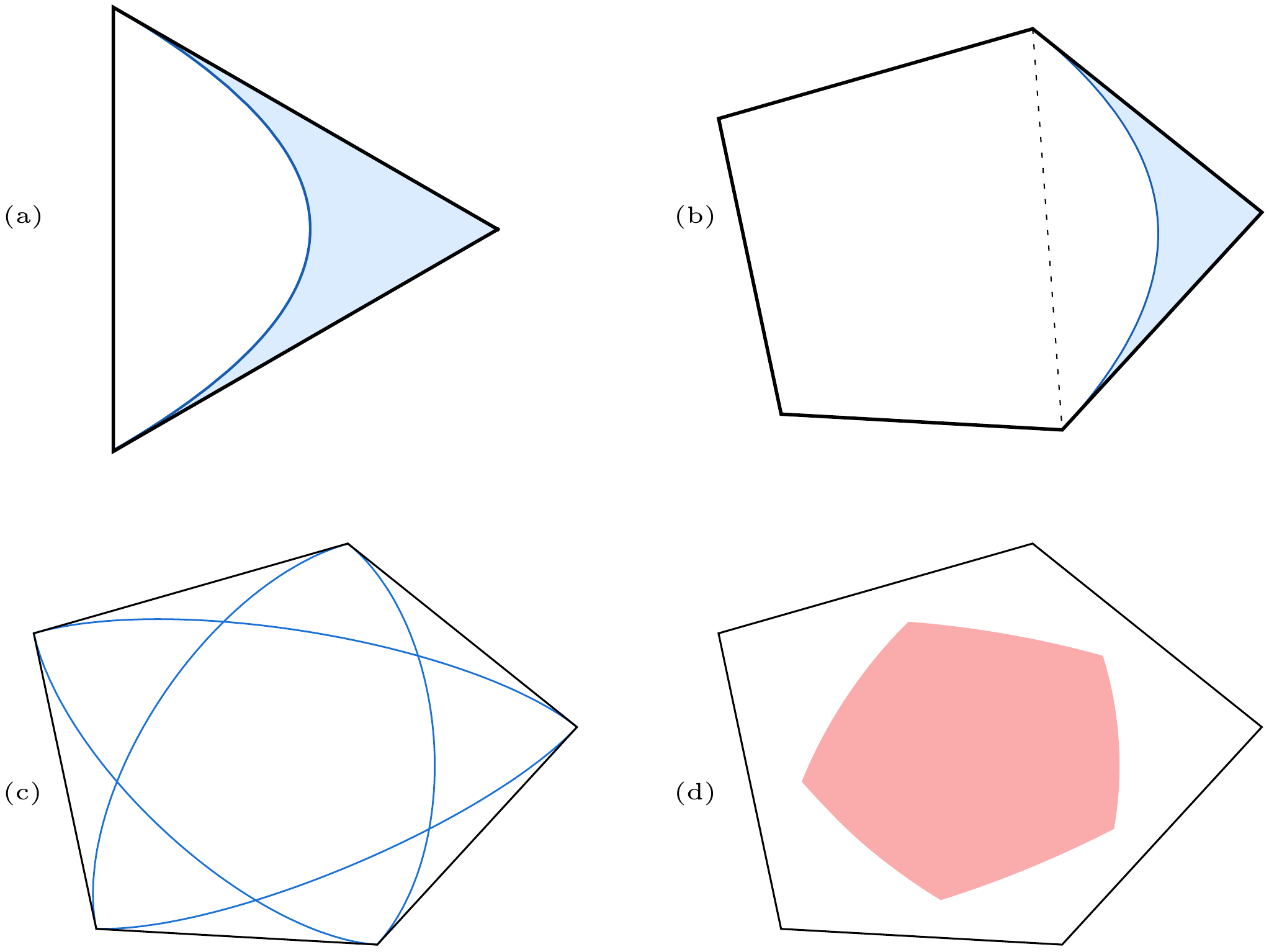}
\caption{(a) The image of $\{\Re(z) > R\}$ in the triangle
  projectivization of a \titeica surface; (b) The corresponding region
  in a vertex inscribed triangle of a polygon; (c) Repeating this at
  each vertex, we obtain a set of \textbf{barrier curves}, and (d) the
  \textbf{core}, a compact set containing all zeros of the Pick differential.}
\label{fig:barrier-and-core}
\end{center}
\end{figure}

Using this construction, we can establish some key properties of the
Pick differential and Blaschke metric:

\begin{lem}
\label{lem:pick-properties}
\mbox{}
\begin{rmenumerate}
\item The Pick differential has finitely many zeros.
\item The Pick differential metric $|C|^{2/3}$ of $M$ is quasi-isometric to the
Blaschke metric of $M$, and in particular, it is complete.
\end{rmenumerate}
\end{lem}

\begin{prooflist}
\item The Pick differential has no zeros in any of the half-planes given by
Lemma \ref{lem:locally-bilip-half-planes}, hence the zeros all lie in
the core $K$, which is compact.  The Pick differential is holomorphic and does
not vanish identically, so its zeros have no accumulation point.  The
zero set is therefore compact and discrete, hence finite.

\item Outside the core, the Pick differential and Blaschke metrics of $M$ are
uniformly comparable.  Because the core is compact, it has finite
diameter for both metrics.  Thus a geodesic for one metric can be
split into a part of bounded diameter and a part in which the other
metric is bounded above and below, giving quasi-isometry.  The
Blaschke metric is complete, so this shows $|C|^{2/3}$ is complete as
well.
\end{prooflist}

By analyzing the continuity of the construction Lemma
\ref{lem:locally-bilip-half-planes} as a function of the vertices of
the polygon, we can also show that this compact set containing the
zeros of the Pick differential for $P$ has the same property for polygons
sufficiently close to $P$.  This observation will be used in Section
\ref{sec:mapping}.

\begin{lem}
\label{lem:compact-containing-zeros}
Let $P \in \Poly_n$ be a convex polygon in $\RP^2$.  Then there exists
a compact subset $\hat{K}$ of the interior of $P$ and a neighborhood $U$ of
$P$ in $\Poly_n$ with the following property:  If $P' \in U$ and if
$M'$ is the complete hyperbolic affine sphere asymptotic to the cone
over $P'$, then all of the zeros of the Pick differential of $M'$ lie over $K$.
\end{lem}

\begin{proof}
While the construction of barrier curves in a polygon involves some
choices, we will show that one can make the construction continuous in
a small neighborhood of $P$, i.e.~so that the barrier curve at a
vertex varies continuously in the Hausdorff topology when a small
deformation is applied to the vertices of $P$.  Of course this will
also imply that the core $K(P)$ varies continuously as well.

The lemma will then follow by taking $\hat{K}$ to be a compact set containing a
neighborhood of $K(P)$.  For $P'$ sufficiently close to $P$, the core
$K(P')$ will be contained in this neighborhood of $K(P)$ and hence
$\hat{K}$ will contain the Pick zeros of $P'$ as well.

To choose barrier curves continuously, first consider polygons $P$
that have a fixed vertex inscribed triangle $T$ at $v$ (that is, we
have fixed the location of $v$ and its two neighbors).  In the proof
of Lemma \ref{lem:locally-bilip-half-planes} the barrier is
constructed as the image of a vertical line $\{ \Re(z) = R \}$ in the
conformal parameterization of the \titeica surface over $T$.  The
barrier curve is completely determined by the real number $R$, which
must be large enough so that the associated subset of a maximal torus
in $\SL_3\R$ maps the polygon $P$ into a certain Hausdorff
neighborhood of $T$.  Choosing $R$ large enough, we can ensure this
not only holds for $P$, but also for the union of all polygons in a small
neighborhood of $P$ that share this vertex inscribed triangle.  Hence
a fixed barrier curve works for all of these polygons.

Now consider the general case, i.e.~polygons $P'$ near $P$ with no
restriction on the vertices.  Working in a sufficiently small
neighborhood of $P$ gives a natural bijection from the vertices of
$P'$ to those of $P$.  As in the normalization construction of Section
\ref{sec:polygons}, there is a unique projective transformation
$A(P')$ that maps four chosen vertices of $P'$ to the corresponding
vertices of $P$, and this projective map varies continuously with
$P'$.  Selecting $v$, its two neighbors, and an arbitrary fourth
vertex, we get normalizing projective transformations $A(P') \in
\SL_3\R$ so that $A(P') \cdot P'$ shares the vertex inscribed triangle
at $v$ with $P$.

Thus, after applying a projective transformation $A(P')$, we are
reduced to the case considered before, where a fixed barrier curve
could be used.  We therefore define the barrier curve for $P'$ by
applying $A(P')^{-1}$ to this fixed curve.  Since $A(P')$ is a
continuous function of $P'$, the curves constructed this way also vary
continuously.
\end{proof}

Returning to consideration of a fixed affine sphere $M$ over a polygon
$P$, we can now identify the conformal type of the Blaschke metric (or
the conformally equivalent Pick differential metric).

\begin{lem}
\label{lem:conformally-parabolic}
The affine sphere $M$ is conformally equivalent to $\C$.
\end{lem}

\begin{proof}
Since $M$ is simply-connected and noncompact, we need only show that
it is not conformally equivalent to the unit disk $\Delta$.

Suppose for contradiction that $M \simeq \Delta$ and write $C = C(z) dz^3$
where $C(z)$ is a holomorphic function.  By Lemma
\ref{lem:pick-properties}.(i) we have $C(z) = p(z) H(z)$ where $p$ is
a polynomial and $H$ has no zeros.

The unit disk does not admit a complete \emph{flat} conformal metric,
since the developing map of the Euclidean structure induced by such
a metric would be a conformal isomorphism $\Delta \to \C$.  
The conformal metric $|H|^{2/3} |dz|^2$ is flat, because $\log |H|$ is harmonic,
and therefore it is not complete.

But a divergent path of finite $|H|^{2/3}$-length also has finite
$|C|^{2/3}$-length because the polynomial $p$ is bounded on
$\Delta$.  Thus the Pick differential metric is not complete, contradicting
Lemma \ref{lem:pick-properties}.(ii).
\end{proof}

We remark that the finiteness of the zero set of $C$ means that the
\emph{integral curvature} of the Pick differential metric is finite.  Huber
showed that any Riemann surface which admits a complete conformal
metric of finite integral curvature is conformally parabolic
\cite[Thm.~15]{huber:subharmonic-functions}, and the proof above is an
adaptation of Huber's argument to this special case.  For smooth
conformal metrics and simply-connected surfaces, the same result was
proved earlier by Blanc and Fiala \cite{blanc-fiala}.

\begin{lem}
The function $C(z)$ is a polynomial.
\end{lem}

The follows from a lemma of Osserman
\cite[Lem.~9.6]{osserman:minimal-surfaces}, generalizing a result of
Finn \cite[Thm.~17]{finn}.  While these authors consider complete
conformal metrics of the form $|f|^2 |dz|^2$, where $f$ is
holomorphic, their arguments easily extend to $|C|^{2/3}$.  For the
reader's convenience we sketch the argument while incorporating the
necessary changes for this case:

\begin{proof}
Write $C(z) = p(z) e^{G(z)}$ where $p$ is a polynomial and $G$ an
entire function on $\C$.  We show that completeness of $|C|^{2/3}$ implies
that $G$ is constant.

Taking an integer $N > \frac{1}{3} \deg(p)$ we have
\begin{equation}
\label{eqn:c-length-estimate}
|C(z)|^{1/3} = O \left ( \left |z^N e^{G(z)} \right | \right )
\text{ as } z \to \infty.
\end{equation}
The function $F(z)$ with $F'(z) = z^N e^{G(z)}$ and $F(0) = 0$ has a
zero of order exactly $N+1$ at $0$, hence $\zeta = F^{1/(N+1)}$ is
single-valued and has an inverse function $z(\zeta)$ in some neighborhood
of $0$.

In fact this inverse must exist globally: Otherwise there would be a radial path
of the form $t \mapsto t \zeta_0$, $t \in [0,1)$, $|\zeta_0| = R$ on which 
$z(\zeta)$ is defined but cannot be extended.  The image $\Gamma$ of this path by $z(\zeta)$ satisfies
$$ \int_\Gamma |z^N e^{G(z)}| |dz| = \int_\Gamma |F'(z)| |dz| =
\int_\Gamma |d(\zeta^n)| =
R^{N+1}.$$ The path $\Gamma$ is not divergent, since by
\eqref{eqn:c-length-estimate} this would contradict completeness.
Thus along the path there is a sequence $z_n(\zeta_n) \to z_0$ with $\zeta_n
\to \zeta_0$.  But $F'(z_0) \neq 0$, allowing extension of
$z(\zeta)$ over $\zeta_0$, a contradiction.

Thus $F^{1/(N+1)}$ is entire and invertible, hence linear, making $F$
a polynomial.  Thus $G$ is constant, and $C$ is also a polynomial.
\end{proof}

By Theorem \ref{thm:polygon}, the degree $d$ of the polynomial $C(z)$
is $(n-3)$, completing the proof of Theorem \ref{thm:polynomial}.
\end{proof}

\section{Mapping of moduli spaces}
\label{sec:mapping}

The two preceding sections show that a complete hyperbolic affine
sphere is asymptotic to a polygon if and only if it has conformal type
$\C$ and polynomial Pick differential, and that all polynomials arise
in this way from polygons.  In this section we combine and extend
these results to prove the main theorem (also relying on the Cheng-Yau
theorem and the results of Section \ref{subsec:continuity} on
continuous variation of solutions to the vortex equation).
Precisely, we show:

\begin{thm}
\label{thm:main-in-text}
For any integer $d \geq 0$, the construction of an affine sphere with
polynomial Pick differential given by Theorem
\ref{thm:polynomial-sphere-existence} induces a $\Z/(d+3)$-equivariant
homeomorphism 
$$ \alpha : \NormCubic_d \to \NormPoly_{d+3},$$
and thus also a quotient homeomorphism $\boldalpha : \ModCubic_d \to \ModPoly_{d+3}$.
\end{thm}

\begin{proof}
The proof will proceed in several steps.

\boldpoint{Construction of the map.}  Let $C \in \NormCubic_d$ be a
normalized polynomial cubic differential of degree $d$.  By Theorem
\ref{thm:polynomial-sphere-existence} we have a complete conformally
parameterized affine sphere $f_C^0 : \C \to \R^3$ with Pick
differential $C$.  By Theorem \ref{thm:polygon}, the projectivized
image of this affine sphere is a polygon $P_C^0$ and the rays of
$\star_d$ map to curves that limit projectively to the vertices of
$P_C^0$.  Thus the counterclockwise order of the edges of $\star_d$,
starting from $\R^+$, induces a labeling of the vertices of $P_C^0$ by
$(p_1, \ldots, p_{d+3})$.  Let $A$ be a projective transformation that
normalizes this polygon at $p_1$, i.e.~mapping $(p_1,p_2,p_3,p_4)$ to
$(q_1,q_2,q_3,q_4)$.  Then $P_C := A \cdot P_C^0 \in \NormPoly_{d+3}$
is the projectivized image of the conformally parameterized affine
sphere $f_C := A \cdot f_C^0$ (which still has Pick differential $C$).
Define
$$\alpha(C) := P_C.$$

To summarize, $\alpha(C)$ is the polygon obtained by solving the
vortex equation with cubic differential $C$, integrating to obtain an
affine sphere with vertices naturally labeled by the $(d+3)$-roots of
unity, and then adjusting by a projective transformation to normalize
the polygon at the first vertex.

\boldpoint{Equivariance.}  Let $\zeta = \exp(2 \pi i/(d+3))$
be the generator of the group $\unity_{d+3}$
of $(d+3)$-roots of unity and denote its action on
$\NormCubic_d$ by pushforward through $z \mapsto \zeta z$ by $C
\mapsto \zeta \cdot C$.

Since the Pick differential of $f_C(\zeta z)$ is the pullback
$\zeta^{-1} \cdot C$, we find that $f_{\zeta \cdot C}(\zeta z)$ has
Pick differential $\zeta^{-1} \zeta \cdot C = C$, so by the uniqueness
part of Theorem \ref{thm:polynomial-sphere-existence} there exists $A
\in \SL_3\R$ such that
$$f_{\zeta \cdot C}(\zeta z) = A \cdot f(z).$$
Note that $z \mapsto \zeta z$ permutes the rays of $\star_d$, acting
as a $(d+3)$-cycle.  Thus up to projective transformations the
normalized polygons $\alpha(C)$ and $\alpha(\zeta \cdot C)$ are the
same, but under this isomorphism the labeling of their vertices by $1,
\ldots, (d+3)$ is shifted by one.  This is the definition of the
action of $\rot$, the generator of the $\Z/(d+3)$ action on
$\NormPoly_d$, and so $\alpha$ is equivariant.

It follows that $\alpha$ induces a map $\boldalpha : \ModCubic_d \to
\ModPoly_{d+3}$, and that for any $C \in \Cubic_d$ the image
$\alpha([C])$ is simply $\SL_3\R$ equivalence class of the
projectivized image of $f_C$ or $f_C^0$ (i.e.~in describing the quotient map, no
normalization is required).

\boldpoint{Injectivity.}  If $\alpha(C) = \alpha(C') = P$ then the
uniqueness part of the Cheng-Yau theorem (\ref{thm:cheng-yau})
shows that the associated affine spheres coincide, and so the two
conformal parameterizations $f_C$, $f_{C'}$ are related by an
automorphism of $\C$.  Both $C$ and $C'$ are normalized, so this
automorphism must be multiplication by a $(d+3)$-root of unity, which
permutes the rays of $\star_d$.  Since $\alpha(C) = \alpha(C')$, we
also have that $f_C$ and $f_{C'}$ induce the same map from the rays of
$\star_d$ to vertices of $P$.  Hence this permutation must be trivial
and the automorphism is the identity, i.e.~$C=C'$.

\boldpoint{Surjectivity.}  By equivariance it is enough to check that
$\boldalpha$ is surjective.  Let $[P] \in \ModPoly_{d+3}$.  We can
choose a representative oriented polygon $P$ that lies in a fixed
affine chart $\C \simeq \R^2 \simeq \{(x_1, x_2, x_3) \suchthat x_3 =
1 \} \subset \R^3$ of $\RP^2$ in such a way that the orientation of
$P$ agrees with that of $\C$.

By the Cheng-Yau existence theorem (\ref{thm:cheng-yau}) we have a
complete hyperbolic affine sphere $M_P$ asymptotic to the cone over
$P$.  By Theorem \ref{thm:polynomial}, there is a conformal
parameterization of $M_P$ by $\C$ such that Pick differential is a
polynomial cubic differential $C$.  Using Theorem
\ref{thm:polynomial-sphere-existence} we obtain another parameterized
affine sphere $M_C$ with Pick differential $C$, which by definition
has projectivization representing $\boldalpha([C])$.

By the uniqueness part of Theorem
\ref{thm:polynomial-sphere-existence} there is an element of $\SL_3\R$
mapping $M_C$ to $M_P$, thus identifying their oriented
projectivizations and giving $\boldalpha([C]) = [P]$.  Hence
$\boldalpha$ is surjective.

\boldpoint{Continuity.}  For $d \leq 1$ all of the moduli spaces in
question are finite sets with the discrete topology, and there is
nothing to prove.  We assume for the rest of the proof that $d > 1$.

First we consider a map related to $\alpha$ in which the polygon is
normalized in a different way.  For any $C \in \NormCubic_d$ we can
compose $f_C$ with an element of $\SL_3\R$ so that its complexified
frame at $0 \in \C$ agrees with that of the normalized \titeica
surface.  Let $f_C^\circ : \C \to \R^3$ denote the resulting map; note
that the projectivized image of $f_C^\circ$ is a polygon, but not
necessarily a normalized one.  We have the associated map
\begin{equation}
\begin{split}
\alpha^\circ : \NormCubic_d &\to \Poly_{d+3},\\
C &\mapsto \P(f_{C}^\circ(\C)).
\end{split}
\end{equation}

The advantage of working with $\alpha^\circ$ is that the shared frame
at the origin means that developing maps $\{ f^\circ_C \}_{C \in
  \NormCubic_d}$ are solutions of a fixed initial value problem for
the system of ODEs \eqref{eqn:structure} where only the
coefficients of the system are varying.  In contrast, the maps $\{
f_{C} \}$ have a shared normalization only ``at infinity''.

Using Theorem \ref{thm:continuity} we will now show that
$\alpha^\circ$ is continuous with respect to the Hausdorff topology.

Fix $C \in \NormCubic_d$ and $\epsilon > 0$.  We must find a
neighborhood $V$ of $C$ in $\NormPoly_{d+3}$ such that for all $C' \in
V$ we have $\alpha^\circ(C) \subset N_{\epsilon}(\alpha^\circ(C'))$
and $\alpha^\circ(C') \subset N_{\epsilon}(\alpha^\circ(C))$.

Let $P = \alpha^\circ(C)$.  First select a radius $R$ large enough so
that so that $P \subset N_{\epsilon/2}(f_C^\circ(B_R))$ where $\bar{B}_R =
\{ |z| \leq R \}$.  Now, by Theorem \ref{thm:continuity} we can ensure
that the Blaschke metric densitities of $C$ and $C'$ are arbitrarily
close in $C^1(\bar{B}_R)$ by making the coefficients of $C$ and $C'$
sufficiently close.  By \eqref{eqn:structure} this shows that the
coefficients of the respective connection forms can be made uniformly
close ($C^0$), and applying continuous dependence of solutions to ODE
initial value problems gives the same conclusion for $f_C^\circ,
f_{C'}^\circ$ (and moreover, for their respective frame fields
$F_C^\circ$, $F_{C'}^\circ$).  In particular we can choose a
neighborhood of $C$ so that $\P(f_{C'}^\circ)$ is $\epsilon/2$-close
to $\P(f_{C}^\circ)$ in $\bar{B}_R$, giving
$$\alpha^\circ(C) = P \subset N_{\epsilon}(f_C'(\bar{B}_R)) \subset N_{\epsilon}(\alpha^\circ(C'))$$
for all $C'$ in this neighborhood.

The ``outer'' continuity follows similarly by
considering tangent planes to the affine sphere.  Recall that the
image of $f_C^\circ$ is strictly convex and asymptotic to the boundary
of the cone over $P$, which is a polyhedral cone with $d+3$ planar
faces (corresponding to the edges of $P$).  
Thus, we may approximate each
of the planes in $\R^3$ containing one of these faces as closely as we wish by the tangent plane to an appropriately chosen point on $f_C^\circ$.
Selecting one
such point for each face---call these ``sample points''---the tangent
planes become lines in $\RP^2$ which determine a $(d+3)$-gon $\hat{P}$
that approximates $P$.  Moreover, by convexity $f_C^\circ$ lies above
its tangent planes, which means that $P$ lies
inside $\hat{P}$.  We call $\hat{P}$ the \emph{outer polygon}
determined by the sample points.

Choose $R>0$ so that $\bar{B}_R$ contains a set of $(d+3)$ sample points for
which the outer polygon approximates $P$ well enough that 
$\hat{P} \subset N_{\epsilon/2}(P)$.  As above we conclude from Theorem
\ref{thm:continuity} that any $C'$ close to $C$ determines a frame
field $F_{C'}^\circ$ that is uniformly close to $F_{C'}^\circ$ on
$\bar{B}_R$, and in particular the tangent planes to $f_{C'}^\circ$
approximate those of $f_{C}^\circ$.  Thus there is a neighborhood of
$C$ in which the outer polygon for $C'$ lies in a
$\epsilon/2$-neighborhood of that for $C$, giving
$$\alpha^\circ(C') \subset N_{\epsilon/2}(\hat{P}) \subset
N_{\epsilon}(P) = N_\epsilon(\alpha^\circ(C)),$$
for all $C'$ in this neighborhood.

Thus $\alpha^\circ$ is continuous with respect to the Hausdorff
topology.  By Proposition \ref{prop:hausdorff-vertex-equivalent} this
implies that $\alpha^\circ$ is also continuous with respect to the
usual vertex topology on $\Poly_{d+3}$.

Finally, we return to the original map $\alpha$: Since the polygon
$\alpha(C)$ is simply the normalization of $\alpha^\circ(C)$, the
continuous variation of the vertices of $\alpha^\circ(C)$ implies that
the projective transformations that accomplish this normalization also
vary continuously.  Hence the map $\alpha$ is the composition of
$\alpha^\circ$ with a continuous family of projective transformations,
and hence also continuous.

\boldpoint{Continuity of the inverse.}
At this point we have shown that $\alpha$ is a continuous bijection
between spaces homeomorphic to $\R^N$.  To establish continuity of
$\alpha^{-1}$ we need to show that $\alpha$ is closed.  Since proper
continuous maps on locally compact spaces are closed, it suffices to
show that $\alpha$ is a proper map.

Suppose for contradiction that $\alpha$ is not proper, i.e.~that there
exists a sequence $C_n \to \infty$ in $\NormCubic_{d}$ so that
$\alpha(C_n) \to P$.  Write $P_n = \alpha(C_n)$.

Let $Z_n \subset \C$ denote the set of roots of the polynomial $C_n$.
Note that the cardinality of $Z_n$ is at most $d$.  Since $C_n$ is
monic and centered, a bound on the diameter of the set $Z_n$ (with
respect to the Euclidean metric of $\C$) would give a bound on all
coefficients of $C_n$.  Since $C_n \to \infty$, no such bound applies,
and we find that the Euclidean diameter of $Z_n$ is unbounded as $n
\to \infty$.  Replacing $C_n$ with a subsequence, we assume from now
on that the diameter of $Z_n$ actually tends to infinity.

Let $r_n$ denote the diameter of the set $Z_n$ with respect to the
flat metric $|C_n|^{2/3}$ of the cubic differential.  We claim that
$r_n$ tends to infinity.  To see this, consider paths in $\C$ that
connect all of the zeros of $C_n$, i.e.~maps $[0,1] \to \C$ such that
$Z_n$ is a subset of the image.  We call these \emph{spanning paths}.
Ordering the elements of $Z_n$ and connecting them in order by
$|C_n|^{2/3}$-geodesic segments gives a spanning path of length at most
$(d-1)r_n$.  We will show that the minimum length of a spanning path
tends to infinity, and hence that $r_n \to \infty$ as well.

Suppose that there exists a spanning path whose Euclidean distance
from $Z_n$ is never greater than $k$.  Then the Euclidean
$k$-neighborhood of $Z_n$ is connected and hence the set $Z_n$ has
Euclidean diameter at most $2kd$.  Since this diameter tends to
infinity, we find that for large $n$, any spanning path contains a
point that is very far from $Z_n$ in the Euclidean sense.  Since
$|C_n(z)| > 1$ whenever $d(z,Z_n) > 1$ (by monicity), this also shows
that the minimum $|C_n|^{2/3}$-length such a path diverges as $n \to
\infty$, as required.

On the other hand, Lemma \ref{lem:compact-containing-zeros} gives a
compact set $\hat{K}$ in the interior of $P$ that contains the zeros
of the Pick differential for all polygons in a neighborhood $U$ of $P$.  Since $\alpha(C_n)
\to P$, for large $n$ we have $\alpha(C_n) \in U$.  Since the
restrictions of the Pick differential metrics to $\hat{K}$ vary continuously in the
Hausdorff topology (by Corollary
\ref{cor:blaschke-and-pick-continuous}), the diameter of $\hat{K}$ in
the Pick differential metric of $\alpha(C_n)$ is bounded as $n \to \infty$, as
is the sequence $r_n$.  This is the desired contradiction.

We conclude that the continuous bijection $\alpha$ is proper, and so
$\alpha^{-1}$ is continuous.
\end{proof}

This completes the proof of Theorem \ref{introthm:main} from the introduction.

\section{Complements and conjectures}
\label{sec:conjectures}

In this final section we discuss alternative approaches to some of the
results proved above and a few directions for further work related to
the homeomorphism $\boldalpha : \ModCubic_d \to \ModPoly_{d+3}$.

\subsection{Continuity method}

By the Invariance of Domain Theorem, a continuous, locally injective,
and proper map between manifolds of the same dimension is a
homeomorphism.  Using this to establish that a map is homeomorphic is
sometimes called the ``continuity method''.

Since we establish continuity, injectivity, and properness of
the map $\alpha$, the continuity method could be applied to show that
it is a homeomorphism.  Such an approach would obviate the construction of the inverse
map $\alpha^{-1}$ and the need to establish surjectivity of $\alpha$,
giving a slightly shorter proof of the main theorem.

We prefer the argument given in Section \ref{sec:mapping} because it
highlights the way in which existence theorems for the vortex equation
and the Monge-Ampere equation (i.e.~the Cheng-Yau theorem) give rise
to mutually inverse maps between moduli spaces.  Also, while the
continuity method typically gives only indirect information about the
properties of the inverse map, we hope that the explicit constructions of
both $\alpha$ and $\alpha^{-1}$ will be helpful toward further study
of the local or differential properties of this isomorphism of moduli
spaces.

\subsection{Hilbert metric geometry}

There is a classical projectively invariant Finsler metric on convex
domains in $\RP^n$, the \emph{Hilbert metric}, which is defined using
projective cross-ratios, generalizing the Beltrami-Klein model of the
hyperbolic plane (see e.g.~\cite{busemann-kelly}).  This metric is
Finsler and is not Riemannian unless the domain is bounded by a conic
\cite{kay:ptolemaic-inequality}.  Thus, in general the Hilbert metric
is quite different from the Blaschke and Pick differential (Riemannian)
metrics considered above.

However, Benoist and Hulin used the Benz\'{e}cri cocompactness theorem
in projective geometry to show that the Hilbert metric of a properly
convex domain in $\RP^2$ is uniformly comparable to the Blaschke
metric, in the sense that the ratio of their norm functions is bounded
above and below by universal constants
\cite[Prop.~3.4]{benoist-hulin:finite-volume}.  The same arguments
show that some multiple of the Hilbert metric gives an upper bound on
the Pick differential metric.

Therefore, in any instance where coarse geometric properties of the
Blaschke metric are considered, this comparison principle would allow
one to work instead with the Hilbert metric.  Since the Hilbert
metrics of polyhedra have been extensively studied (e.g.~in
\cite{de-la-harpe:simplices} \cite{foertsch-karlsson} \cite{bernig}
\cite{colbois-verovic} \cite{colbois-verovic-vernicos}), one might ask
whether such results could be brought to bear on the study of
polygonal affine spheres.  We mention here only one result in this
direction, an alternative proof of a weaker form of Theorem
\ref{thm:polygon}:

\begin{thm}
\label{thm:polygon-weak}
Suppose $M \subset \R^3$ is an complete affine sphere conformally equivalent to $\C$ and
having polynomial Pick differential $C$.  Then $M$ is asymptotic to the cone over
a convex polygon.
\end{thm}

\begin{proof}
Colbois and Verovic showed that a convex domain in $\RP^n$ whose
Hilbert metric is quasi-isometric to a normed vector space (or even
which quasi-isometrically embeds in such a space) is a convex
polyhedron \cite{colbois-verovic}.  Applying the $n=2$ case of this
theorem, we can then conclude that the projectivization of $M$ is a
convex polygon if we show that its Hilbert metric is quasi-isometric
to the Euclidean plane.

For any polynomial cubic differential $C$, the singular flat
metric $|C|^{2/3}$ is quasi-isometric to the Euclidean plane.  Hence
the Pick differential metric of $M$ has this property.

Theorem \ref{thm:stronger-polynomial-sphere-uniqueness} shows that the
Blaschke metric of $M$ comes from a solution of the vortex equation
for the polynomial Pick differential, whereupon Corollary
\ref{cor:coarse-bound} implies that the Pick differential and Blaschke metrics
of $M$ are quasi-isometric.  Hence the Blaschke metric of $M$ is
quasi-isometric to the plane.

Since the Hilbert metric is bilipschitz to the Blaschke metric, it too is
quasi-isometric to the plane, as required.
\end{proof}

Note that this argument does not relate the number of vertices of the
polygon to the degree of the polynomial.  It would be interesting to
know if these Hilbert-geometric techniques could be pushed further to
give a complete proof of Theorem \ref{thm:polygon}.

\subsection{Pick zeros and the Fence conjecture}

By Theorem \ref{thm:polynomial}, each convex polygon $P$ in $\RP^2$
with $n$ vertices is associated to an affine sphere whose Pick differential is
a polynomial of degree $(n-3)$.  The Pick zeros therefore give a
projective invariant of $P$ whose value is a set of $(n-3)$ interior
points counted with multiplicity.  While constructed through
transcendental and analytic methods, it would be interesting to
understand whether any properties of the Pick zeros can be related
directly to the projective or algebraic geometry of the polygon $P$.

We will state one conjecture in this direction about bounding the Pick
zeros in terms of diagonals of $P$.  To formulate it, we first recall
the compact region constructed in Section \ref{sec:polygons-to-polynomials} which
contains all of the Pick zeros.  Corresponding to each vertex $v$ of
$P$ there is a smooth arc inside $P$ (a barrier curve) which joins the
neighbors $v_-,v_+$ of $v$ and which lies inside the triangle
$\triangle v v_- v_+$.  We cut $P$ along these arcs, each time
discarding the region on the same side of the barrier as $v$.  What
remains is the \emph{core}.

Some simple computer experiments, in which the Pick zeros of some
families of convex $n$-gons were computed for $n \leq 7$, suggest that
it might be possible to replace the barrier curve at $v$ with the line
segment $[v_-,v_+]$ and still bound the Pick zeros.  That is, cutting
away from $P$ each of the triangles formed by a consecutive triple of
vertices, we obtain a smaller convex $n$-gon, which we call the
\emph{fence}, and we conjecture:

\begin{conj}[Fence conjecture]
\label{conj:fence}
For any convex polygon in $\RP^2$ with $n \geq 5$ vertices, the Pick
zeros lie inside the fence.  Equivalently, the Pick differential of the affine
sphere over a convex polygon has no zeros over the vertex inscribed
triangles.
\end{conj}

\begin{figure}
\begin{center}
\includegraphics[width=\textwidth]{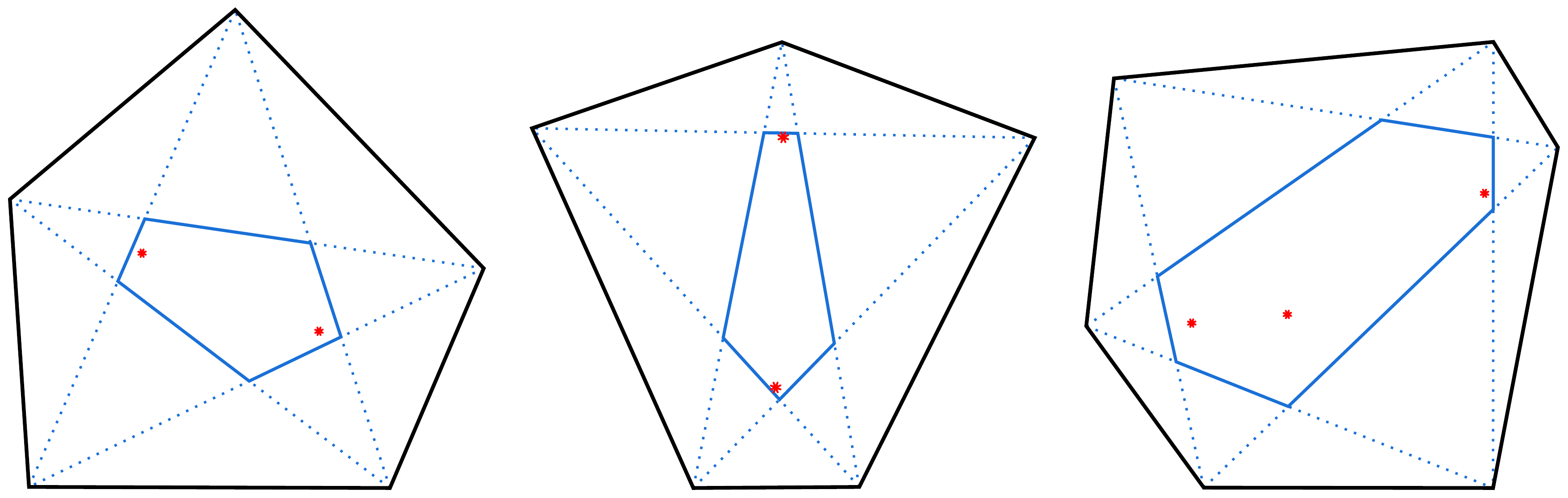}
\caption{Pick zeros and fences for some convex polygons, computed by
  numerical solution of the vortex equation and integration of the
  affine frame field.}
\label{fig:fence}
\end{center}
\end{figure}

The restriction to $n \geq 5$ only excludes cases in which the
conjecture is vacuous or trivial: For triangles there are no Pick
zeros.  For quadrilaterals, the fence always reduces to a point, which
is the intersection of the two diagonals.  All convex quadrilaterals
are projectively equivalent, and symmetry considerations show the
intersection of diagonals is also the unique zero of the Pick differential
(for which we could choose $C=z dz^3$ as a representative).

Finally we note that the fence conjecture can be seen as a limiting
case of the construction of the barrier curves in Section
\ref{sec:polygons-to-polynomials}: That construction involves the
choice of a real constant $R$ so that the barrier curves 
correspond to $\{ \Re(z) = R \}$ in the \titeica affine
spheres over the inscribed triangles.
Lemma \ref{lem:locally-bilip-half-planes} implies that any
sufficiently large $R$ can be used.  On the other hand, the fence is
obtained as the limit of these curves when $R$ tends to $-\infty$.
Thus the fence conjecture would follow if the hypothesis $\{
\Re(z) > R \}$ could be dropped in Lemma
\ref{lem:locally-bilip-half-planes}.

\subsection{Differentiability, Poisson structures, and flows}

The spaces $\NormCubic_d$ and $\NormPoly_{d+3}$ are smooth manifolds,
and we have shown that the map $\alpha$ is a homeomorphism between
them.  We expect that $\alpha$ has additional regularity:

\begin{conj}
\label{conj:diffeo}
The map $\alpha : \NormCubic_d \to \NormPoly_{d+3}$ is a diffeomorphism.
\end{conj}

Of course, since $\alpha$ is $\Z/(d+3)$-equivariant, the conjecture is equivalent to the
statement that the quotient map $\boldalpha : \ModCubic_d \to
\ModPoly_{d+3}$ is a diffeomorphism of orbifolds.

The differentiability of $\alpha$ itself would follow from a
sufficiently strong estimate concerning the smooth dependence of the
solution to Wang's equation on the holomorphic cubic differential.
While estimates of this type are routine when considering a fixed
compact subset of the domain, the global nature of the map $\alpha$
would seem to require more control.  For example, the constructions of
Section \ref{sec:polynomials-to-polygons} show that the vertices of
the polygon $\alpha(C)$ are determined by fine limiting behavior of
the Blaschke metric at infinity, through the unipotent factors
constructed in Lemma \ref{lem:unipotent-factors}.

Similarly, the differentiability of $\alpha^{-1}$ might be established
by studying the dependence of the $k$-jet of the Blaschke metric at a
point of a convex polygon as a function of the vertices, generalizing
Theorem \ref{thm:support-function-continuous}.  However, one would
also need to control the variation of the uniformizing coordinate $z$
in which the Pick differential becomes a polynomial.

Assuming for the moment that $\alpha$ is a diffeomorphism, several
questions arise about its possible compatibility with additional
differential-geometric structures of its domain and range.

For example, in addition to its complex structure, the space
$\ModCubic_d$ carries a holomorphic action of $\C^*$, which is the
quotient of the action on polynomials by scalar multiplication.
Restricting to the subgroup $S^1 = \{ e^{i \theta} \} \subset \C^*$
gives a flow on $\ModCubic_d$ with closed leaves, which we call the
\emph{circle flow}.  Since the Wang equation involves the cubic
differential only through its norm, the Blaschke metric (as a function
on $\C$) is constant on these orbits, and the associated affine
spheres are intrinsically isometric.  The extrinsic geometry is
necessarily changing, however, since the image of a (nontrivial)
$S^1$-orbit by $\boldalpha$ is a circle in $\ModPoly_{d+3}$.

It is natural to ask whether the images of circle orbits in
$\ModPoly_{d+3}$ could be recognized in terms of intrinsic features of
that space, or in terms of projective geometry of polygons, without
direct reference to the map $\boldalpha$.  

For example, there is a natural Poisson structure on a space closely related
to $\ModPoly_k$: Define a \emph{twisted polygon} with $k$ vertices to
be a map $P : \Z \to \RP^2$ that conjugates the translation $i \mapsto
(i+k)$ of $\Z$ with a projective transformation $M \in \SL_3\R$.  Here
we say $M$ is the \emph{monodromy} of the twisted polygon.  The space
$\Tilde{\ModPoly}_k$ of $\SL_3\R$-equivalence classes of twisted
$k$-gons is a real algebraic variety which is stratified by conjugacy
classes of the monodromy, and which contains $\ModPoly_k$ as the
stratum with trivial monodromy. The variety $\Tilde{\ModPoly}_k$ is
smooth and of dimension $2k$ in a neighborhood of $\ModPoly_k$.

Ovsienko, Schwartz, and Tabachnikov introduced in \cite{OST} a natural
Poisson structure on $\Tilde{\ModPoly}_k$ as part of their study of
the \emph{pentagram map}, a dynamical system on polygons and twisted
polygons which, in our terminology, maps a convex polygon to its
fence.  We wonder if this same Poisson structure could be related to
the circle flow considered above, and in particular if the circle flow
is defined by a Hamiltonian.  More precisely, we ask:

\begin{question}
Conjugating the circle flow on $\ModCubic_d$ by $\boldalpha$ we obtain
a flow on $\ModPoly_{d+3}$.  Is it the restriction of
Hamiltonian flow on $\Tilde{\ModPoly}_{d+3}$?
\end{question}

We remark that recent work of Bonsante-Mondello-Schlenker give an
affirmative answer to a seemingly analogous question for quadratic
differentials on compact surfaces: In \cite{bms1} they introduce a
circle flow on $\sT(S) \times \sT(S)$ that is induced by the $e^{i
  \theta}$ multiplication of holomorphic quadratic differentials and a
harmonic maps construction that involves the $k=2$ case of the vortex
equation \eqref{eqn:vortex} from Section \ref{sec:vortex}.
In \cite{bms2} it is shown that this \emph{landslide flow} is
Hamiltonian for the product of Weil-Petersson symplectic structures.

We close with a final conjecture about the circle flow which was
suggested by computer experiments in the pentagon case ($d=2$).

\begin{figure}
\begin{center}
\includegraphics[width=0.7\textwidth]{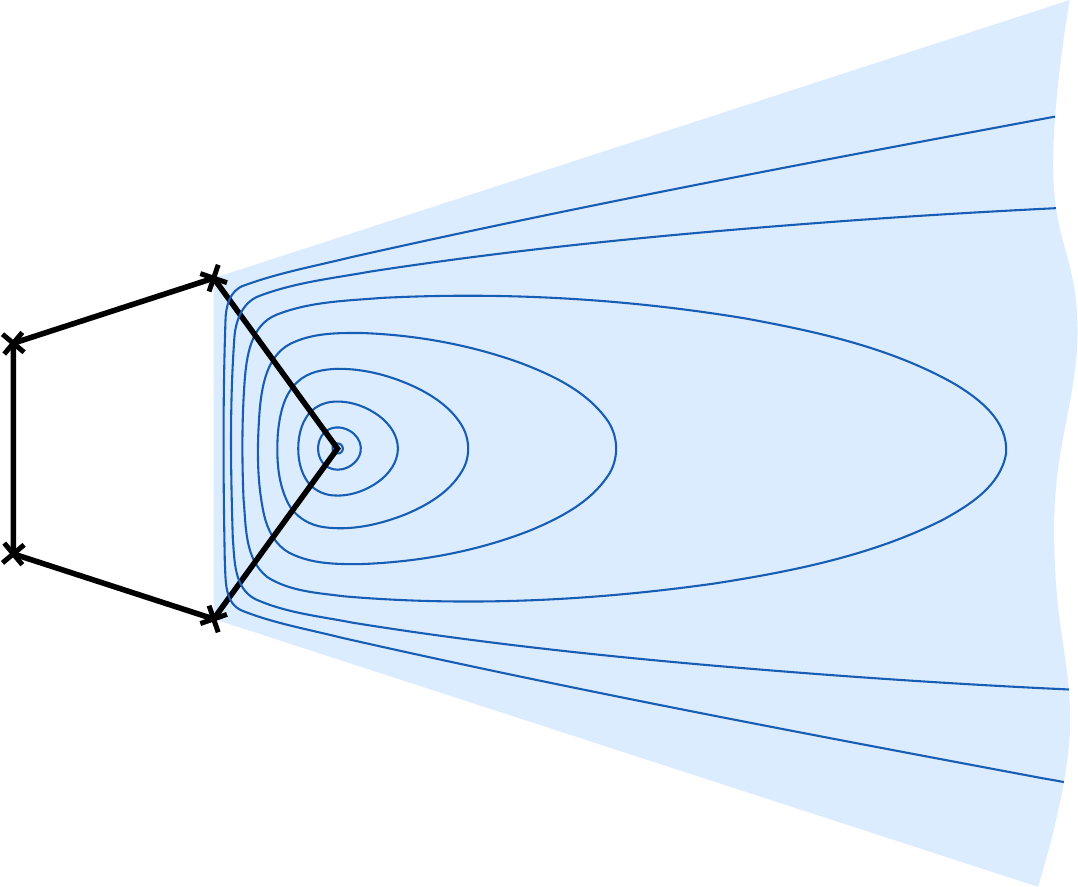}
\caption{Stratification of $\NormPoly_5$ by level sets of the product
  of corner invariants.  Here $\NormPoly_5$ is identified with the set
  of possible locations for a fifth vertex when the first four are
  fixed (the marked points).  Thus $\NormPoly_5$ is a triangle in $\RP^2$ (the shaded region,
  which extends beyond this affine chart).}
\label{fig:circles}
\end{center}
\end{figure}

Associated to each vertex $v$ of a polygon $P$ in $\RP^2$ there is a
projective invariant $x_v = x_v(P) \in \R$ known as the \emph{corner
  invariant}.  It is defined as follows: Consider the chain of five
consecutive vertices of the polygon in which $v$ is the middle.  Join
$v$ to each other vertex in the chain by lines, obtaining four lines
that are concurrent at $v$.  The cross ratio of these lines is $x_v$.

Schwartz observed in \cite{schwartz:pentagram} that a pentagon in
$\RP^2$ is uniquely determined up to projective transformations by its
corner invariants, and more generally that the product $X = \prod_v
x_v$ of the corner invariants is a ``special'' function on the space
of pentagons.  For example, this function is invariant under the
pentagram map and it has a unique minimum at the regular pentagon.
The non-minimal level sets of the function $X$ foliate the rest of
$\ModPoly_{5}$ by real-algebraic curves homeomorphic to $S^1$.

Our computational experiments suggest that these curves are images of
orbits of the circle flow on $\ModCubic_2$:

\begin{conj}
The map $\boldalpha : \ModCubic_2 \to \ModPoly_5$ sends each orbit of
the circle flow to a level set of the product of the corner
invariants.
\end{conj}

The corner invariants and stratification by level sets of $X$ can also
be defined on the manifold cover $\NormPoly_5$, and we recall from
Section \ref{sec:polygons} that this space is naturally a triangle in
$\RP^2$.  The corresponding stratification of this triangle is shown
in Figure \ref{fig:circles}.

Generalizing the previous conjecture, it would be interesting to know
whether there exist any nontrivial circle orbits in $\ModCubic_d$ that
map to real-algebraic curves in $\ModPoly_{d+3}$ for $d > 2$, or more
generally whether these circle orbits are contained in real-algebraic
subvarieties of positive codimension.  Positive answers would evince a
compatibility between $\boldalpha$ and the algebraic structure of
$\ModPoly_{d+3}$.

\appendix

\section{Existence of standard half-planes}
\label{appendix:half-planes}

In this appendix, we construct the half-plane subdomains we need for
our analysis of the large scale geometry of the affine spheres in
Section~\ref{sec:polygons-to-polynomials} and of the basic decay
estimates for the general vortex equation in
Section~\ref{subsec:estimates}.
We begin with the case of finding half-planes for cubic differentials
and then generalize the argument to holomorphic differentials of any
order.

\subsection{Half-planes for Cubic Differentials.}
Recall the precise statement from Section \ref{sec:cubic}:

\begingroup
\def\thethm{\ref{prop:standard-half-planes}}
\begin{prop}[Standard half-planes]
Let $C$ be a monic polynomial cubic differential.  Then there are
$(d+3)$ $C$-right-half-planes $\{ (U_k, w_k) \}_{k=0,\ldots,d+2}$ with the
following properties:
\begin{rmenumerate}
\item The complement of $\bigcup_k U_k$ is compact.

\item The ray $\{ \arg(z) = \frac{2 \pi k}{d+3} \}$ is
eventually contained in $U_k$.

\item The rays $\{ \arg(z) = \frac{2 \pi (k\pm 1)}{d+3} \}$ 
are disjoint from $U_k$.

\item On $U_k \cap U_{k+1}$ we have $w_{k+1} = %\exp(-2 \pi i/3) 
\omega^{-1}
w_k +
c$ for some constant $c$, and each of $w_k, w_{k+1}$ maps this
intersection onto a sector of angle $\pi/3$ based at a point on $i \R$. (Recall $\omega = \exp(2
\pi i/3)$.)

\item Each ray of $\star_d$ is a $C$-quasi-ray of angle zero in the
associated half-plane $U_k$.  More generally any Euclidean ray in $\C$
is a $C$-quasi-ray and is eventually contained in $U_k$ for some $k$.
\end{rmenumerate}
\end{prop}
\addtocounter{thm}{-1}
\endgroup

\begin{proof}
The point of the proof is to treat $C$ as a small deformation
of $z^d$, and to construct half-planes for $C$ as small deformations of
the ones described above for $z^d$.  

We construct $U_k$ and then verify its properties.  Define
\begin{equation}
\label{eqn:zeta}
 \zeta_k = \tfrac{3}{d+3} z^{\frac{d+3}{3}} \exp \left ( \tfrac{2 \pi
  i k}{d+3} \right ).
\end{equation}
Here we use the principal branch of the
logarithm to define this fractional power of $z$, so $\{ \zeta_k \in
\R^+ \}$ corresponds to $\{ \arg{z} = 2 \pi k/(d+3) \}$, and so that
$\zeta_k$ is a conformal coordinate on a sector centered at $\{
\arg{z} = 2 \pi k/(d+3) \}$ mapping it to $\C \setminus \R^-$.

For the moment we fix $k$ and for brevity write $\zeta = \zeta_k$.
Observe that $z^d dz^3 = d\zeta^3$.  For $C = C(z) dz^3$ where $C(z)$
is a general monic polynomial of degree $d$, we instead have
$$C = \left ( 1 + O\left (|\zeta|^{-\frac{3}{d+3}} \right) \right) d\zeta^3 $$
where the implicit constant depends on $C$ but can be made uniform
if an upper bound is imposed on the coefficients of the polynomial.

Restricting attention to $|\zeta|$ large enough so that
$C$ is nonzero, we find that $C$ has holomorphic cube root of the form
\begin{equation}
\label{eqn:cube-root}
\sqrt[3]{C} = \left( 1 + O\left (|\zeta|^{-\frac{3}{d+3}} \right) \right)d\zeta
\end{equation}

Fix a small $\epsilon > 0$.  For any $s \gg 0$, consider the region
$$ \Omega_{\epsilon,s} = \left \{ \zeta \suchthat \arg(\zeta - s) \in
\left (-\tfrac{\pi}{2}-\epsilon, \tfrac{\pi}{2} +\epsilon \right)  \right\} $$
which is a ``slightly enlarged $C$-right half-plane'' with $s$ on its
boundary. 

This domain has the following properties:
\begin{itemize}
\item It is nearly convex, i.e.~any pair of points $x,y \in
\Omega_{\epsilon,s}$ are joined by a path $\gamma \subset
\Omega_{\epsilon,s}$ of length $|\gamma| \leq L|x-y|$, for some $L =
L(\epsilon) > 1$, and 
\item The real part of $\zeta$ approaches $-\infty$ at a linear rate
on the boundary, i.e.~if $\zeta \in \partial \Omega_{\epsilon,s}$ and
$|\zeta|$ is large enough, then $-\Re(\zeta) \geq c |\zeta|$ for some
$c>0$.
\item It is far from the origin, i.e.~the minimum of $|\zeta|$ on
$\Omega_{\epsilon,s}$ is $s \cos(\epsilon) \gg 0$.
\end{itemize}
Since $\tfrac{d}{d+3} < 1$, the estimate \eqref{eqn:cube-root} and the
last property above show that for any $\delta > 0$ we can choose $s$
large enough so that
\begin{equation}
\label{eqn:ratio-almost-one}
 \left | \frac{\sqrt[3]{C}}{d \zeta} - 1 \right | < \delta
 \end{equation}
throughout $\Omega_{\epsilon,s}$.

Now integrate $\sqrt[3]{C}$ as in \eqref{eqn:developing-map} to get a
natural coordinate $w$ for $C$.  The estimate above shows that $w$ is
approximately a constant multiple of $\zeta$.  For example, it follows
easily from this bound and the near-convexity of $\Omega_{\epsilon,s}$
that $w$ is injective on $\Omega_{\epsilon,s}$ as long as
$L \delta < 1$.  Fix $s$ large enough so that this holds.

The linear growth of $-\Re(\zeta)$ on $\partial \Omega_{\epsilon,s}$
and the sublinear bound on $(\sqrt[3]{C} / d \zeta - 1)$ from \eqref{eqn:cube-root}
also show that
that $\Re(w)$ is bounded from above on $\partial
w(\Omega_{\epsilon,s})$, and therefore that $w(\Omega_{\epsilon,s})$
contains a half-plane $\{ \Re(w) > t \}$.  Let $U_k^{(t)}$ denote the
region in the $z$-plane corresponding to $\{ \Re(w) > t \}$, and let
$w_k = w - t$ be the adjusted natural coordinate making
$(U_k^{(t)},w_k)$ into a $C$-right-half-plane.

Applying \eqref{eqn:cube-root} again we can estimate the shape of
$U_k^{(t)}$ in the $\zeta$ coordinate: It is a perturbation of a right
half-plane $\{ \Re(\zeta) > c \}$ by $o(|\zeta|)$, and thus for any
$\epsilon' > 0$ it contains all but a compact subset of the sector of
angle $\pi - \epsilon'$ centered on $\{ \zeta \in \R^+\}$.  It also
follows that for $t$ large enough, the set $U_k^{(t)}$ is disjoint
from the rays $\arg(\zeta) = \pm 2 \pi/3$ (which lie in the left
half-plane $\{ \Re(\zeta) < 0\}$).  Fixing such $t$, let $U_k =
U_k^{(t)}$.

Using $z = \left ( \frac{d+3}{3} \zeta \right )^{\frac{3}{d+3}}$, we
have corresponding estimates for the shape of $U_k$ in the $z$
coordinate, which show that it is asymptotic to a sector of angle
$\frac{3 \pi}{d+3}$.  More precisely, for any $\delta' >0$ there
exists $R>0$ so that $U_k$ contains the part of a sector of angle $\tfrac{3
\pi}{(d+3)} - \delta'$ outside the $R$-disk, i.e.
\begin{equation}
\label{eqn:z-sector-almost-contained}
\left \{ |z| > R \text{ and } \left |\arg(z) - \tfrac{2 \pi
  k}{d+3}\right| < \tfrac{3
  \pi}{2(d+3)} - \tfrac{\delta'}{2} \right \} \subset U_k.
\end{equation}
Here $R$ depends on $k$ for the moment, and we also note that $U_k$ is disjoint from $\{ \arg(z) = \tfrac{2 \pi (k \pm 1)}{d+3} \}$.
The latter condition means that $U_k$ is contained in the sector
\begin{equation}
\label{eqn:z-sector-enclosing}
U_k \subset \left \{ |\arg(z) - \tfrac{2 \pi k}{d+3}\right| < \tfrac{2 \pi}{(d+3)} \}.
\end{equation}

Repeating the construction above for each $k$ we obtain $(d+3)$ such
$C$-right-half-planes, and by taking a maximum over radii of excluded
balls, we assume \eqref{eqn:z-sector-almost-contained} holds for a
uniform constant $R$.

To complete the proof we must verify (i)--(v).

Property (ii) is an immediate consequence of
\eqref{eqn:z-sector-almost-contained} and (iii) is
immediate from \eqref{eqn:z-sector-enclosing}.  Taking the union of
\eqref{eqn:z-sector-almost-contained} over $0 \leq k \leq (d+2)$ also
shows $\{ |z| > R \} \subset \bigcup_k U_k$, giving (i).

Now we consider the relation between natural coordinates on $U_k \cap
U_{k+1}$.  Since any two natural coordinates are related by an
additive constant and a power of $\omega$, the ratio $dw_{k+1}/dw_k$
is constant.  To establish (iv) we need only show this constant is
equal to $\omega$.  It is immediate from \eqref{eqn:zeta} that the
coordinates $\zeta_k, \zeta_{k+1}$ satisfy $d\zeta_{k+1}/d\zeta_k =
\omega$ on their common domain.  Since the natural coordinate $w_k$
for $U_k$ satisfies $dw_k = (1 + o(|\zeta_k|) ) d\zeta_k$, we find
that $dw_{k+1} / dw_k$ approaches $\omega$ at infinity, and is
therefore equal to $\omega$ everywhere.

Finally, the $C$-right-half-plane $U_k$ is constructed so that the ray $\arg(z)
= 2 \pi k /(d+3)$ of corresponds to $\{ \zeta \in \R^+ \}$,
and integrating \eqref{eqn:cube-root} along this path shows that
it is a $C$-quasi-ray of angle zero.  Similarly, any ray in $\C$
eventually lies in one of the sectors
\eqref{eqn:z-sector-almost-contained}, and therefore in some $U_k$,
where \eqref{eqn:cube-root} shows it is a $C$-quasi-ray.  Thus
(v) follows.
\end{proof}

\subsection{Half-planes for $k$-differentials.}

We now extend and adapt some of the previous discussion of half-planes
to $k$-differentials $\phi = \phi(z) dz^k$, where $\phi(z)$ is a
polynomial of degree $d$.  These results are used in section
\ref{subsec:estimates}.

Define a \emph{$|\phi|$-upper-half-plane} to be a pair $(U,w)$ where
$U \subset \C$ is an open set and $w : U \to \H$ a conformal map to
the upper half-plane $\H$ such that $|\phi| = |dw|^k$ on $U$.  We will
show that every point in $\C$ that is far enough from the zeros of
$\phi$ lies in such a half-plane.

Note that unlike the discussion for cubic differentials above, the
phase of $\phi$ is ignored here; a $|\phi|$-upper-half-plane is also a
$|e^{i \theta} \phi|$-upper-half-plane.  We are also constructing
\emph{upper} half-planes for the absolute value of a $k$-differential,
rather than the \emph{right} half-planes for a cubic differential that we
did previously.  These different conventions are convenient for the
respective applications of the constructions in the main text.

Define $r : \C \to \R^{\geq 0}$ by
$$ r(p) = \mathrm{d}_{|\phi|}(p,\phi^{-1}(0)) $$
where $\mathrm{d}_{|\phi|}(\param,\param)$ denotes the distance function
associated to the singular flat metric $|\phi|^{2/k}$.  Thus $r$ is
the $|\phi|^{2/k}$-distance to the zeros of $\phi$, or equivalently
the maximal radius of a flat disk that embeds in $(\C,|\phi|^{2/k})$
with center at $p$.

\begin{prop}
\label{prop:k-differential-half-planes}
Let $\phi = \phi(z) dz^k$ be a $k$-differential on $\C$ with $\phi(z)$
a monic polynomial of degree $k$.  Let $K$ be a compact set in the
plane containing the zeroes of $\phi$. Then there are constants
$C,c,R_0$ with $c>0$ so that for any point $p\in \C$ with $r(p) >
R_0$, there exists a $|\phi|$-upper-half-plane $(U,w)$ with $U \cap K
= \emptyset$ such that $\im(w(p)) \geq r(p) - C$.  In addition, on the
boundary of this half-plane we have $r(x) \geq c |\Re(w(x))|$, for $x$
large.
\end{prop}

\begin{proof}
Pulling back by $z \mapsto e^{i \theta} z$ we can reduce to the case
where $p \in \R$, at the cost of replacing the monic polynomial with
one having leading coefficient of unit modulus.  We assume this from
now on.

The basic existence argument is very similar to Proposition
\ref{prop:standard-half-planes}, so we will simply explain what must
be changed.  (The direct translation of that argument will of course give a
$\phi$-right half-plane; at the last step we will rotate by $\tfrac{\pi}{2}$.)  Define
$$ \zeta = \frac{k}{d+k} z^{\frac{d+k}{k}} $$
using the principal branch of the logarithm, so that $z^d dz^k =
(d\zeta)^k$.  Expressing $\phi$ in this coordinate and estimating as
in the proof of Proposition
\ref{prop:standard-half-planes} we find
$$ \phi^{\frac{1}{k}} = e^{i \eta} \left ( 1 + O\left (|\zeta|^{-\frac{k}{d+k}}
\right ) \right ) d\zeta, $$ for some $\eta \in \R$.  Now fix a small positive constant
$\epsilon$.  For any $s>0$ define
$$ \Omega_{s,\epsilon} = \left \{ \zeta \: | \: \arg(\zeta-s) \in \left (
-\tfrac{\pi}{2}-\epsilon, \tfrac{\pi}{2}+\epsilon \right ) \right
\}.$$ For $s$ large enough this region is disjoint from $K$
(hence it contains no zeros of $\phi$) and as before we find that
integration of $\phi^{\frac{1}{k}}$ gives a conformal mapping of $\Omega_{s,\epsilon}$.  More
precisely, defining
$$ w_0(\zeta) = \int_s^\zeta e^{-i \eta} \phi^{\frac{1}{k}} $$
we have $|dw_0|^k = |\phi|$, $w_0(s) = 0$, and the map $w_0$ is a small perturbation of a
translation, i.e.
\begin{equation}
\label{eqn:almost-translation}
 w_0(\zeta) = (\zeta - s) + O(|\zeta-s|^{\frac{d}{d+k}})
\end{equation}
We fix such $s$, noting that this constant can be taken to depend only on the
coefficients of the polynomial $\phi(z)$ (and not on the point $p$ under consideration).

The estimate above shows that the boundary of $w_0(\Omega_s)$ is
approximated by the union of two rays $\{ \arg w_0 = \pm
(\tfrac{\pi}{2} + \epsilon) \}$, with the actual boundary being a
displacement of this by $o(|w_0|)$.  In particular $\Re(w_0) \to
-\infty$ linearly along this boundary curve, and the same holds
for a sufficiently small rotation of this region about the origin, e.g.~the
set $e^{i \theta} w_0(\Omega_s)$ for $\theta < \frac{\epsilon}{2}$.
We conclude that $e^{i \theta} w_0(\Omega_s)$ contains a right
half-plane $\{ \Re(w_0) > t \}$ for a constant $t$ depending only on
the coefficients of $\phi(z)$, and for all sufficiently small
$\theta$.

Estimate \eqref{eqn:almost-translation} also shows that $\arg(w_0(p))$
is small for large $p \in \R^+$.  Assuming $p$ is large enough so that
$\arg w_0(p) < \frac{\epsilon}{2}$, and defining $w_1 = e^{-i \arg
  w_0(p)} w_0$ we have $w_1(p) \in \R^+$ and
$w_1(\Omega_{s,\epsilon})$ contains $\{ \Re(w_1) > t \}$.  Finally,
taking $w = i(w_1 - t)$ we get a conformal map onto the upper
half-plane $\H$ taking $p$ to a point on $i \R$.

We claim $(U, w)$ is the desired $|\phi|$-upper-half-plane, where $U$
is the region in the $z$-plane corresponding to $w^{-1}(\H) \subset
\Omega_{s,\epsilon}$.  First, we have constructed this set under the
assumption that $|p|$ is larger than some constant depending on
the coefficients of
$\phi(z)$; since the function $r$ is continuous, we can choose $R_0$ so
that $r(p) > R_0$ implies that $|p|$ is sufficiently large.  We
have $|dw| = |dw_0|$ and thus $|dw|^k = |\phi|$ and this region is
a $|\phi|$-upper-half-plane.  The function $r$ grows linearly on the
boundary of $U$ because the boundary of $w(\Omega_{s,\epsilon})$
approximates (with sublinear error) a ray whose argument differs from
that of the boundary of $U$ by at least $\frac{\epsilon}{2}$; this
ensures a zero-free disk centered at each boundary point of $U$ with
$|\phi|$-radius growing 
linearly with $|w|$, giving the desired
constant $c$.

Finally, we must consider the relation between $\Im(w(p))$ and $r(p)$.
Let $p_0$ denote the point that corresponds to the origin in the
$w$-plane.  The segment on $i \R$ in the $w$-plane from $p$ to $p_0$
is a $|\phi|$-geodesic of length $\Im(w(p))$, hence
$r(p) \leq \Im(w(p)) + r(p_0)$.  But from the definition of the map
$w$ we see that, in the $z$-plane, the point $p_0$ has modulus bounded
in terms of the constants $s$ and $t$ chosen above, which in turn
depend only on the coefficients of $\phi$.  Thus $r(p_0)$ is bounded
by the supremum of $r$ on a fixed closed disk in the $z$-plane, and
taking $C$ to be this supremum we conclude $\Im(w(p)) \geq r(p) - C$.
\end{proof}

\section{ODE asymptotics}
\label{appendix:ode}

In this section we collect some results on asymptotics of solutions to
initial value problems for ODE that are used in Section
\ref{sec:polynomials-to-polygons}.  These techniques and results are
certainly well-known; our goal here is simply to collect precise
statements and corresponding references to standard texts.

We consider the equation
\begin{equation}
\label{eqn:ode}
F'(t) = F(t) A(t)
\end{equation}
on intervals $J \subset \R$, where the coefficient $A : J \to \gl_n\R$
is a continuous function and the solution is a matrix-valued
function $F : J \to \GL_n\R$.  This equation is equivalent
to the statement that $A(t) dt = F(t)^{-1} dF(t)$ is the pullback of
the Maurer-Cartan form on $\GL_n\R$ by the map $F$.

In the small coefficient case ($A$ near zero), one expects the solution to
\eqref{eqn:ode} to be approximately constant.  To quantify this, fix a
norm $\| \param \|$ on the space of $n \times n$ matrices.
Considering bounded and unbounded intervals separately, we have:

\begin{lem}\mbox{}
\label{lem:ode-small-coef}
\begin{rmenumerate}
\item There exist $C, \delta_0 > 0$ such that if $\|A(t)\| < \delta < \delta_0$ for
all $t \in [a,b]$, then the solution $F$ of \eqref{eqn:ode} with $F(a)
= I$ satisfies $|F(t) - I| < C \delta$ for all $t \in [a,b]$.  The
constants $C$ and $\delta_0$ can be taken to depend only on an upper
bound for $|b-a|$.

\item If $\displaystyle \int_a^\infty \| A(t)\| dt < \infty$ then any solution of
\eqref{eqn:ode} on $[a,\infty)$ satisfies $F(t) \to F_0$ as $t
\to \infty$, for some $F_0 \in \GL_n\R$. 
\end{rmenumerate}
\noproof
\end{lem}

While we have stated these results only for the $\GL_n \R$ case, they
are standard facts about linear ODE that can be found, for example, in
\cite{hartman:ode}: Part (i) is an application of \cite[Lemma
IV.4.1]{hartman:ode} to the equation satisfied by $F(t) - I$, while
(ii) follows from \cite[Theorem X.1.1]{hartman:ode}.

Next we consider the case when the coefficient $A(t)$ is not
pointwise bounded, but instead is nearly concentrated in a
1-dimensional subspace of $\gl_n\R$ and has bounded mass.

\begin{lem}
\label{lem:ode-nearly-abelian}
There exist $M, C, \delta_1 > 0$ with the following property: Let
$A(t) = s(t) \cdot X + B(t)$ where $X \in \gl_n \R$ and $s : [a,b] \to
\R$ and $B : [a,b] \to \gl_n \R$ are continuous functions.  If $\int_a^b
|s(t)| \: dt < M$ and $\|B(t)\| < \delta < \delta_1$ for all $t \in
[a,b]$, then the solution of \eqref{eqn:ode} with $F(a) = I$
satisfies
$$\left \| F(t) - \exp \left ( \left
(\int_a^t s(t)dt \right ) \cdot X \right ) \right \| \leq C \delta,$$
for all $t \in [a,b]$.
\end{lem}

\begin{proof}
Define $G(t) = \exp \left ( \left ( \int_a^t s(t) \: dt \right) \cdot
X \right )$.  This function satisfies $G'(t) =
G(t) s(t) X$ and $G(a) = I$.  The integral bound on
$s(t)$ and continuity of the exponential map give a uniform upper bound on
$\| G(t) \|$ in terms of $M$ and $\|X\|$.

Let $H(t) = F(t) G(t)^{-1} $.  Then we have 
$$H'(t) =  H(t) \left( G(t)
B(t) G(t)^{-1} \right ).$$  Choosing $\delta_1$ small enough and using
the uniform bound on $\|G(t)\|$ we can assume that $\| G(t)B(t)
G(t)^{-1}  \| < C' \delta < \delta_0$ for some $C'$, where $\delta_0$ is the
constant from Lemma \ref{lem:ode-small-coef}.  Applying part (1) of
that lemma to the equation above we obtain
$$ \|F(t) G(t)^{-1} - I\| = \|H(t) - I\| < C'' \delta.$$
Since $\|G(t)\|$ is bounded this gives $\| F(t) - G(t)\| < C \delta$
as desired.
\end{proof}

\nocite{}

\vspace{1.5em}

\noindent Department of Mathematics, Statistics, and Computer Science\\
University of Illinois at Chicago\\
\texttt{ddumas@math.uic.edu}\\

\medskip

\noindent Department of Mathematics\\
Rice University\\
\texttt{mwolf@rice.edu}\vspace{-0.2em}

\end{document}